\documentclass[11pt]{article}
\usepackage[tbtags]{amsmath}
\usepackage{amssymb}
\usepackage{amsthm}
\usepackage[misc]{ifsym}
\usepackage{cases}
\usepackage{mathrsfs}
\usepackage{graphicx}
\usepackage{float}
\usepackage{color}
\usepackage{xcolor}
\usepackage{tikz}
\usetikzlibrary{arrows,shapes,chains}

\numberwithin{equation}{section}   
\normalsize
\title{\bf Closed-Loop Solvability of Linear Quadratic Mean-Field Type Stackelberg Stochastic Differential Games
	\thanks{This work is supported by National Key Research \& Development Program of China (2022YFA1006104), National Natural Science Foundations of China (11971266, 11831010), and Shandong Provincial Natural Science Foundations (ZR2022JQ01, ZR2020ZD24, ZR2019ZD42).}}
\author{\normalsize Zixuan Li\thanks{\it School of Mathematics, Shandong University, Jinan 250100, P.R. China, E-mail: 201812064@mail.sdu.edu.cn},\quad Jingtao Shi\thanks{\it Corresponding author, School of Mathematics, Shandong University, Jinan 250100, P.R. China, E-mail: shijingtao@sdu.edu.cn}}
\date{}
\newtheorem{mypro}{Proposition}[section]
\newtheorem{mythm}{Theorem}[section]
\newtheorem{mydef}{Definition}[section]
\newtheorem{mylem}{Lemma}[section]
\newtheorem{Remark}{Remark}[section]
\begin{document}
	\maketitle
	
	\noindent{\bf Abstract:}\quad
	This paper is devoted to a Stackelberg stochastic differential game for a linear mean-field type stochastic differential system with a mean-field type quadratic cost functional in finite horizon. The coefficients in the state equation and weighting matrices in the cost functional are all deterministic. Closed-loop Stackelberg equilibrium strategies are introduced which require to be independent of initial states. Follower's problem is solved firstly, which is a stochastic linear quadratic optimal control problem. By converting the original problem into a new one whose optimal control is known, the closed-loop optimal strategy of the follower is characterized by two coupled Riccati equations as well as a linear mean-field type backward stochastic differential equation. Then the leader turns to solve a stochastic linear quadratic optimal control problem for a mean-field type forward-backward stochastic differential equation. Necessary conditions for the existence of closed-loop optimal strategies for the leader is given by the existence of two coupled Riccati equations with a linear mean-field type backward stochastic differential equation. The solvability of Riccati equations of leader's optimization problem is discussed in the case where the diffusion term of the state equation does not contain the control process of the follower. Moreover, leader's value function is expressed via two backward stochastic differential equations and two Lyapunov equations.
	\vspace{2mm}
	
	
	\vspace{2mm}
	
	
	\section{Introduction}
	
	Stackelberg differential game is also known as leader-follower differential game. The main idea is that the players determine their own strategies to ensure that they can maximize the benefits under the influence of the strategies of other participants. In 1934, Stackelberg game theory was first proposed by German economist H. von Stackelberg to reflect the asymmetric competition between firms in \cite{Stackelberg1934}. In the market, there are always some companies that have more power than other companies or individuals. In order to simulate this asymmetry in the market, a hierarchical structure is defined between decision makers in Stackelberg game model. The player making the decision firstly is called the leader, and he announces his decision at the beginning of the game. Then the remaining participant (called follower) makes decision based on leader's decision. At this time, the differential game of two players is transformed into the optimal control problem of the follower. Simply put, the follower immediately reacts to the actions of the leader, adjusts his decision, and makes his own cost functional reach the optimal. After the follower makes a decision, the leader then adjusts his strategy according to the follower's decision, so as to maximize his interests. Now, it becomes another optimal control problem. The hierarchical feature fits well with some problems in financial and economic markets, so Stackelberg game theory has great theoretical and application appeal to scholars.
	
	Simann and Cruz \cite{SC1973-1,SC1973-2} studied the multi-stages and dynamic {\it linear quadratic} (LQ for short) Stackelberg differential games, where feedback Stackelberg equilibria are expressed by matrix-valued Riccati differential equations in Hilbert space. Castanon and Athans \cite{CA1976} considered feedback Stackelberg strategies for the two-person linear multi-stages game with quadratic performance criteria and noisy measurements, and gave an explicit solution when information sets are nested in a stochastic case. Bagchi and Ba\c{s}ar \cite{BB1981} investigated the LQ Stackelberg stochastic differential game, where the diffusion term of the state equation did not contain the state and control variables. Yong \cite{Yong2002} extended the LQ Stackelberg stochastic differential game to random and state-control dependent coefficients and the weight matrices for the controls in the cost functionals are not necessarily positive definite, and obtained the feedback representation of the open-loop equilibrium via some stochastic Riccati equations. In the past decades, there have been a great deal of works on this issue, for jump diffusions see \O ksendal et al. \cite{OSU2013}, Moon \cite{Moon2021}, for different information structures see Ba\c{s}ar and Olsder \cite{BO1998}, Bensoussan et al. \cite{BCS2015}, for multiple followers see Mukaidani and Xu \cite{MX2015}, for time-delayed systems see Xu and Zhang \cite{XZ2016}, Xu et al. \cite{XSZ2018}, for partial/asymmetric/overlapping information see Shi et al. \cite{SWX2016,SWX2017,SWX2020}, for large populations see Huang et al. \cite{HSW2020}, Moon and Ba\c{s}ar \cite{MB2018}, for backward stochastic systems see Zheng and Shi \cite{ZS2020}, for the multilevel hierarchy see Li and Yu \cite{LiYu2018}, Li et al. \cite{LiXiongYu2021}, for the model uncertain system see Huang et al. \cite{HWW2022}.
	
	For mean-field type model, Lin et al. \cite{LJZ2019} concerned the LQ Stackelberg game of the mean-field type stochastic systems in finite horizon where the open-loop Stackelberg equilibrium was represented as the feedback form involving the state as well as its mean. Du and Wu \cite{DW2019} was devoted to a kind of Stackelberg differential game of {\it mean-field type backward stochastic differential equations} (MF-BSDEs for short). Wang and Zhang \cite{WZ2020}, Wang et al. \cite{WWZ2020} considered mean-field type LQ Stackelberg stochastic differential game with one leader and two followers under symmetric and asymmetric information, respectively. Moon and Yang \cite{MY2021} considered the LQ time-inconsistent mean-field type Stackelberg stochastic differential game. Moon \cite{Moon2022} studied LQ Stackelberg differential games for mean-field type stochastic systems with jump diffusions.
	
	The open-loop optimal control depends on the initial state of the system, that is to say, it is necessary to know the initial state of the system when solving the open-loop optimal control. But this is not practical in engineering systems and real life, so we are interested in the closed-loop optimal solution which is independent of the initial state. For this consideration, the closed-loop solvability of mean-field type LQ Stackelberg stochastic differential game is considered in this paper. Concepts of open-loop and closed-loop solvabilities of {\it stochastic LQ} (SLQ for short) optimal control problems were introduced in Sun and Yong \cite{SY2014} in 2014, which studied a LQ zero-sum stochastic differential game and described the existence of open-loop and closed-loop saddle points respectively. The SLQ optimal control problem was introduced as a special case with only one player/controller at the end of this paper. More detailed characterizations of the closed-loop solvability for the SLQ optimal control problem were further given in Sun et al. \cite{SLY2016}, where they studied the relationship between convexity, uniform convexity of cost functional and open-loop, closed-loop solvabilities. Li et al. \cite{LSY2016} was devoted to the closed-loop solvability of the mean-field type SLQ optimal control problem. They obtained Riccati equations by transforming the original problem into a deterministic problem and using the matrix maximum principle. Using the same method, Tang et al. \cite{TangLiHuang2020} studied the open-loop and closed-loop solvabilities for the indefinite mean-field type SLQ optimal control with Poisson jumps. Sun and Yong \cite{SY2019} studied the SLQ nonzero-sum differential game and gave sufficient and necessary conditions for the existence of open-loop and closed-loop Nash equilibria. For more details, please refer to the books \cite{SunYong2020-1,SunYong2020-2}. Lv \cite{Lv2019,Lv2020} studied the closed-loop solvability of the SLQ optimal control problem for systems governed by stochastic evolution equations and the expression of the closed-loop optimal strategy was given. Lyapunov equations were introduced directly and Riccati equations were got from these. Li and Shi \cite{LiShi2022} studied the closed-loop solvability of SLQ optimal control problems with Poisson jumps.
	
	Recently, Li and Shi \cite{LiShi2021} researched the closed-loop solvability of LQ Stackelberg stochastic differential game. In contrast to our previous paper \cite{LiShi2021}, we consider a SLQ system with mean-field terms in this paper, where the coefficients in the state equation and the weighting matrices in the cost functionals are all deterministic, and the weighting matrices in the cost functional are allowed to be indefinite. This paper is never a trivial extension. The main contribution of this paper can be summarized as follows.
	\begin{itemize}
		\item [(\romannumeral 1)] The SLQ system considered in this paper is an inhomogeneous system, i.e. the state equation includes inhomogeneous terms and cost functionals contain linear terms of state process and control processes. The existence of these terms greatly increases the computational difficulties of closed-loop optimal strategies and value functions. Especially for leader's problem, as shown in Section 4, the existence of these terms makes the optimization problem of the leader become a stochastic optimal control problem of a inhomogeneous {\it mean-field type forward-backward SDE} (MF-FBSDE for short), which makes a MF-BSDE is necessary to characterize leader's closed-loop optimal strategy. In addition, the appearance of these terms also leads to the failure of the completion-of-square technique, and we need to find a new way to calculate the value function of leader.
		\item [(\romannumeral 2)] Sufficient and necessary conditions for the closed-loop solvability of the follower's optimization problem are given. Different from Li et al. \cite{LSY2016}, when solving the optimization problem of the  follower, a mean-field type SLQ optimal control problem, we transform the closed-loop solvability problem of the follower into a new mean-field type SLQ optimal control problem for which the open-loop optimal control is known. We can get two Lyapunov equations from there and then we obtain the closed-loop optimal strategy of the follower which is characterized by two Riccati equations and a linear MF-BSDE.
		\item [(\romannumeral 3)]  The definitions of leader's closed-loop optimal strategy and closed-loop Stackelberg equilibrium are given. The expression for leader's closed-loop optimal strategy which is characterized by two Riccati equations and a linear MF-BSDE is derived. Since leader's problem is a stochastic optimal control problem of a MF-FBSDE, the appearance of the MF-BSDE in leader's state makes the method in Sun et al. \cite{SLY2016} is invalid in when constructing Riccati equations, where they defined $P(\cdot)=\mathbb{Y}\mathbb{X}^{-1}$ in which $\mathbb{Y(\cdot)}$ was a matrix-valued BSDE and $\mathbb{X(\cdot)}$ was a matrix-valued SDE. However, in our leader's problem, some product terms of $\mathbb{X}(\cdot)$ and $\mathbb{Z}(\cdot)$ in the shape of $\big(\mathbb{Z}-\mathbb{E}\mathbb{Z}\big)\mathbb{X}^{-1}$ or $\mathbb{E}\mathbb{Z}\cdot\mathbb{E}\mathbb{X}^{-1}$ appear in the expression of $\mathbb{Y}\mathbb{X}^{-1}$. We have no way to deal with such terms, where $\mathbb{Z}(\cdot)$ is the second variable of matrix-valued MF-BSDE $\big(\mathbb{Y}(\cdot),\mathbb{Z}(\cdot)\big)$. Similarly, the method in \cite{LSY2016} also fails. Inspired by \cite{LSY2016,Lv2019,Lv2020}, we overcome this difficulty by transforming the original problem into a new stochastic optimal control problem whose open-loop optimal control is $\bar{v}(\cdot)=0$.
		\item [(\romannumeral 4)] The general solvability of leader's Riccati equations $(P_2(\cdot), \Pi_2(\cdot))$ is very difficult due to the appearance of terms like $[\tilde{R}+\tilde{\mathcal{D}}^\top(I-P_2\tilde{\mathcal{K}})^{-1}P_2\tilde{\mathcal{D}}]^{-1}$ and $[\check{R}+\check{\mathcal{D}}^\top(I-P_2\check{\mathcal{K}})^{-1}P_2\check{\mathcal{D}}]^{-1}$. The solvabilities of Riccati equations of leader's optimization problem is studied under some coefficient assumptions in a special case where the diffusion term of the state equation does not include the control process of the follower. We use iterative method to prove the existence of a solution of $P_2(\cdot)$, then $\Pi_2(\cdot)$ is reduced to the classical Riccati equation whose solvability can be proved under some conditions in this case.
		\item [(\romannumeral 5)] Using the equivalent cost functional method, we introduce two coupled Lyapunov equations and two coupled BSDEs, and use them to get the value function of the leader. This method was proposed by Yu \cite{Yu2013} to deal with the SLQ optimal control problem, that is, using a function $f(\cdot)$ to transform the original cost functional into a more ideal form. This method also appeared in Huang et al. \cite{HWW2022}.
	\end{itemize}
	
	The rest of this paper is organized as follows. Section 2 gives some preliminaries, to introduce closed-loop Stackelberg equilibrium for the mean-field type LQ Stackelberg stochastic differential game. Section 3 is devoted to solve the optimization problem of the follower. With the aid of two Riccati equations, sufficient and necessary conditions of the closed-loop solvability for follower's problem are given. In Section 4, necessary conditions for the closed-loop solvability for leader's problem is obtained. Finally, some concluding remarks are given in Section 5.
	
	\section{Problem formulation and preliminaries}\label{Preliminaries}
	
	In this section, we present the problem and some notations used throughout this paper. Let $T>0$ be a constant, and $[0,T]$ is a finite time duration. Let $\mathbb{R}^n$ be a $n$-dimensional real Euclidean space, $\mathbb{R}^{n \times m}$ be the set of all $(n \times m)$ real matrices, and $\mathbb{S}^n$ be the set of all $(n \times n)$ symmetric matrices, $\mathbb{S}^n_+$ be the set of all $(n \times n)$ symmetric matrices which are semi-positive. We use $M^\top$ and $\mathcal{R}(M)$ to denote the transpose and the range of the matrix M, respectively. And we let $\langle \cdot,\cdot \rangle $ denote inner products on a Hilbert space given by $\langle M,N \rangle \to tr(M^\top N)$ where $tr(M)$ is the trace of the square matrix $M$, and $|\cdot|$ denotes the Frobenius norm induced by $|M|=\sqrt{tr(MM^\top)}$. For $M,N\in \mathbb{S}^n$, $M\geq N$ (respectively, $M>N$) implies that $M-N$ is a positive semi-definite matrix (respectively, positive definite matrix). Let $M^\dagger$ denote the pseudo-inverse of the matrix $M \in \mathbb{R}^{n \times m}$. If the inverse $M^{-1}$ of $M \in \mathbb{R}^{n \times n}$ exists, then the pseudo-inverse is equal to the inverse.
	
	For any Banach space $H$ (for example, $H=\mathbb{R}^n, \mathbb{R}^{n \times m}, \mathbb{S}^n, \mathbb{S}^n_+$) and $t \in [0,T)$, let $L^p(t,T;H)$ $(1 \leq p < \infty)$  be the space of all Lebesgue measurable functions $\varphi: [t,T] \to H$ such that $\int_t^T|\varphi(s)|^pds<\infty$, $L^\infty(t,T;H)$ be the set of all essentially bounded measurable functions $\varphi: [t,T] \to H$, and $C([t,T];H)$ be the space of all continuous functions $\varphi: [t,T] \to H$.
	
	Let $(\Omega,\mathcal{F},\mathbb{F},\mathbb{P})$ be a completed filtered probability space on which a standard one-dimensional Brownian motion $W=\{W(t);0 \leq t < \infty \}$ is defined, where $\mathbb{F}=\{\mathcal{F}_t\}_{t \leq 0}$ is the natural filtration of $W$ augmented by all $\mathbb{P}$-null sets in $\mathcal{F}$. For $t\in(0,T]$, we denote by $L^2_{\mathcal{F}_t}(\Omega;H)$ the space of all $\mathcal{F}_t$-measurable random variables $\xi$ such that $\mathbb{E}|\xi|^2<\infty$, by $L^2_{\mathbb{F}}(t,T;H)$ the space of all $\mathbb{F}$-progressively measurable processes $f(\cdot)$ such that $\mathbb{E}\int_t^T|f(s)|^2ds<\infty$, and by $L^2_{\mathbb{F}}(\Omega;L^1(t,T;H))$ the space of all $\mathbb{F}$-progressively measurable processes $f(\cdot)$ such that
	$\mathbb{E}\big(\int_t^T|f(s)|ds\big)^2<\infty$.
	
	For any initial pair $(t,\xi)\in [0,T]\times L^2_{\mathcal{F}_t}(\Omega;\mathbb{R}^n) $, we consider the following controlled linear {\it mean-field type stochastic differential equation} (MF-SDE for short):
	\begin{equation}\label{state}
		\left\{
		\begin{aligned}
			dx(s)&=\big[A(s)x(s)+B_1(s)u_1(s)+B_2(s)u_2(s)+b(s)\\
			&\qquad +\hat{A}(s)\mathbb{E}x(s)+\hat{B}_1(s)\mathbb{E}u_1(s)+\hat{B}_2(s)\mathbb{E}u_2(s)\big]ds\\
			&\quad+\big[C(s)x(s)+D_1(s)u_1(s)+D_2(s)u_2(s)+\sigma(s)\\
			&\qquad +\hat{C}(s)\mathbb{E}x(s)+\hat{D}_1(s)\mathbb{E}u_1(s)+\hat{D}_2(s)\mathbb{E}u_2(s)\big]dW(s),\quad s \in [t,T],\\
			x(t)&=\xi,
		\end{aligned}
		\right.
	\end{equation}
	where $A(\cdot),\hat{A}(\cdot),B_i(\cdot),\hat{B}_i(\cdot),C(\cdot),\hat{C}(\cdot),D_i(\cdot),\hat{D}_i(\cdot),i=1,2$ are given deterministic matrix-valued functions of proper dimensions, and $b(\cdot),\sigma(\cdot)$ are $\mathbb{F}$-progressively measurable processes. In the above, $x(\cdot)$ is the state process with values in $\mathbb{R}^n$, and $u_1(\cdot),u_2(\cdot)$ are control processes taken by two players in the game, labeled 1 and 2, with values in $\mathbb{R}^{m_1}$ and $\mathbb{R}^{m_2}$, respectively.
	
	We define the admissible control set: For $i=1,2$,
	\begin{equation}\label{cost functional}
		\mathcal{U}_i[t,T]:=\bigg\{u_i:[t,T] \times \Omega \to \mathbb{R}^{m_i} \big|\,\mathbb{F}\mbox{-progressively measurable},\,\mathbb{E}\int_t^T |u_i(s)|^2 ds < \infty \bigg\}.
	\end{equation}
	The control processes $u_1(\cdot)\in\mathcal{U}_1[t,T]$ and $u_2(\cdot)\in\mathcal{U}_2[t,T]$ are called admissible controls. We consider cost functionals for two players as follows. For $i=1,2$,
	\begin{equation}\label{cf}
		\begin{aligned}
			&J_i(t,\xi;u_1(\cdot),u_2(\cdot))=\mathbb{E} \Bigg\{\big\langle G^ix(T),x(T)\big\rangle+\big\langle \hat{G}^i\mathbb{E}x(T),\mathbb{E}x(T)\big\rangle
			+2\big\langle g^i ,x(T)\big\rangle\\
			&\quad +2\big\langle \hat{g}^i ,\mathbb{E}x(T)\big\rangle +\int_t^T \bigg[\bigg\langle
			\left( \begin{array}{ccc}
				Q^i(s)   & S^i_1(s)^\top  & S^i_2(s)^\top \\
				S^i_1(s) & R^i_{11}(s)    & R^i_{21}(s)   \\
				S^i_2(s) & R^i_{12}(s)    & R^i_{22}(s)   \\
			\end{array} \right)
			\left( \begin{array}{c} x(s) \\ u_1(s) \\ u_2(s) \end{array} \right),
			\left( \begin{array}{c} x(s) \\ u_1(s) \\ u_2(s) \end{array}\right)\bigg\rangle\\
			&\quad +2\bigg\langle
			\left( \begin{array}{c} q^i(s) \\ \rho^i_1(s) \\ \rho^i_2(s) \end{array} \right),
			\left( \begin{array}{c} x(s) \\ u_1(s) \\ u_2(s) \end{array}\right)\bigg\rangle\\
			&\quad +\bigg\langle
			\left( \begin{array}{ccc}
				\hat{Q}^i(s)   & \hat{S}^i_1(s)^\top  & \hat{S}^i_2(s)^\top \\
				\hat{S}^i_1(s) & \hat{R}^i_{11}(s)    & \hat{R}^i_{21}(s)   \\
				\hat{S}^i_2(s) & \hat{R}^i_{12}(s)    & \hat{R}^i_{22}(s)   \\
			\end{array} \right)
			\left( \begin{array}{c} \mathbb{E}x(s) \\ \mathbb{E}u_1(s) \\ \mathbb{E}u_2(s) \end{array} \right),
			\left( \begin{array}{c} \mathbb{E}x(s) \\ \mathbb{E}u_1(s) \\ \mathbb{E}u_2(s) \end{array}\right)\bigg\rangle \bigg]ds \Bigg\},
		\end{aligned}
	\end{equation}
	where $G^i,\hat{G}^i$ are symmetric matrices, $Q^i(\cdot),\hat{Q}^i(\cdot),S^i_1(\cdot),\hat{S}^i_1(\cdot),S^i_2(\cdot),\hat{S}^i_2(\cdot),R^i_{11}(\cdot),\hat{R}^i_{11}(\cdot),R^i_{12}(\cdot)$,
	$\hat{R}^i_{12}(\cdot),R^i_{21}(\cdot),\hat{R}^i_{21}(\cdot),R^i_{22}(\cdot),\hat{R}^i_{22}(\cdot)$ are deterministic matrix-valued functions with
	\begin{equation}\nonumber
		\begin{aligned}
			&Q^i(\cdot)^\top =Q^i(\cdot),\,\quad\,R^i_{jj}(\cdot)^\top =R^i_{jj}(\cdot),\,\quad\,R^i_{12}(\cdot)^\top =R^i_{21}(\cdot),\\
			&\hat{Q}^i(\cdot)^\top =\hat{Q}^i(\cdot),\,\quad\,\hat{R}^i_{jj}(\cdot)^\top =\hat{R}^i_{jj}(\cdot),\,\quad\,\hat{R}^i_{12}(\cdot)^\top =\hat{R}^i_{21}(\cdot),\qquad\,\, i,j=1,2,
		\end{aligned}
	\end{equation}
	$g^i$ are $\mathcal{F}_T$-measurable random vectors, $\hat{g}^i$ are deterministic vectors, and $q^i(\cdot),\rho^i_1(\cdot),\rho^i_2(\cdot)$ are $\mathbb{F}$-progressively measurable processes.
	
	Now, we introduce the following assumptions that will be in force throughout this paper.
	\begin{itemize}
		\item [\textbf{(H1)}] Coefficients of the state equation (\ref{state}) satisfy the following:
		\begin{equation*}
			\begin{cases}
				A(\cdot), \hat{A}(\cdot) \in L^1(0,T;\mathbb{R}^{n \times n}),\quad B_i(\cdot),\hat{B}_i(\cdot) \in L^2(0,T;\mathbb{R}^{n \times m_i}),\\
				b(\cdot) \in L^2_\mathbb{F}(\Omega;L^1(0,T;\mathbb{R}^n)),\quad C(\cdot), \hat{C}(\cdot) \in L^2(0,T;\mathbb{R}^{n \times n}),\\
				D_i(\cdot),\hat{D}_i(\cdot) \in L^\infty(0,T;\mathbb{R}^{n \times m_i}),\quad \sigma(\cdot) \in L^2_\mathbb{F}(0,T;\mathbb{R}^n),\quad i=1,2.
			\end{cases}
		\end{equation*}
		\item [\textbf{(H2)}] Weighting coefficients of cost functionals (\ref{cf}) satisfy the following:
		\begin{equation*}
			\begin{cases}
				Q^i(\cdot), \hat{Q}^i(\cdot) \in L^1(0,T;\mathbb{S}^n),\quad S^i_j(\cdot), \hat{S}^i_j(\cdot) \in L^2(0,T;\mathbb{R}^{m_j \times n}),\\
				R^i_{jj}(\cdot), \hat{R}^i_{jj}(\cdot) \in L^{\infty}(0,T;\mathbb{S}^{m_j}),\quad R^i_{12}(\cdot), \hat{R}^i_{12}(\cdot) \in L^{\infty}(0,T;\mathbb{R}^{m_2 \times m_1}),\\
				q^i(\cdot) \in L^2_{\mathbb{F}}(\Omega;L^1(0,T;\mathbb{R}^{n})),\quad \rho^i_j(\cdot) \in L^2_{\mathbb{F}}(0,T;\mathbb{R}^{m_j}),\\
				G^i, \hat{G}^i \in \mathbb{S}^n,\quad g^i \in L^2_{\mathcal{F}_T}(\Omega;\mathbb{R}^{n}),\quad \hat{g}^i \in \mathbb{R}^n, \quad i, j=1, 2.
			\end{cases}
		\end{equation*}
	\end{itemize}
	
	Under (H1), (H2), for given $\xi \in L^2_{\mathcal{F}_t}(\Omega;\mathbb{R}^n)$, $u_i(\cdot) \in \mathcal{U}_i[t,T],\,\,i=1,\,2$, (\ref{state}) admits a unique adapted solution $x(\cdot) \in L^2_{\mathbb{F}}(t,T;\mathbb{R}^n)$ (see, for example, Sun \cite{Sun2017}) and (\ref{cf}) are well-defined.
	
	Now let us explain the mean-field type Stackelberg stochastic differential game framework. In this game, Player 1 is the follower, and Player 2 is the leader. At the beginning of Stackelberg game, for any given initial state, the leader announces his control $u_2(\cdot)\in\mathcal{U}_2[t,T]$ over the whole time horizon $[t,T]$. The follower determines his optimal strategy $\bar{u}_1(\cdot)\in\mathcal{U}_1[t,T]$ over $[t,T]$ to minimize $J_1(t,\xi;u_1(\cdot),u_2(\cdot))$ after he know the decision of the leader. Then, considering follower's optimal strategy, the leader would like to find an optimal control $\bar{u}_2(\cdot)\in\mathcal{U}_2[t,T]$ such that $J_2(t,\xi;\bar{u}_1(\cdot),\bar{u}_2(\cdot))$ is the minimum of $J_2(t,\xi;\bar{u}_1(\cdot),u_2(\cdot))$ over $u_2(\cdot)\in\mathcal{U}_2[t,T]$. In a more rigorous way, we give the following definition of the Stackelberg game.
	
	\begin{mydef}\label{def2.1}
		For any given initial pair $(t,\xi) \in [0,T] \times L^2_{\mathcal{F}_t}(\Omega;\mathbb{R}^n)$. A pair $(\bar{u}_1(\cdot),\bar{u}_2(\cdot)) \in \mathcal{U}_1[t,T] \times \mathcal{U}_2[t,T]$ is called a (unique) open-loop Stackelberg equilibrium point to the above Stackelberg game if it satisfies the following conditions:
		\begin{itemize}
			\item [(\romannumeral 1)] For any  $u_2(\cdot)\in\mathcal{U}_2[t,T]$, there exists a (unique) map $\bar{\mu}_1:\mathcal{U}_2[t,T] \times L^2_{\mathcal{F}_t}(\Omega;\mathbb{R}^n) \to \mathcal{U}_1[t,T]$ such that
			\begin{equation*}
				J_1(t, \xi;\bar{\mu}_1[u_2(\cdot),\xi](\cdot),u_2(\cdot))=\mathop{\min}\limits_{u_1(\cdot) \in\, \mathcal{U}_1[t,T]}J_1(t,\xi;u_1(\cdot),u_2(\cdot)).
			\end{equation*}
			\item [(\romannumeral 2)] There exists a (unique) $\bar{u}_2(\cdot)\in\mathcal{U}_2[t,T]$ such that
			\begin{equation*}
				J_2(t,\xi;\bar{\mu}_1[\bar{u}_2(\cdot),\xi](\cdot),\bar{u}_2(\cdot))=\mathop{\min}\limits_{u_2(\cdot) \in\, \mathcal{U}_2[t,T]}J_2(t,\xi;\bar{\mu}_1[u_2(\cdot),\xi](\cdot),u_2(\cdot)).
			\end{equation*}
			\item [(\romannumeral 3)] The open-loop optimal strategy of the follower is $\bar{u}_1(\cdot)=\bar{\mu}_1[\bar{u}_2(\cdot),\xi](\cdot)$.
		\end{itemize}
	\end{mydef}
	
	We now consider the closed-loop Stackelberg equilibrium of our mean-field type LQ Stackelberg stochastic differential game. To this end, we denote
	\begin{equation*}
		\mathcal{Q}_i[t,T]\equiv L^2(t,T;\mathbb{R}^{m_i\times n}),\qquad i=1,2.
	\end{equation*}
	For convenience, we set in this paper
	\begin{equation*}
		\boldsymbol{W}\equiv W+\hat{W}, \quad \text{for } W=A,\, B_i,\, C,\, D_i,\, Q^i,\, S^i_j,\, R^i_{jj},\, R^i_{12},\, G^i,\,g^i, \quad i,j=1,2.
	\end{equation*}
	
	Firstly, for any initial pair $(t,\xi) \in [0,T] \times L^2_{\mathcal{F}_t}(\Omega;\mathbb{R}^n)$ and any given $u_2(\cdot) \in \mathcal{U}_2[t,T]$, taking $\big(\Theta_1(\cdot) ,\hat{\Theta}_1(\cdot),v_1(\cdot)\big)\in \mathcal{Q}_1[t,T] \times \mathcal{Q}_1[t,T] \times \mathcal{U}_1[t,T]$, let us consider the following MF-SDE: (Some time variables are usually omitted.)
	\begin{equation}\label{follower closed-loop state}
		\left\{
		\begin{aligned}
			dx^{u_2}(s)&=\big\{(A+B_1\Theta_1)(x^{u_2}-\mathbb{E}x^{u_2})+B_1(v_1-\mathbb{E}v_1)+B_2u_2+b\\
			&\qquad +(\boldsymbol{A}+\boldsymbol{B}_1\hat{\Theta}_1)\mathbb{E}x^{u_2}+\boldsymbol{B}_1\mathbb{E}v_1+\hat{B}_2\mathbb{E}u_2\big\}ds\\
			&\quad+\big\{(C+D_1\Theta_1)(x^{u_2}-\mathbb{E}x^{u_2})+D_1(v_1-\mathbb{E}v_1)+D_2u_2+\sigma\\
			&\qquad +(\boldsymbol{C}+\boldsymbol{D}_1\hat{\Theta}_1)\mathbb{E}x^{u_2}+\boldsymbol{D}_1\mathbb{E}v_1+\hat{D}_2\mathbb{E}u_2\big\}dW(s),\quad s\in[t,T],\\
			x^{u_2}(t)&=\xi,
		\end{aligned}
		\right.
	\end{equation}
	which admits a unique solution $x^{u_2}(\cdot) \equiv x(\cdot;t,\xi,\Theta_1(\cdot),\hat{\Theta}_1(\cdot),v_1(\cdot),u_2(\cdot))$, depending on $\Theta_1(\cdot)$, $\hat{\Theta}_1(\cdot)$ and $v_1(\cdot)$. The above MF-SDE is called the follower's closed-loop system of the original state equation (\ref{state}) under the follower's closed-loop strategy $(\Theta_1(\cdot),\hat{\Theta}_1(\cdot),v_1(\cdot))$. We point out that $(\Theta_1(\cdot),\hat{\Theta}_1(\cdot),v_1(\cdot)))$ is independent of the initial state $\xi$. With the above $x^{u_2}(\cdot)$, we define
	\begin{equation}\label{follower closed-loop cost}
		\begin{aligned}
			&J_1\big(t,\xi;\Theta_1\big(x^{u_2}-\mathbb{E}x^{u_2}\big)+\hat{\Theta}_1\mathbb{E}x^{u_2}+v_1,u_2(\cdot)\big)\\
			&\hspace{-3mm} =\mathbb{E}\bigg\{\big\langle G^1\big(x^{u_2}(T)-\mathbb{E}x^{u_2}(T)\big),x^{u_2}(T)-\mathbb{E}x^{u_2}(T)\big\rangle
			+2\big\langle g^1,x^{u_2}(T)-\mathbb{E}x^{u_2}(T)\big\rangle\\
			&\qquad +\big\langle \boldsymbol{G}^1\mathbb{E}x^{u_2}(T),\mathbb{E}x^{u_2}(T)\big\rangle +2\big\langle \boldsymbol{g}^1,\mathbb{E}x^{u_2}(T)\big\rangle\\
			&\qquad +\int_t^T \Big[\big\langle(Q^1+\Theta_1^\top S^1_1+S^{1\top}_1\Theta_1+\Theta_1^\top R^1_{11}\Theta_1)(x^{u_2}-\mathbb{E}x^{u_2}),x^{u_2}-\mathbb{E}x^{u_2}\big\rangle\\
			&\qquad\qquad +2\big\langle(S^1_1+R^1_{11}\Theta_1)(x^{u_2}-\mathbb{E}x^{u_2}),v_1-\mathbb{E}v_1\big\rangle
			+\big\langle R^1_{11}(v_1-\mathbb{E}v_1),v_1-\mathbb{E}v_1\big\rangle\\
			&\qquad\qquad +2\big\langle(S^1_2+R^1_{12}\Theta_1)^\top u_2+q^1+\Theta_1^\top \rho^1_1,x^{u_2}-\mathbb{E}x^{u_2}\big\rangle
			+2\big\langle R^1_{21}u_2+\rho^1_1,v_1-\mathbb{E}v_1\big\rangle\\
			&\qquad\qquad +\big\langle R^1_{22}u_2,u_2 \big\rangle+2\big\langle \rho^1_2,u_2 \big\rangle  \Big]ds\\
			&\qquad +\int_t^T \Big[\big\langle(\boldsymbol{Q}^1+\hat{\Theta}_1^\top \boldsymbol{S}^1_1+{\boldsymbol{S}^1_1}^\top\hat{\Theta}_1
			+\hat{\Theta}_1^\top\boldsymbol{R}^1_{11}\hat{\Theta}_1)\mathbb{E}x^{u_2},\mathbb{E}x^{u_2}\big\rangle\\
			&\qquad\qquad +2\big\langle(\boldsymbol{S}^1_1+\boldsymbol{R}^1_{11}\hat{\Theta}_1)\mathbb{E}x^{u_2},\mathbb{E}v_1\big\rangle
			+\big\langle\boldsymbol{R}^1_{11}\mathbb{E}v_1,\mathbb{E}v_1\big\rangle
			+2\big\langle(\boldsymbol{S}^1_2+\boldsymbol{R}^1_{12}\hat{\Theta}_1)^\top \mathbb{E}u_2\\
			&\qquad\qquad +\mathbb{E}q^1+\hat{\Theta}_1^\top \mathbb{E}\rho^1_1,\mathbb{E}x^{u_2}\big\rangle+2\big\langle \boldsymbol{R}^1_{21}u_2+\rho^1_1,\mathbb{E}v_1\big\rangle
			+\big\langle \hat{R}^1_{22}\mathbb{E}u_2,\mathbb{E}u_2 \big\rangle \Big]ds \bigg\}.
		\end{aligned}
	\end{equation}
	
	Then we introduce the following definition.
	
	\begin{mydef}
		For $u_2(\cdot)\in\mathcal{U}_2[t,T]$, a triple $(\bar{\Theta}_1(\cdot),\bar{\hat{\Theta}}_1(\cdot),\bar{v}_1(\cdot)) \in \mathcal{Q}_1[t,T] \times \mathcal{Q}_1[t,T] \times \mathcal{U}_1[t,T]$ is called (unique) follower's closed-loop optimal strategy on $[t,T]$, if the follower could find a (unique) pair  $(\bar{\Theta}_1(\cdot),\bar{\hat{\Theta}}_1(\cdot)) \in \mathcal{Q}_1[t,T] \times \mathcal{Q}_1[t,T]$ and a  (unique) map $\bar{v}_1:\mathcal{U}_2[t,T]\rightarrow\mathcal{U}_1[t,T]$ such that
		\begin{equation}
			\begin{aligned}
				&J_1\big(t,\xi;\bar{\Theta}_1\big(\bar{x}^{u_2}-\mathbb{E}\bar{x}^{u_2}\big)+\bar{\hat{\Theta}}_1\mathbb{E}\bar{x}^{u_2}+\bar{v}_1[u_2],u_2(\cdot)\big)
				\leq J_1\big(t,\xi;\Theta_1\big(x^{u_2}-\mathbb{E}x^{u_2}\big)+\hat{\Theta}_1\mathbb{E}x^{u_2}+v_1,u_2(\cdot)\big),\\
				&\quad \forall\, (\Theta_1(\cdot),\hat{\Theta}_1(\cdot),v_1(\cdot)) \in \mathcal{Q}_1[t,T] \times \mathcal{Q}_1[t,T] \times \mathcal{U}_1[t,T],\quad
				\forall\, (t,\xi) \in [0,T] \times L^2_{\mathcal{F}_t}(\Omega;\mathbb{R}^n),
			\end{aligned}
		\end{equation}
		where $\bar{x}^{u_2}(\cdot)\equiv x(\cdot;t,\xi,\bar{\Theta}_1(\cdot),\bar{\hat{\Theta}}_1(\cdot),\bar{v}_1[u_2](\cdot),u_2(\cdot))$, $x^{u_2}(\cdot)\equiv x(\cdot;t,\xi,\Theta_1(\cdot),\hat{\Theta}_1(\cdot),v_1(\cdot),u_2(\cdot))$. In this case, follower's problem is said to be (uniquely) closed-loop solvable under $u_2(\cdot)\in \mathcal{U}_2[t,T]$.
	\end{mydef}
	
	After knowing follower's closed-loop optimal strategy $(\bar{\Theta}_1(\cdot),\bar{\hat{\Theta}}_1(\cdot),\bar{v}_1[\cdot](\cdot))$, the leader seeks his closed-loop optimal strategy $(\bar{\Theta}_2(\cdot),\bar{\hat{\Theta}}_2(\cdot),\bar{v}_2(\cdot))$. By substituting $u_2$ with $\Theta_2(\bar{x}^{\Theta_2,\hat{\Theta}_2,v_2}-\mathbb{E}\bar{x}^{\Theta_2,\hat{\Theta}_2,v_2})+\hat{\Theta}_2\mathbb{E}\bar{x}^{\Theta_2,\hat{\Theta}_2,v_2}+v_2$
	in (\ref{follower closed-loop state}), we can define leader's closed-loop system under leader's closed-loop strategy $(\Theta_2(\cdot),\hat{\Theta}_2(\cdot),v_2(\cdot))$, which we have denoted
	$$
	\bar{x}^{\Theta_2,\hat{\Theta}_2,v_2}(\cdot)\equiv x(\cdot;t,\xi,\bar{\Theta}_1(\cdot),\bar{\hat{\Theta}}_1(\cdot),\bar{v}_1[\Theta_2,\hat{\Theta}_2,v_2](\cdot),\Theta_2(\cdot),\hat{\Theta}_2(\cdot),v_2(\cdot))
	$$
	and similarly we can define
	$$
	J_2\big(t,\xi;\bar{\Theta}_1,\bar{\hat{\Theta}}_1,\bar{v}_1[\Theta_2,\hat{\Theta}_2,v_2],\Theta_2\big(\bar{x}^{\Theta_2,\hat{\Theta}_2,v_2}
	-\mathbb{E}\bar{x}^{\Theta_2,\hat{\Theta}_2,v_2}\big)+\hat{\Theta}_2\mathbb{E}\bar{x}^{\Theta_2,\hat{\Theta}_2,v_2}+v_2\big).
	$$
	We note the dependence of $\bar{v}_1(\cdot)$ on $(\Theta_2(\cdot),\hat{\Theta}_2(\cdot),v_2(\cdot))$. The formal definition of leader's closed-loop optimal strategy is given below, and a more precise version will be given in Section \ref{leader's problem}.
	
	\begin{mydef}
		A triple $(\bar{\Theta}_2(\cdot),\bar{\hat{\Theta}}_2(\cdot),\bar{v}_2(\cdot)) \in \mathcal{Q}_2[t,T] \times \mathcal{Q}_2[t,T] \times \mathcal{U}_2[t,T] $ is called (unique) leader's closed-loop optimal strategy on $[t,T]$, if
		\begin{equation}
			\begin{aligned}
				&J_2\big(t,\xi;\bar{\Theta}_1,\bar{\hat{\Theta}}_1,\bar{v}_1[\bar{\Theta}_2,\bar{\hat{\Theta}}_2,\bar{v}_2],\bar{\Theta}_2\big(\bar{x}
				-\mathbb{E}\bar{x}\big)+\bar{\hat{\Theta}}_2\mathbb{E}\bar{x}+\bar{v}_2\big)\\
				&\leq J_2\big(t,\xi;\bar{\Theta}_1,\bar{\hat{\Theta}}_1,\bar{v}_1[\Theta_2,\hat{\Theta}_2,v_2],\Theta_2\big(\bar{x}^{\Theta_2,\hat{\Theta}_2,v_2}
				-\mathbb{E}\bar{x}^{\Theta_2,\hat{\Theta}_2,v_2}\big)+\hat{\Theta}_2\mathbb{E}\bar{x}^{\Theta_2,\hat{\Theta}_2,v_2}+v_2\big),\\
				&\qquad\forall\, (\Theta_2(\cdot),\hat{\Theta}_2(\cdot),v_2(\cdot))\in \mathcal{Q}_2[t,T] \times \mathcal{Q}_2[t,T] \times \mathcal{U}_2[t,T],
			\end{aligned}
		\end{equation}
		where $\bar{x}(\cdot)\equiv\bar{x}^{\bar{\Theta}_2,\bar{\hat{\Theta}}_2,\bar{v}_2}(\cdot)$.
	\end{mydef}
	
	Moreover, we have
	\begin{equation*}
		\begin{aligned}
			\bar{u}_1(\cdot)&=\bar{\Theta}_1(\cdot)\big(\bar{x}(\cdot)-\mathbb{E}\bar{x}(\cdot)\big)+\bar{\hat{\Theta}}_1(\cdot)\mathbb{E}\bar{x}(\cdot)
			+\bar{v}_1[\bar{\Theta}_2\big(\bar{x}-\mathbb{E}\bar{x}\big)+\bar{\hat{\Theta}}_2\mathbb{E}\bar{x}+\bar{v}_2](\cdot)\\
			&:=\big[\bar{\Theta}_1(\cdot)+\bar{\Gamma}_1[\bar{\Theta}_2,\bar{\hat{\Theta}}_2,\bar{v}_2](\cdot)\big]\big(\bar{x}(\cdot)-\mathbb{E}\bar{x}(\cdot)\big)\\
			&\quad +\big[\bar{\hat{\Theta}}_1(\cdot)+\bar{\Gamma}_2[\bar{\Theta}_2,\bar{\hat{\Theta}}_2,\bar{v}_2](\cdot)\big]\mathbb{E}\bar{x}(\cdot)
			+\bar{\Gamma}_3[\bar{\Theta}_2,\bar{\hat{\Theta}}_2,\bar{v}_2](\cdot)\\
			&:=\bar{\Xi}_1(\cdot)\big(\bar{x}(\cdot)-\mathbb{E}\bar{x}(\cdot)\big)+\bar{\hat{\Xi}}_1(\cdot)\mathbb{E}\bar{x}(\cdot)+\bar{V}_1(\cdot),
		\end{aligned}
	\end{equation*}
	then the 6-tuple $(\bar{\Xi}_1(\cdot),\bar{\hat{\Xi}}_1(\cdot),\bar{V}_1(\cdot),\bar{\Theta}_2(\cdot),\bar{\hat{\Theta}}_2(\cdot),\bar{v}_2(\cdot))$ is called a (unique) closed-loop Stackelberg equilibrium of our mean-field type LQ Stackelberg stochastic differential game on $[t,T]$.
	
	\section{Optimization problem of the follower}\label{follower's problem}
	
	In this section, we consider the optimization problem of the follower. For any given $u_2(\cdot) \in \mathcal{U}_2[t,T]$, the follower wants to solve the following mean-field type SLQ optimal control problem.
	
	\textbf{Problem (MF-SLQ)$_f$}. For any initial pair $(t,\xi) \in [0,T]\times L^2_{\mathcal{F}_t}(\Omega;\mathbb{R}^n)$, and given $u_2(\cdot) \in \mathcal{U}_2[t,T]$, find $\bar{u}_1(\cdot) \in \mathcal{U}_1[t,T]$ such that
	\begin{equation}\label{follower problem}
		J_1(t,\xi;\bar{u}_1(\cdot),u_2(\cdot))=\underset{u_1(\cdot)\in\, \mathcal{U}_1[t,T]} {\min}J_1(t,\xi;u_1(\cdot),u_2(\cdot)) \equiv V_1(t,\xi;u_2(\cdot)).
	\end{equation}
	
	It is worth noting that the strategy of the follower depends on the choice of the leader. If there exists a (unique) $\bar{u}_1(\cdot) \in \mathcal{U}_1[t,T]$ such that (\ref{follower problem}) holds, we say that Problem (MF-SLQ)$_f$ is (uniquely) open-loop solvable. $\bar{u}_1(\cdot)$ is called an (unique) open-loop optimal control, and $(\bar{x}^{u_2}(\cdot),\bar{u}_1(\cdot))\equiv(x(\cdot;t,\xi,\bar{u}_1(\cdot),u_2(\cdot)),\bar{u}_1(\cdot))$ is called an open-loop optimal pair of Problem (MF-SLQ)$_f$. The map $V_1(\cdot,\cdot; u_2(\cdot))$ is called the value function of Problem (MF-SLQ)$_f$.
	
	In particular, we denote the follower's problem as \textbf{Problem (MF-SLQ)$_{f0}$} and the cost functional as $J_1^0(t,\xi;u_1(\cdot),0)$, when $b=\sigma=g^1=\hat{g}^1=q^1=\rho^1_1=\rho^1_2=0$ and $u_2(\cdot)\equiv0$.
	
	The following result is very similar to \cite{Sun2017}, which gives sufficient and necessary conditions for the open-loop solvability for Problem (MF-SLQ)$_f$.
	
	\begin{mylem}\label{follower open-loop ns condition}
		Let (H1)-(H2) hold. For any given $u_2(\cdot) \in \mathcal{U}_2[t,T]$ and any initial pair $(t,\xi) \in [0,T]\times L^2_{\mathcal{F}_t}(\Omega;\mathbb{R}^n)$, a control $\bar{u}_1(\cdot) \in \mathcal{U}_1[t,T]$ is an open-loop optimal control of Problem (MF-SLQ)$_f$ if and only if the following holds:
		\begin{equation}\label{follower open-loop stationarity condition}
			\begin{aligned}		
				&0=B_1^\top(\bar{y}-\mathbb{E}\bar{y})+\boldsymbol{B}_1^\top\mathbb{E}\bar{y}+D^\top_1(\bar{z}-\mathbb{E}\bar{z})+\boldsymbol{D}_1^\top\mathbb{E}\bar{z}
				+S^1_1(\bar{x}^{u_2}-\mathbb{E}\bar{x}^{u_2})+\boldsymbol{S}^1_1\mathbb{E}\bar{x}^{u_2}\\
				&\quad +R^1_{11}(\bar{u}_1-\mathbb{E}\bar{u}_1)+\boldsymbol{R}^1_{11}\mathbb{E}\bar{u}_1+R^1_{21}u_2+\hat{R}^1_{21}\mathbb{E}u_2+\rho^1_1,\quad a.e.,\, \mathbb{P}\mbox{-}a.s.,
			\end{aligned}
		\end{equation}
		where $(\bar{x}^{u_2}(\cdot),\bar{y}(\cdot),\bar{z}(\cdot)) \in L^2_{\mathbb{F}}(t,T;\mathbb{R}^n) \times L^2_{\mathbb{F}}(t,T;\mathbb{R}^n) \times L^2_{\mathbb{F}}(t,T;\mathbb{R}^n)$ is the adapted solution to the following MF-FBSDE:
		\begin{equation}\label{open-loop Hamiltonian}
			\hspace{-1.8mm}\left\{\begin{aligned}
				d\bar{x}^{u_2}(s)&=\big\{A(\bar{x}^{u_2}-\mathbb{E}\bar{x}^{u_2})+\boldsymbol{A}\mathbb{E}\bar{x}^{u_2}+B_1(\bar{u}_1-\mathbb{E}\bar{u}_1)
				+\boldsymbol{B}_1\mathbb{E}\bar{u}_1+B_2u_2+\hat{B}_2\mathbb{E}u_2+b \big\}ds\\
				&\quad +\big\{C(\bar{x}^{u_2}-\mathbb{E}\bar{x}^{u_2})+\boldsymbol{C}\mathbb{E}\bar{x}^{u_2}+D_1(\bar{u}_1-\mathbb{E}\bar{u}_1)+\boldsymbol{D}_1\mathbb{E}\bar{u}_1+D_2u_2
				+\hat{D}_2\mathbb{E}u_2+\sigma\big\}dW(s),\\
				-d\bar{y}(s)&=\big\{A^\top(\bar{y}-\mathbb{E}\bar{y})+\boldsymbol{A}^\top\mathbb{E}\bar{y}+C^\top(\bar{z}-\mathbb{E}\bar{z})+\boldsymbol{C}^\top\mathbb{E}\bar{z}+Q^1(\bar{x}^{u_2}-\mathbb{E}\bar{x}^{u_2})\\
				&\qquad +\boldsymbol{Q}^1\mathbb{E}\bar{x}^{u_2}+S^{1\top}_1(\bar{u}_1-\mathbb{E}\bar{u}_1)+{\boldsymbol{S}^1_1}^\top\mathbb{E}\bar{u}_1 +S^{1\top}_2u_2+\hat{S}^{1\top}_2\mathbb{E}u_2+q^1\big\}ds-\bar{z}dW(s),\\
				\bar{x}^{u_2}(t)&=\xi,\quad \bar{y}(T)=G^1\big(\bar{x}^{u_2}(T)-\mathbb{E}\bar{x}^{u_2}(T)\big)+\boldsymbol{G}^1\mathbb{E}\bar{x}^{u_2}(T)+\boldsymbol{g}^1,
			\end{aligned}\right.
		\end{equation}
		and the following convexity condition holds: For any $u_1(\cdot) \in \mathcal{U}_1[t,T]$,
		\begin{equation}
			\begin{aligned}
				&\mathbb{E}\bigg\{ \big\langle G^1\big(x_0(T)-\mathbb{E}x_0(T)\big),x_0(T)-\mathbb{E}x_0(T)\big\rangle+\big\langle \boldsymbol{G}^1\mathbb{E}x_0(T),\mathbb{E}x_0(T)\big\rangle\\
				&\quad +\int_0^T \Big[\big\langle Q^1(x_0-\mathbb{E}x_0),x_0-\mathbb{E}x_0\big\rangle +\big\langle \boldsymbol{Q}^1\mathbb{E}x_0,\mathbb{E}x_0 \big\rangle
				+2\big\langle S^1_1(x_0-\mathbb{E}x_0),u_1-\mathbb{E}u_1 \big\rangle\\
				&\qquad\qquad +2\big\langle \boldsymbol{S}^1_1\mathbb{E}x_0,\mathbb{E}u_1 \big\rangle +\big\langle R^1_{11}(u_1-\mathbb{E}u_1),u_1-\mathbb{E}u_1 \big\rangle
				+\big\langle \boldsymbol{R}^1_{11}\mathbb{E}u_1,\mathbb{E}u_1 \big\rangle\Big]ds\bigg\} \geq 0,
			\end{aligned}
		\end{equation}
		where $x_0(\cdot) \in L^2_{\mathbb{F}}(t,T;\mathbb{R}^n)$ is the adapted solution to the following MF-SDE:
		\begin{equation}\left\{
			\begin{aligned}
				dx_0(s)&=\big\{A(x_0-\mathbb{E}x_0)+\boldsymbol{A}\mathbb{E}x_0+B_1(u_1-\mathbb{E}u_1)+\boldsymbol{B}_1\mathbb{E}u_1\big\}ds\\
				&\quad+\big\{C(x_0-\mathbb{E}x_0)+\boldsymbol{C}\mathbb{E}x_0+D_1(u_1-\mathbb{E}u_1)+\boldsymbol{D}_1\mathbb{E}u_1\big\}dW(s),\\
				x_0(t)&=0.
			\end{aligned}\right.
		\end{equation}
	\end{mylem}
	
	The following result is directly from Proposition 2.4 of \cite{LSY2016}.
	
	\begin{mylem}\label{relation}
		Let (H1)-(H2) hold and let $(\bar{\Theta}_1(\cdot),\bar{\hat{\Theta}}_1(\cdot),\bar{v}_1(\cdot)) \in \mathcal{Q}_1[t,T] \times\mathcal{Q}_1[t,T] \times \mathcal{U}_1[t,T]$. Then for any given $u_2(\cdot) \in \mathcal{U}_2[t,T]$, the following statements are equivalent:\\
		(\romannumeral 1) $(\bar{\Theta}_1(\cdot),\bar{\hat{\Theta}}_1(\cdot),\bar{v}_1(\cdot))$ is a closed-loop optimal strategy of Problem (MF-SLQ)$_f$ on $[t,T]$.\\
		(\romannumeral 2) For any initial pair $(t,\xi) \in [0,T]\times L^2_{\mathcal{F}_t}(\Omega;\mathbb{R}^n)$ and any $v_1(\cdot) \in \mathcal{U}_1[t,T]$,
		\begin{equation*}
			J_1\big(t,\xi;\bar{\Theta}_1(\bar{x}^{u_2}-\mathbb{E}\bar{x}^{u_2})+\bar{\hat{\Theta}}_1\mathbb{E}\bar{x}^{u_2}+\bar{v}_1,u_2(\cdot)\big)
			\leq J_1\big(t,\xi;\bar{\Theta}_1(x^{u_2}-\mathbb{E}x^{u_2})+\bar{\hat{\Theta}}_1\mathbb{E}x^{u_2}+v_1,u_2(\cdot)\big),
		\end{equation*}
		where $\bar{x}^{u_2}(\cdot)\equiv x(\cdot;t,\xi,\bar{\Theta}_1(\cdot),\bar{\hat{\Theta}}_1(\cdot),\bar{v}_1(\cdot),u_2(\cdot))$
		and $x^{u_2}(\cdot)\equiv x(\cdot;t,\xi,\bar{\Theta}_1(\cdot),\bar{\hat{\Theta}}_1(\cdot),v_1(\cdot),u_2(\cdot))$.\\
		(\romannumeral 3) For any initial pair $(t,\xi) \in [0,T]\times L^2_{\mathcal{F}_t}(\Omega;\mathbb{R}^n)$ and any $u_1(\cdot) \in \mathcal{U}_1[t,T]$,
		\begin{equation*}
			J_1\big(t,\xi;\bar{\Theta}_1(\bar{x}^{u_2}-\mathbb{E}\bar{x}^{u_2})+\bar{\hat{\Theta}}_1\mathbb{E}\bar{x}^{u_2}+\bar{v}_1,u_2(\cdot)\big) \leq J_1(t,\xi;u_1(\cdot),u_2(\cdot)),
		\end{equation*}
		where $\bar{x}^{u_2}(\cdot)\equiv x(\cdot;t,\xi,\bar{\Theta}_1(\cdot),\bar{\hat{\Theta}}_1(\cdot),\bar{v}_1(\cdot),u_2(\cdot))$.	
	\end{mylem}
	
	\begin{Remark}
		According to the equivalence relation of (\romannumeral 1) and (\romannumeral 2), $(\bar{\Theta}_1(\cdot),\bar{\hat{\Theta}}_1(\cdot),\bar{v}_1(\cdot)) \in \mathcal{Q}_1[t,T] \times\mathcal{Q}_1[t,T] \times \mathcal{U}_1[t,T]$ is the closed-loop optimal strategy of Problem (MF-SLQ)$_f$  if and only if $\bar{v}_1(\cdot)$ is the open-loop optimal control of the mean-field type SLQ optimal control problem whose systems is (\ref{follower closed-loop state}) and (\ref{follower closed-loop cost}) with $\Theta_1(\cdot)=\bar{\Theta}_1(\cdot)$ and $\hat{\Theta}_1(\cdot)=\bar{\hat{\Theta}}_1(\cdot)$. Let us label this stochastic optimal control problem for $v_1(\cdot)$ as \textbf{Problem (MF-SLQ)$^{\bar{\Theta}_1,\bar{\hat{\Theta}}_1}_f$}.
	\end{Remark}
	
	By Lemma \ref{follower open-loop ns condition}, when $(\bar{\Theta}_1(\cdot),\bar{\hat{\Theta}}_1(\cdot),\bar{v}_1(\cdot))$ is the closed-loop optimal strategy of Problem (MF-SLQ)$_f$, we obtain the following optimality system:
	\begin{equation}\label{follower optimal system}
		\left\{\begin{aligned}
			d\bar{x}^{u_2}(s)&=\big\{(A+B_1\bar{\Theta}_1)(\bar{x}^{u_2}-\mathbb{E}\bar{x}^{u_2})+(\boldsymbol{A}+\boldsymbol{B}_1\bar{\hat{\Theta}}_1)\mathbb{E}\bar{x}^{u_2}
			+B_1(\bar{v}_1-\mathbb{E}\bar{v}_1)+\boldsymbol{B}_1\mathbb{E}\bar{v}_1\\
			&\qquad +B_2u_2+\hat{B}_2\mathbb{E}u_2+b\big\}ds+\big\{(C+D_1\bar{\Theta}_1)(\bar{x}^{u_2}-\mathbb{E}\bar{x}^{u_2})+(\boldsymbol{C}+\boldsymbol{D}_1\bar{\hat{\Theta}}_1)\mathbb{E}\bar{x}^{u_2}\\
			&\qquad +D_1(\bar{v}_1-\mathbb{E}\bar{v}_1)+\boldsymbol{D}_1\mathbb{E}\bar{v}_1+D_2u_2+\hat{D}_2\mathbb{E}u_2+\sigma\big\}dW(s),\\
			-d\bar{y}^{u_2}(s)&=\big\{(A+B_1\bar{\Theta}_1)^\top(\bar{y}^{u_2}-\mathbb{E}\bar{y}^{u_2})+(\boldsymbol{A}+\boldsymbol{B}_1\bar{\hat{\Theta}}_1)^\top\mathbb{E}\bar{y}^{u_2}
			+(C+D_1\bar{\Theta}_1)^\top(\bar{z}^{u_2}-\mathbb{E}\bar{z}^{u_2})\\
			&\qquad +(\boldsymbol{C}+\boldsymbol{D}_1\bar{\hat{\Theta}}_1)^\top\mathbb{E}\bar{z}^{u_2}+(Q^1+S^1_1\bar{\Theta}_1
			+\bar{\Theta}_1^\top S^1_1+\bar{\Theta}_1R^1_{11}\bar{\Theta}_1)(\bar{x}^{u_2}-\mathbb{E}\bar{x}^{u_2})\\	
			&\qquad +(\boldsymbol{Q}^1+\bar{\hat{\Theta}}^\top_1\boldsymbol{S}^1_1+{\boldsymbol{S}^1_1}^\top\bar{\hat{\Theta}}_1
			+\bar{\hat{\Theta}}^\top_1\boldsymbol{R}^1_{11}\bar{\hat{\Theta}}_1)\mathbb{E}\bar{x}^{u_2}+(S^1_1+R^1_{11}\bar{\Theta}_1)^\top(\bar{v}_1-\mathbb{E}\bar{v}_1)\\
			&\qquad +(\boldsymbol{S}^1_1+\boldsymbol{R}^1_{11}\bar{\hat{\Theta}}_1)^\top\mathbb{E}\bar{v}_1+(S^1_2+R^1_{12}\bar{\Theta}_1)^\top(u_2-\mathbb{E}u_2)
			+(\boldsymbol{S}^1_2+\boldsymbol{R}^1_{12}\bar{\hat{\Theta}}_1)^\top\mathbb{E}u_2\\
			&\qquad +\bar{\Theta}^\top_1(\rho_1^1-\mathbb{E}\rho_1^1)+\bar{\hat{\Theta}}_1^\top\mathbb{E}\rho_1^1+q^1\big\}ds-\bar{z}^{u_2}dW(s),\\
			\bar{x}^{u_2}(t)&=\xi,\quad \bar{y}^{u_2}(T)=G^1\big(\bar{x}^{u_2}(T)-\mathbb{E}\bar{x}^{u_2}(T)\big)+\boldsymbol{G}^1\mathbb{E}\bar{x}^{u_2}(T)+\boldsymbol{g}^1,\\
		\end{aligned}\right.
	\end{equation}
	with
	\begin{equation}\label{closed-loop stationarity condition of follower}
		\begin{aligned}
			0=&\ B^\top_1(\bar{y}^{u_2}-\mathbb{E}\bar{y}^{u_2})+\boldsymbol{B}_1^\top\mathbb{E}\bar{y}^{u_2}+D^\top_1(\bar{z}^{u_2}-\mathbb{E}\bar{z}^{u_2})+\boldsymbol{D}_1^\top\mathbb{E}\bar{z}^{u_2}\\
			&+(S^1_1+R^1_{11}\bar{\Theta}_1)(\bar{x}^{u_2}-\mathbb{E}\bar{x}^{u_2})+(\boldsymbol{S}^1_1+\boldsymbol{R}^1_{11}\bar{\hat{\Theta}}_1)\mathbb{E}\bar{x}^{u_2}+R^1_{11}(\bar{v}_1-\mathbb{E}\bar{v}_1)\\
			&+\boldsymbol{R}^1_{11}\mathbb{E}\bar{v}_1+R^1_{21}(u_2-\mathbb{E}u_2)+\boldsymbol{R}^1_{21}\mathbb{E}u_2+\rho_1^1,\quad a.e.,\, \mathbb{P}\mbox{-}a.s..
		\end{aligned}
	\end{equation}
	
	\begin{Remark}
		From the equivalence of (\romannumeral 1) and (\romannumeral 3), if Problem (MF-SLQ)$_f$ is closed-loop solvable, it must be open-loop solvable and the outcome $\bar{u}_1(\cdot)=\bar{\Theta}_1(\cdot)\big(\bar{x}^{u_2}(\cdot)-\mathbb{E}\bar{x}^{u_2}(\cdot)\big)+\bar{\hat{\Theta}}_1(\cdot)\mathbb{E}\bar{x}^{u_2}(\cdot)+\bar{v}_1(\cdot)$ is its open-loop optimal control.
	\end{Remark}
	
	Similar to Proposition 3.1 of \cite{LSY2016}, we have the following Lemma.
	
	\begin{mylem}\label{transformation}
		Let (H1)-(H2) hold. If $(\bar{\Theta}_1(\cdot),\bar{\hat{\Theta}}_1(\cdot),\bar{v}_1(\cdot)) \in \mathcal{Q}_1[t,T] \times\mathcal{Q}_1[t,T] \times \mathcal{U}_1[t,T]$ is the closed-loop optimal strategy of Problem (MF-SLQ)$_f$, then $(\bar{\Theta}_1(\cdot),\bar{\hat{\Theta}}_1(\cdot),0)$ is the closed-loop optimal strategy of Problem (MF-SLQ)$_{f0}$ on $[t,T]$.
	\end{mylem}
	
	\begin{Remark}
		We note that in \cite{LSY2016}, they first use Proposition 3.1 to transform the original problem into a homogeneous mean-field type SLQ optimal control problem, then they further transform this problem into a deterministic optimal control problem with two control variables. However, in this paper, our work is to transform the closed-loop solvability problem of the follower into an open-loop solvability problem of an ideal mean-field type SLQ optimal control problem whose open-loop optimal control is known, which is different from \cite{LSY2016}. See below for details.
	\end{Remark}
	
	Now let us consider Problem (MF-SLQ)$_{f0}$ with its closed-loop optimal strategy $(\bar{\Theta}_1(\cdot),\bar{\hat{\Theta}}_1(\cdot),\\0)$. Using the equivalence relationship of (\romannumeral 1) and (\romannumeral 2) in Lemma \ref{relation} again, we can know that $\bar{v}_1(\cdot)=0$ is the open-loop optimal control of the following mean-field type SLQ optimal control problem. The state equation is
	\begin{equation}\label{Problem follower (f0) closed-loop state}
		\left\{\begin{aligned}
			dx^0(s)&=\big\{(A+B_1\bar{\Theta}_1)(x^0-\mathbb{E}x^0)+B_1(v_1-\mathbb{E}v_1)+(\boldsymbol{A}+\boldsymbol{B}_1\bar{\hat{\Theta}}_1)\mathbb{E}x^0+\boldsymbol{B}_1\mathbb{E}v_1\big\}ds\\
			&\quad +\big\{(C+D_1\bar{\Theta}_1)\big(x^0-\mathbb{E}x^0\big)+D_1(v_1-\mathbb{E}v_1)+(\boldsymbol{C}+\boldsymbol{D}_1\bar{\hat{\Theta}}_1)\mathbb{E}x^0+\boldsymbol{D}_1\mathbb{E}v_1\big\}dW(s),\\
			x^0(t)&=\xi,
		\end{aligned}\right.
	\end{equation}
	and the cost functional is
	\begin{equation}\label{Problem follower (f0) closed-loop cost}
		\begin{aligned}
			&\tilde{J}_1(t,\xi;v_1(\cdot))=\mathbb{E}\bigg\{\big\langle G^1\big(x^0(T)-\mathbb{E}x^0(T)\big),x^0(T)-\mathbb{E}x^0(T)\big\rangle+\big\langle \boldsymbol{G}^1\mathbb{E}x^0(T),\mathbb{E}x^0(T)\big\rangle\\
			&\quad +\int_t^T \Big[\big\langle(Q^1+\bar{\Theta}_1^\top S^1_1+S^{1\top}_1\bar{\Theta}_1+\bar{\Theta}_1^\top R^1_{11}\bar{\Theta}_1)(x^0-\mathbb{E}x^0),x^0-\mathbb{E}x^0\big\rangle\\
			&\qquad\quad +2\big\langle(S^1_1+R^1_{11}\bar{\Theta}_1)(x^0-\mathbb{E}x^0),v_1-\mathbb{E}v_1\big\rangle+\big\langle R^1_{11}(v_1-\mathbb{E}v_1),v_1-\mathbb{E}v_1\big\rangle\Big]ds\\
			&\quad +\int_t^T \Big[\big\langle(\boldsymbol{Q}^1+\bar{\hat{\Theta}}_1^\top \boldsymbol{S}^1_1+{\boldsymbol{S}^1_1}^\top\bar{\hat{\Theta}}_1
			+\bar{\hat{\Theta}}_1^\top \boldsymbol{R}^1_{11}\bar{\hat{\Theta}}_1)\mathbb{E}x^0,\mathbb{E}x^0\big\rangle\\
			&\qquad\quad +2\big\langle(\boldsymbol{S}^1_1+\boldsymbol{R}^1_{11}\bar{\hat{\Theta}}_1)\mathbb{E}x^0,\mathbb{E}v_1\big\rangle
			+\big\langle\boldsymbol{R}^1_{11}\mathbb{E}v_1,\mathbb{E}v_1\big\rangle\Big]ds\bigg\}.
		\end{aligned}
	\end{equation}
	By Lemma \ref{follower open-loop ns condition}, we can obtain the following optimality system: (Noting that $\bar{v}_1(\cdot)\equiv0$.)
	\begin{equation}\label{MF-SLQ-f0-Theta optimal system}
		\left\{
		\begin{aligned}
			d\bar{x}^0(s)&=\big\{(A+B_1\bar{\Theta}_1)(\bar{x}^0-\mathbb{E}\bar{x}^0)+(\boldsymbol{A}+\boldsymbol{B}_1\bar{\hat{\Theta}}_1)\mathbb{E}\bar{x}^0\big\}ds\\
			&\quad +\big\{(C+D_1\bar{\Theta}_1)(\bar{x}^0-\mathbb{E}\bar{x}^0)+(\boldsymbol{C}+\boldsymbol{D}_1\bar{\hat{\Theta}}_1)\mathbb{E}\bar{x}^0\big\}dW(s),\\
			-d\bar{y}^0(s)&=\big\{(A+B_1\bar{\Theta}_1)^\top(\bar{y}^0-\mathbb{E}\bar{y}^0)+(\boldsymbol{A}+\boldsymbol{B}_1\bar{\hat{\Theta}}_1)^\top\mathbb{E}\bar{y}^0+(C+D_1\bar{\Theta}_1\big)^\top\big(\bar{z}^0-\mathbb{E}\bar{z}^0)\\
			&\qquad +(\boldsymbol{C}+\boldsymbol{D}_1\bar{\hat{\Theta}}_1)^\top\mathbb{E}\bar{z}^0+(Q^1+\bar{\Theta}_1^\top S^1_1+S^{1\top}_1\bar{\Theta}_1
			+\bar{\Theta}_1^\top R^1_{11}\bar{\Theta}_1)(\bar{x}^0-\mathbb{E}\bar{x}^0)\\
			&\qquad +(\boldsymbol{Q}^1+\bar{\hat{\Theta}}_1^\top \boldsymbol{S}^1_1+{\boldsymbol{S}^1_1}^\top\bar{\hat{\Theta}}_1
			+\bar{\hat{\Theta}}_1^\top \boldsymbol{R}^1_{11}\bar{\hat{\Theta}}_1)\mathbb{E}\bar{x}^0\big\}ds-\bar{z}^0dW(s),\\
			\bar{x}^0(t)&=\xi,\quad \bar{y}^0(T)=G^1\big(\bar{x}^0(T)-\mathbb{E}\bar{x}^0(T)\big)+\boldsymbol{G}^1\mathbb{E}\bar{x}^0(T),\\
			0=&\ B_1^\top(\bar{y}^0-\mathbb{E}\bar{y}^0)+\boldsymbol{B}_1^\top\mathbb{E}\bar{y}^0+D_1^\top(\bar{z}^0-\mathbb{E}\bar{z}^0)+\boldsymbol{D}_1^\top\mathbb{E}\bar{z}^0\\
			& +(S^1_1+R^1_{11}\bar{\Theta}_1)(\bar{x}^0-\mathbb{E}\bar{x}^0)+(\boldsymbol{S}^1_1+\boldsymbol{R}^1_{11}\bar{\hat{\Theta}}_1)\mathbb{E}\bar{x}^0,\quad a.e.,\ \mathbb{P}\mbox{-}a.s..
		\end{aligned}\right.
	\end{equation}
	
	With some straightforward computation on (\ref{MF-SLQ-f0-Theta optimal system}), we have
	\begin{equation}\label{X0,X0-EX0}
		\left\{\begin{aligned}
			\mathbb{E}\bar{x}^0(s)&=\mathbb{E}\xi\cdot \exp\biggl\{\int_{t}^{s}\big(\boldsymbol{A}+\boldsymbol{B}_1\bar{\hat{\Theta}}_1\big)dr\biggr\},\\
			(\bar{x}^0-\mathbb{E}\bar{x}^0)(s)&=\frac{1}{L(s)}\biggl\{\xi-\mathbb{E}\xi-\int_t^s\lambda\mathcal{C} \cdot L(r) dr+\int_t^s \lambda \cdot L(r)dW(r)\biggr\},\quad s\in [t,T],
			\quad
		\end{aligned}\right.
	\end{equation}
	for any $\xi \in L^2_{\mathcal{F}_t}(\Omega;\mathbb{R}^n)$, where
	\begin{equation*}
		\mathcal{A}:=A+B_1\bar{\Theta}_1,\quad \mathcal{C}:=C+D_1\bar{\Theta}_1,\quad \lambda:=\big(\boldsymbol{C}+\boldsymbol{D}_1\bar{\hat{\Theta}}_1\big)\mathbb{E}\bar{x}^0,
	\end{equation*}
	and
	\begin{equation*}
		L(s):=\exp\bigg\{\int_t^s\Big(\frac{\mathcal{C}^2}{2}-\mathcal{A}\Big)dr-\int_t^s\mathcal{C}dW\bigg\},\quad s\in [t,T].
	\end{equation*}
	
	Notice that $\bar{y}^0(T)=G^1\big(\bar{x}^0(T)-\mathbb{E}\bar{x}^0(T)\big)+\boldsymbol{G}^1\mathbb{E}\bar{x}^0(T)$, we assume that
	\begin{equation*}
		\bar{y}^0(\cdot)=P_1(\cdot)(\bar{x}^0-\mathbb{E}\bar{x}^0)(\cdot)+\Pi_1(\cdot)\mathbb{E}\bar{x}^0(\cdot),
	\end{equation*}
	where $(P_1(\cdot),\Pi_1(\cdot)) \in C([t,T],\mathbb{S}^n) \times C([t,T],\mathbb{S}^n) $ with $P_1(T)=G^1$, $\Pi_1(T)=\boldsymbol{G}^1$. Then we can obtain
	\begin{equation}\label{follower decoupled relation}
		\left\{\begin{aligned}
			\mathbb{E}\bar{y}^0(\cdot)&=\Pi_1(\cdot)\mathbb{E}\bar{x}^0(\cdot),\\
			\bar{y}^0(\cdot)-\mathbb{E}\bar{y}^0(\cdot)&=P_1(\cdot)(\bar{x}^0-\mathbb{E}\bar{x}^0)(\cdot).
		\end{aligned}\right.
	\end{equation}
	Applying It\^o's formula to the second equation in (\ref{follower decoupled relation}), we have
	\begin{equation}\label{y0-Ey0}
		\begin{aligned}
			&d(\bar{y}^0-\mathbb{E}\bar{y}^0)(s)=\big\{\dot{P}_1(\bar{x}^0-\mathbb{E}\bar{x}^0)+P_1(A+B_1\bar{\Theta}_1)(\bar{x}^0-\mathbb{E}\bar{x}^0)\big\}ds\\
			&\qquad +\big\{P_1(C+D_1\bar{\Theta}_1)(\bar{x}^0-\mathbb{E}\bar{x}^0)+P_1(\boldsymbol{C}+\boldsymbol{D}_1\bar{\hat{\Theta}}_1)\mathbb{E}\bar{x}^0\big\}dW(s)\\
			&=-\big\{(A+B_1\bar{\Theta}_1)^\top(\bar{y}^0-\mathbb{E}\bar{y}^0)+(C+D_1\bar{\Theta}_1)^\top(\bar{z}^0-\mathbb{E}\bar{z}^0)\\
			&\qquad\quad +(Q^1+\bar{\Theta}_1^\top S^1_1+S^{1\top}_1\bar{\Theta}_1+\bar{\Theta}_1^\top R^1_{11}\bar{\Theta}_1)(\bar{x}^0-\mathbb{E}\bar{x}^0)\big\}ds+\bar{z}^0dW(s).
		\end{aligned}
	\end{equation}
	Thus,
	\begin{equation}\label{z0}
		\bar{z}^0(\cdot)=P_1(C+D_1\bar{\Theta}_1)(\bar{x}^0-\mathbb{E}\bar{x}^0)(\cdot)+P_1(\boldsymbol{C}+\boldsymbol{D}_1\bar{\hat{\Theta}}_1)\mathbb{E}\bar{x}^0(\cdot),\ \mathbb{P}\mbox{-}a.s..
	\end{equation}
	Then
	\begin{equation}\label{z0-Ez0,Ez0}
		\left\{\begin{aligned}
			&\mathbb{E}\bar{z}^0(\cdot)=P_1(\boldsymbol{C}+\boldsymbol{D}_1\bar{\hat{\Theta}}_1)\mathbb{E}\bar{x}^0(\cdot),\\
			&\bar{z}^0(\cdot)-\mathbb{E}\bar{z}^0(\cdot)=P_1(C+D_1\bar{\Theta}_1)(\bar{x}^0-\mathbb{E}\bar{x}^0)(\cdot),\ \mathbb{P}\mbox{-}a.s.,
		\end{aligned}\right.
	\end{equation}
	and
	\begin{equation*}
		\begin{aligned}
			0&=\big[\dot{P}_1+P_1(A+B_1\bar{\Theta}_1)+(A+B_1\bar{\Theta}_1)^\top P_1+(C+D_1\bar{\Theta}_1)^\top P_1(C+D_1\bar{\Theta}_1)\\
			&\qquad +Q^1+\bar{\Theta}_1^\top S^1_1+S^{1\top}_1\bar{\Theta}_1+\bar{\Theta}_1^\top R^1_{11}\bar{\Theta}_1\big](\bar{x}^0-\mathbb{E}\bar{x}^0),\quad a.e.\ \mathbb{P}\mbox{-}a.s..
		\end{aligned}
	\end{equation*}
	Thus, we obtain the following Lyapunov equation of $P_1(\cdot)\in C([t,T],\mathbb{S}^n)$:
	\begin{equation}\label{follower Lyapunov equation P1}
		\left\{\begin{aligned}
			&0=\dot{P}_1+P_1(A+B_1\bar{\Theta}_1)+(A+B_1\bar{\Theta}_1)^\top P_1+(C+D_1\bar{\Theta}_1)^\top P_1(C+D_1\bar{\Theta}_1)\\
			&\qquad +Q^1+\bar{\Theta}_1^\top S^1_1+S^{1\top}_1\bar{\Theta}_1+\bar{\Theta}_1^\top R^1_{11}\bar{\Theta}_1,\quad s\in[t,T],\\
			&P_1(T)=G^1.
		\end{aligned}\right.
	\end{equation}
	Similarly, using It\^o's formula to the first equation in (\ref{follower decoupled relation}), we get another Lyapunov equation of $\Pi_1(\cdot) \in C([t,T],\mathbb{S}^n)$:
	\begin{equation}\label{follower Lyapunov equation Pi1}
		\left\{\begin{aligned}
			&0=\dot{\Pi}_1+\Pi_1(\boldsymbol{A}+\boldsymbol{B}_1\bar{\hat{\Theta}}_1)+(\boldsymbol{A}+\boldsymbol{B}_1\bar{\hat{\Theta}}_1)^\top\Pi_1
			+(\boldsymbol{C}+\boldsymbol{D}_1\bar{\hat{\Theta}}_1)^\top P_1(\boldsymbol{C}+\boldsymbol{D}_1\bar{\hat{\Theta}}_1)\\
			&\qquad +\boldsymbol{Q}^1+\bar{\hat{\Theta}}_1^\top \boldsymbol{S}^1_1+{\boldsymbol{S}^1_1}^\top\bar{\hat{\Theta}}_1
			+\bar{\hat{\Theta}}_1^\top \boldsymbol{R}^1_{11}\bar{\hat{\Theta}}_1,\quad s\in[t,T],\\
			&\Pi_1(T)=\boldsymbol{G}^1.
		\end{aligned}\right.
	\end{equation}
	From the stationarity condition in (\ref{MF-SLQ-f0-Theta optimal system}), we have
	\begin{equation*}
		\left\{\begin{aligned}
			&0=\big[B^\top_1P_1+D^\top_1P_1C+S^1_1+(R^1_{11}+D^\top_1P_1D_1)\bar{\Theta}_1\big](\bar{x}^0-\mathbb{E}\bar{x}^0),\ \mathbb{P}\mbox{-}a.s.,\\
			&0=\big[\boldsymbol{B}_1^\top\Pi_1+\boldsymbol{D}_1^\top P_1\boldsymbol{C}+\boldsymbol{S}^1_1+(\boldsymbol{R}^1_{11}+\boldsymbol{D}_1^\top P_1\boldsymbol{D}_1)\bar{\hat{\Theta}}_1\big]\mathbb{E}\bar{x}^0.
		\end{aligned}\right.
	\end{equation*}
	Since these equations hold for all $(t,\xi) \in [0,T]\times L^2_{\mathcal{F}_t}(\Omega;\mathbb{R}^n)$, we must have (noting that (\ref{X0,X0-EX0}))
	\begin{equation}\label{two Theta and barTheta}
		\left\{\begin{aligned}
			&0=B^\top_1P_1+D^\top_1P_1C+S^1_1+(R^1_{11}+D^\top_1P_1D_1)\bar{\Theta}_1,\\
			&0=\boldsymbol{B}_1^\top\Pi_1+\boldsymbol{D}_1^\top P_1\boldsymbol{C}+\boldsymbol{S}^1_1+(\boldsymbol{R}^1_{11}+\boldsymbol{D}_1^\top P_1\boldsymbol{D}_1)\bar{\hat{\Theta}}_1,\quad a.e.,\,\mathbb{P}\mbox{-}a.s..
		\end{aligned}\right.
	\end{equation}
	
	Denoting $\Sigma_1\equiv R^1_{11}+D^\top_1P_1D_1$ and $\hat{\Sigma}_1\equiv \boldsymbol{R}^1_{11}+\boldsymbol{D}_1^\top P_1\boldsymbol{D}_1$, this implies
	\begin{equation}\label{amage assumption---follower}
		\mathcal{R}(B^\top_1P_1+D^\top_1P_1C+S^1_1) \subseteq \mathcal{R}(\Sigma_1),\quad
		\mathcal{R}(\boldsymbol{B}_1^\top\Pi_1+\boldsymbol{D}_1^\top P_1\boldsymbol{C}+\boldsymbol{S}^1_1)\subseteq \mathcal{R}(\hat{\Sigma}_1),\quad a.e.,\,\mathbb{P}\mbox{-}a.s..
	\end{equation}
	Moreover, since $\Sigma_1^\dagger\Sigma_1$ and $\hat{\Sigma}_1^\dagger\hat{\Sigma}_1$ are orthogonal projections, we see that
	\begin{equation}\label{space assumption---follower}
		\begin{aligned}	
			&\Sigma_1^\dagger(B^\top_1P_1+D^\top_1P_1C+S^1_1) \in L^2(t,T;\mathbb{R}^{m_1\times n}),\\
			&\hat{\Sigma}_1^\dagger(\boldsymbol{B}_1^\top\Pi_1+\boldsymbol{D}_1^\top P_1\boldsymbol{C}+\boldsymbol{S}^1_1) \in L^2(t,T;\mathbb{R}^{m_1\times n}),
		\end{aligned}
	\end{equation}
	and
	\begin{equation}\label{follower closed-loop strategy Theta1 hat{Theta}1}
		\left\{\begin{aligned}
			\bar{\Theta}_1&=-\Sigma_1^\dagger(B^\top_1P_1+D^\top_1P_1C+S^1_1)+(I-\Sigma_1^\dagger\Sigma_1)\theta_1,\\
			\bar{\hat{\Theta}}_1&=-\hat{\Sigma}_1^\dagger(\boldsymbol{B}_1^\top\Pi_1+\boldsymbol{D}_1^\top P_1\boldsymbol{C}+\boldsymbol{S}^1_1)+(I-\hat{\Sigma}_1^\dagger\hat{\Sigma}_1)\hat{\theta}_1,
		\end{aligned}\right.
	\end{equation}
	for some $\big(\theta_1(\cdot),\hat{\theta}_1(\cdot)\big)\in L^2(t,T;\mathbb{R}^{m_1\times n}) \times L^2(t,T;\mathbb{R}^{m_1\times n}) $. By substituting the above equation into (\ref{follower Lyapunov equation P1}) and (\ref{follower Lyapunov equation Pi1}), we get
	\begin{equation}\label{Riccati---1}
		\left\{\begin{aligned}
			&0=\dot{P}_1+P_1A+A^\top P_1+C^\top P_1C+Q^1\\
			&\qquad -(P_1B_1+C^\top PD_1+S^{1\top}_1)\Sigma_1^\dagger(B^\top_1P_1+D^\top_1P_1C+S^1_1),\quad s\in[t,T],\\
			&P_1(T)=G^1,
		\end{aligned}\right.
	\end{equation}
	and
	\begin{equation}\label{Riccati---2}
		\left\{\begin{aligned}
			&0=\dot{\Pi}_1+\Pi_1\boldsymbol{A}+\boldsymbol{A}^\top\Pi_1+\boldsymbol{C}^\top P_1\boldsymbol{C}+\boldsymbol{Q}^1\\
			&\qquad -(\Pi_1\boldsymbol{B}_1+\boldsymbol{C}^\top P_1\boldsymbol{D}_1+{\boldsymbol{S}^1_1}^\top)\hat{\Sigma}_1^\dagger(\boldsymbol{B}_1^\top\Pi_1+\boldsymbol{D}_1^\top P_1\boldsymbol{C}+\boldsymbol{S}^1_1),
			\quad s\in[t,T],\\
			&\Pi_1(T)=\boldsymbol{G}^1.
		\end{aligned}\right.
	\end{equation}
	Using the same method in Theorem 3.3 of \cite{LSY2016}, we can prove that
	\begin{equation}\label{assumption---follower}
		\Sigma_1 \geq 0,\quad  \hat{\Sigma}_1 \geq 0.
	\end{equation}
	
	To determine $\bar{v}_1(\cdot)$, we define
	\begin{equation}\label{follower xyz relationship}
		\left\{\begin{aligned}
			\bar{\eta}_1^{u_2}&=\bar{y}^{u_2}-P_1(\bar{x}^{u_2}-\mathbb{E}\bar{x}^{u_2})-\Pi_1\mathbb{E}\bar{x}^{u_2},\\
			\bar{\zeta}_1^{u_2}&=\bar{z}^{u_2}-P_1\big[(C+D_1\bar{\Theta}_1)(\bar{x}^{u_2}-\mathbb{E}\bar{x}^{u_2})+(\boldsymbol{C}+\boldsymbol{D}_1\bar{\hat{\Theta}}_1)\mathbb{E}\bar{x}^{u_2}\\
			&\quad +D_1(\bar{v}_1-\mathbb{E}\bar{v}_1)+\boldsymbol{D}_1\mathbb{E}\bar{v}_1+D_2(u_2-\mathbb{E}u_2)+\boldsymbol{D}_2\mathbb{E}u_2+\sigma\big].
		\end{aligned}\right.
	\end{equation}
	Then applying It\^{o}'s formula, noting (\ref{follower closed-loop strategy Theta1 hat{Theta}1}), (\ref{Riccati---1}) and (\ref{Riccati---2}), we have
		\begin{equation}\label{follower BSDE}
			\left\{\begin{aligned}
				-d\bar{\eta}_1^{u_2}(s)&=\Big\{\big[A^\top-(P_1B_1+C^\top P_1D_1+S^{1\top}_1)\Sigma_1^{\dagger}B_1^\top\big](\bar{\eta}_1^{u_2}-\mathbb{E}\bar{\eta}_1^{u_2})+\big[C^\top\\
				&\quad\ -(P_1B_1+C^\top P_1D_1+S^{1\top}_1)\Sigma_1^{\dagger}D_1^\top\big](\bar{\zeta}_1^{u_2}-\mathbb{E}\bar{\zeta}_1^{u_2})+\big[P_1B_2+C^\top P_1D_2\\
				&\quad\ +S^{1\top}_2-(P_1B_1+C^\top P_1D_1+S^{1\top}_1)\Sigma_1^\dagger(R^1_{21}+D^\top_1P_1 D_2)\big](u_2-\mathbb{E}u_2)\\
				&\quad\ +P_1(b-\mathbb{E}b)+\big[C^\top-(P_1B_1+C^\top P_1D_1+S^{1\top}_1)\Sigma_1^{\dagger}D_1^\top\big]P_1(\sigma-\mathbb{E}\sigma)\\
				&\quad\ -(P_1B_1+C^\top P_1D_1+S^{1\top}_1)\Sigma_1^{\dagger}(\rho_1^\top-\mathbb{E}\rho^\top_1)-(\Pi_1\boldsymbol{B}_1+\boldsymbol{C}^\top P_1\boldsymbol{D}_1\\
				&\quad\ +{\boldsymbol{S}^1_1}^\top)\hat{\Sigma}_1^{\dagger}\mathbb{E}\rho^\top_1+q^1+\big[\boldsymbol{A}^\top-(\Pi_1\boldsymbol{B}_1
				+\boldsymbol{C}^\top P_1\boldsymbol{D}_1+{\boldsymbol{S}^1_1}^\top)\hat{\Sigma}_1^{\dagger}\boldsymbol{B}_1^\top\big]\mathbb{E}\bar{\eta}_1^{u_2}\\
				&\quad\ +\big[\boldsymbol{C}^\top-(\Pi_1\boldsymbol{B}_1+\boldsymbol{C}^\top P_1\boldsymbol{D}_1+{\boldsymbol{S}^1_1}^\top)\hat{\Sigma}_1^\dagger\boldsymbol{D}_1^\top\big]\mathbb{E}\bar{\zeta}_1^{u_2}
				+\big[\Pi_1\boldsymbol{B}_2+\boldsymbol{C}^\top P_1\boldsymbol{D}_2\\
				&\quad\ +{\boldsymbol{S}^1_2}^\top-(\Pi_1\boldsymbol{B}_1+\boldsymbol{C}^\top P_1\boldsymbol{D}_1+{\boldsymbol{S}^1_1}^\top)\hat{\Sigma}_1^\dagger
				(\boldsymbol{R}^1_{21}+\boldsymbol{D}_1^\top P_1\boldsymbol{D}_2)\big]\mathbb{E}u_2+\big[\boldsymbol{C}^\top\\
				&\quad\ -(\Pi_1\boldsymbol{B}_1+\boldsymbol{C}^\top P_1\boldsymbol{D}_1+{\boldsymbol{S}^1_1}^\top)\hat{\Sigma}_1^\dagger
				\boldsymbol{D}_1^\top\big]P_1\mathbb{E}\sigma+\Pi_1\mathbb{E}b\Big\}ds-\bar{\zeta}_1^{u_2}dW(s),\\
				\bar{\eta}_1^{u_2}(T)&=\boldsymbol{g}^1.
			\end{aligned}\right.
		\end{equation}
		
		According to the stationarity condition (\ref{closed-loop stationarity condition of follower}) and noting (\ref{two Theta and barTheta}), (\ref{follower xyz relationship}), we have
		\begin{equation*}
			\begin{aligned}	
				0&=B_1^\top(\bar{\eta}_1^{u_2}-\mathbb{E}\bar{\eta}_1^{u_2})+\boldsymbol{B}_1^\top\mathbb{E}\bar{\eta}_1^{u_2}+D^\top_1(\bar{\zeta}_1^{u_2}-\mathbb{E}\bar{\zeta}_1^{u_2})\\
				&\quad +\boldsymbol{D}_1^\top\mathbb{E}\bar{\zeta}_1^{u_2}+(R^1_{21}+D_1^\top P_1D_2)(u_2-\mathbb{E}u_2)+(\boldsymbol{R}^1_{21}+\boldsymbol{D}_1^\top P_1\boldsymbol{D}_2)\mathbb{E}u_2\\
				&\quad +D_1^\top P_1(\sigma-\mathbb{E}\sigma)+\boldsymbol{D}_1^\top P_1\mathbb{E}\sigma+\rho_1^1+\Sigma_1(\bar{v}_1-\bar{v}_1)+\hat{\Sigma}_1\mathbb{E}\bar{v}_1,\quad a.e.,\,\mathbb{P}\mbox{-}a.s.,
			\end{aligned}
		\end{equation*}
		and thus
		\begin{equation*}
			\left\{\begin{aligned}
				0&=\boldsymbol{B}_1^\top\mathbb{E}\bar{\eta}_1^{u_2}+\boldsymbol{D}_1^\top\mathbb{E}\bar{\zeta}_1^{u_2}+\boldsymbol{D}_1^\top P_1\mathbb{E}\sigma+\mathbb{E}\rho_1^1+(\boldsymbol{R}^1_{21}
				+\boldsymbol{D}_1^\top P_1\boldsymbol{D}_2)\mathbb{E}u_2+\hat{\Sigma}_1\mathbb{E}\bar{v}_1,\\
				0&=B_1^\top(\bar{\eta}_1^{u_2}-\mathbb{E}\bar{\eta}_1^{u_2})+D^\top_1(\bar{\zeta}_1^{u_2}-\mathbb{E}\bar{\zeta}_1^{u_2})+D_1^\top P_1(\sigma-\mathbb{E}\sigma)+\rho_1^1-\mathbb{E}\rho_1^1\\
				&\quad +(R^1_{21}+D_1^\top P_1D_2)(u_2-\mathbb{E}u_2)+\Sigma_1(\bar{v}_1-\bar{v}_1),\quad a.e.,\,\mathbb{P}\mbox{-}a.s..
			\end{aligned}\right.
		\end{equation*}
		This implies
		\begin{equation}\label{amage assumption---follower---}
			\left\{\begin{aligned}
				&\mathcal{R}\big[\boldsymbol{B}_1^\top\mathbb{E}\bar{\eta}_1^{u_2}+\boldsymbol{D}_1^\top\mathbb{E}\bar{\zeta}_1^{u_2}
				+\boldsymbol{D}_1^\top P_1\mathbb{E}\sigma+\mathbb{E}\rho_1^1+(\boldsymbol{R}^1_{21}+\boldsymbol{D}_1^\top P_1\boldsymbol{D}_2)\mathbb{E}u_2\big] \subseteq \mathcal{R}(\hat{\Sigma}_1),\\
				&\mathcal{R}\big[B_1^\top(\bar{\eta}_1^{u_2}-\mathbb{E}\bar{\eta}_1^{u_2})+D^\top_1(\bar{\zeta}_1^{u_2}-\mathbb{E}\bar{\zeta}_1^{u_2})+D_1^\top P_1(\sigma-\mathbb{E}\sigma)+\rho_1^1-\mathbb{E}\rho_1^1\\
				&\quad +(R^1_{21}+D_1^\top P_1D_2)(u_2-\mathbb{E}u_2)\big] \subseteq \mathcal{R}(\Sigma_1),\quad a.e.,\,\, \mathbb{P}\mbox{-}a.s..
			\end{aligned}\right.
		\end{equation}
		Moreover, since $\Sigma_1^\dagger\Sigma_1$ and $\hat{\Sigma}_1^\dagger\hat{\Sigma}_1$ are orthogonal projections, we see that
		\begin{equation}\label{space assumption---follower---}
			\left\{\begin{aligned}
				&\hat{\Sigma}_1^\dagger\big[\boldsymbol{B}_1^\top\mathbb{E}\bar{\eta}_1^{u_2}+\boldsymbol{D}_1^\top\mathbb{E}\bar{\zeta}_1^{u_2}
				+\boldsymbol{D}_1^\top P_1\mathbb{E}\sigma+\mathbb{E}\rho_1^1+(\boldsymbol{R}^1_{21}+\boldsymbol{D}_1^\top P_1\boldsymbol{D}_2)\mathbb{E}u_2\big]\in L^2(t,T;\mathbb{R}^{m_1}),\\
				&\Sigma_1^\dagger\big[B_1^\top(\bar{\eta}_1^{u_2}-\mathbb{E}\bar{\eta}_1^{u_2})+D^\top_1(\bar{\zeta}_1^{u_2}-\mathbb{E}\bar{\zeta}_1^{u_2})+D_1^\top P_1(\sigma-\mathbb{E}\sigma)+\rho_1^1-\mathbb{E}\rho_1^1\\
				&\quad +(R^1_{21}+D_1^\top P_1D_2)(u_2-\mathbb{E}u_2)\big] \in L^2_{\mathbb{F}}(t,T;\mathbb{R}^{m_1}),
			\end{aligned}\right.
		\end{equation}
		\begin{equation}\label{follower closed-loop strategy v1}
			\left\{\begin{aligned}
				&\mathbb{E}\bar{v}_1=-\hat{\Sigma}_1^\dagger\big[\boldsymbol{B}_1^\top\mathbb{E}\bar{\eta}_1^{u_2}+\boldsymbol{D}_1^\top\mathbb{E}\bar{\zeta}_1^{u_2}
				+\boldsymbol{D}_1^\top P_1\mathbb{E}\sigma+\mathbb{E}\rho_1^1+(\boldsymbol{R}^1_{21}+\boldsymbol{D}_1^\top P_1\boldsymbol{D}_2)\mathbb{E}u_2\big]+(I-\hat{\Sigma}_1^\dagger\hat{\Sigma}_1)\hat{v}_1,\\
				&\bar{v}_1-\mathbb{E}\bar{v}_1=-\Sigma_1^\dagger\big[B_1^\top(\bar{\eta}_1^{u_2}-\mathbb{E}\bar{\eta}_1^{u_2})+D^\top_1(\bar{\zeta}_1^{u_2}-\mathbb{E}\bar{\zeta}_1^{u_2})
				+D_1^\top P_1(\sigma-\mathbb{E}\sigma)+\rho_1^1-\mathbb{E}\rho_1^1\\
				&\qquad\qquad\qquad +(R^1_{21}+D_1^\top P_1D_2)(u_2-\mathbb{E}u_2)\big]+(I-\Sigma_1^\dagger\Sigma_1)v_1,\quad a.e.,\,\mathbb{P}\mbox{-}a.s.,
			\end{aligned}\right.
		\end{equation}
		and thus
		\begin{equation}\label{follower closed-loop optimal}
			\hspace{-2mm}\begin{aligned}
				\bar{v}_1&=-\Sigma_1^\dagger\big[B_1^\top(\bar{\eta}_1^{u_2}-\mathbb{E}\bar{\eta}_1^{u_2})+D^\top_1(\bar{\zeta}_1^{u_2}
				-\mathbb{E}\bar{\zeta}_1^{u_2})+D_1^\top P_1(\sigma-\mathbb{E}\sigma)+\rho_1^1-\mathbb{E}\rho_1^1\\
				&\qquad\quad +(R^1_{21}+D_1^\top P_1D_2)(u_2-\mathbb{E}u_2)\big]+(I-\Sigma_1^\dagger\Sigma_1)v_1+(I-\hat{\Sigma}_1^\dagger\hat{\Sigma}_1)\hat{v}_1\\
				&\quad -\hat{\Sigma}_1^\dagger\big[\boldsymbol{B}_1^\top\mathbb{E}\bar{\eta}_1^{u_2}+\boldsymbol{D}_1^\top\mathbb{E}\bar{\zeta}_1^{u_2}
				+\boldsymbol{D}_1^\top P_1\mathbb{E}\sigma+\mathbb{E}\rho_1^1+(\boldsymbol{R}^1_{21}+\boldsymbol{D}_1^\top P_1\boldsymbol{D}_2)\mathbb{E}u_2\big],\, a.e.,\,\mathbb{P}\mbox{-}a.s.,
			\end{aligned}
		\end{equation}
		for some $\big(v_1(\cdot),\hat{v}_1(\cdot)\big)\in L^2_{\mathbb{F}}(t,T;\mathbb{R}^{m_1}) \times L^2(t,T;\mathbb{R}^{m_1}) $.
		
		\begin{Remark}
			From the above derivation process, it can be seen that the closed-loop optimal strategy of the follower is characterized by two coupled Riccati equations and a linear MF-BSDE, which is different from \cite{LSY2016} where the existence of the closed-loop optimal strategy is characterized by the solutions of two coupled Riccati equations, together with the adapted solution to a linear BSDE and a linear backward ordinary differential equation (BODE for short). It is worth noting that the follower's closed-loop optimal strategy should be unavoidably taken into account, when solving the optimization problem of the leader. So, these equations used to characterize follower's closed-loop optimal strategy can be regarded as a part of the state equation in leader's problem. The method in \cite{LSY2016} does not apply to our problem.
		\end{Remark}
		
		\begin{Remark}
			Riccati equations (\ref{Riccati---1}), (\ref{Riccati---2}) are the same as (4.5), (4.6) in Yong \cite{Yong2013}, respectively (See also Li et al. \cite{LSY2016}, Lin et al. \cite{LJZ2019}). The solvability of them are
			standard.
		\end{Remark}
		
		The following result characterizes the closed-loop solvability of Problem (MF-SLQ)$_f$.
		\begin{mythm}\label{Th-cl-f}
			Let (H1)-(H2) hold. For any $u_2(\cdot) \in \mathcal{U}_2[t,T]$, Problem (MF-SLQ)$_f$ admits a closed-loop optimal strategy on $[t,T]$ if and only if two Riccati equations (\ref{Riccati---1}) and (\ref{Riccati---2}) admit solutions $(P_1(\cdot),\Pi_1(\cdot)) \in C([t,T],\mathbb{S}^n) \times C([t,T],\mathbb{S}^n)$, which satisfy (\ref{amage assumption---follower}), (\ref{space assumption---follower}) and (\ref{assumption---follower})
				and MF-BSDE (\ref{follower BSDE}) admits an adapted solution $(\bar{\eta}_1^{u_2}(\cdot),\bar{\zeta}_1^{u_2}(\cdot)) \in L^2_{\mathbb{F}}(t,T;\mathbb{R}^n) \times L^2_{\mathbb{F}}(t,T;\mathbb{R}^n)$, which satisfies (\ref{amage assumption---follower---}) and (\ref{space assumption---follower---}). In this case, the closed-loop optimal strategy $(\bar{\Theta}_1(\cdot),\bar{\hat{\Theta}}_1(\cdot),\bar{v}_1(\cdot))$ of Problem (MF-SLQ)$_f$ admits the representation of (\ref{follower closed-loop strategy Theta1 hat{Theta}1}), (\ref{follower closed-loop optimal}).
				Further, the value function of the follower is given by
				\begin{equation}\label{value-1}
					\hspace{-3mm}\begin{aligned}
						&V_1(t,\xi;u_2(\cdot))=\mathbb{E}\biggl\{ \big\langle P_1(t)(\xi-\mathbb{E}\xi),\xi-\mathbb{E}\xi\big\rangle+\big\langle \Pi_1(t)\mathbb{E}\xi,\mathbb{E}\xi\big\rangle\\
						&\ +2\big\langle \bar{\eta}_1^{u_2}(t)-\mathbb{E}\bar{\eta}_1^{u_2}(t),\xi-\mathbb{E}\xi\big\rangle+2\big\langle \mathbb{E}\bar{\eta}_1^{u_2}(t),\mathbb{E}\xi\big\rangle\\
						&\ +\int_t^T\Big[\big\langle (R^1_{22}+D^\top_2P_1D_2)(u_2-\mathbb{E}u_2),u_2-\mathbb{E}u_2\big\rangle+\big\langle(\boldsymbol{R}^1_{22}+\boldsymbol{D}_2^\top P_1\boldsymbol{D}_2)\mathbb{E}u_2,\mathbb{E}u_2\big\rangle\\
						&\ +2\big\langle B^\top_2(\bar{\eta}_1^{u_2}-\mathbb{E}\bar{\eta}_1^{u_2})+D^\top_2(\bar{\zeta}_1^{u_2}-\mathbb{E}\bar{\zeta}_1^{u_2})
						+D^\top_2P_1(\sigma-\mathbb{E}\sigma)+\rho^1_2-\mathbb{E}\rho^1_2,u_2-\mathbb{E}u_2 \big\rangle\\
						&\ +2\big\langle \boldsymbol{B}_2^\top\mathbb{E}\bar{\eta}_1^{u_2}+\boldsymbol{D}_2^\top\mathbb{E}\bar{\zeta}_1^{u_2}+\boldsymbol{D}_2^\top P_1\mathbb{E}\sigma+\mathbb{E}\rho^1_2,\mathbb{E}u_2 \big\rangle
						+2\big\langle \bar{\eta}_1^{u_2}-\mathbb{E}\bar{\eta}_1^{u_2},b-\mathbb{E}b\big\rangle+2\big\langle \mathbb{E}\bar{\eta}_1^{u_2},\mathbb{E}b\big\rangle\\
						&\ +2\big\langle \bar{\zeta}_1^{u_2},\sigma \big\rangle
						+\big\langle P_1\sigma,\sigma\big\rangle-\big\langle \Sigma_1(\bar{v}_1-\mathbb{E}\bar{v}_1),\bar{v}_1-\mathbb{E}\bar{v}_1\big\rangle
						-\big\langle \hat{\Sigma}_1\mathbb{E}\bar{v}_1,\mathbb{E}\bar{v}_1\big\rangle\Big]ds \biggr\}.
					\end{aligned}
				\end{equation}
			\end{mythm}
			
			\begin{proof}
				Since the proof of necessity has already been given in the above derivation, now we only need to prove the sufficiency. For any given $u_2(\cdot) \in \mathcal{U}_2[t,T]$, we take $u_1(\cdot) \in \mathcal{U}_1[t,T]$, and let $x(\cdot)\equiv x(\cdot;t,\xi,u_1(\cdot),u_2(\cdot)),\,\bar{x}^{u_2}(\cdot)\equiv x(\cdot;t,\xi,\bar{\Theta}_1(\cdot),\bar{\hat{\Theta}}_1(\cdot),\bar{v}_1(\cdot),u_2(\cdot)) $ be the corresponding state processes. Applying It\^o's formula to $\big\langle P_1(x-\mathbb{E}x),x-\mathbb{E}x \big\rangle$, $\big\langle \bar{\eta}_1^{u_2}-\mathbb{E}\bar{\eta}_1^{u_2},x-\mathbb{E}x \big\rangle$,  $\big\langle \Pi_1\mathbb{E}x,\mathbb{E}x \big\rangle$ and $\big\langle \mathbb{E}\bar{\eta}_1^{u_2},\mathbb{E}x \big\rangle$ respectively, and plugging them into follower's cost functional, we have
				\begin{equation*}
					\begin{aligned}
						&J_1(t,\xi;u_1(\cdot),u_2(\cdot))=\mathbb{E}\biggl\{ \big\langle P_1(t)(\xi-\mathbb{E}\xi),\xi-\mathbb{E}\xi\big\rangle+\big\langle \Pi_1(t)\mathbb{E}\xi,\mathbb{E}\xi\big\rangle
						+2\big\langle \bar{\eta}_1^{u_2}(t)-\mathbb{E}\bar{\eta}_1^{u_2}(t),\xi-\mathbb{E}\xi \big\rangle\\
						&\quad +2\big\langle \mathbb{E}\bar{\eta}_1^{u_2}(t),\mathbb{E}\xi \big\rangle +\int_t^T\Big[\big\langle (R^1_{22}+D^\top_2P_1D_2)(u_2-\mathbb{E}u_2),u_2-\mathbb{E}u_2\big\rangle\\
						&\quad +\big\langle (\boldsymbol{R}^1_{22}+\boldsymbol{D}_2^\top P_1\boldsymbol{D}_2)\mathbb{E}u_2,\mathbb{E}u_2\big\rangle
						+2\big\langle B^\top_2(\bar{\eta}_1^{u_2}-\mathbb{E}\bar{\eta}_1^{u_2})+D^\top_2(\bar{\zeta}_1^{u_2}-\mathbb{E}\bar{\zeta}_1^{u_2})\\
						&\quad +D^\top_2P_1(\sigma-\mathbb{E}\sigma)+\rho^1_2-\mathbb{E}\rho^1_2,u_2-\mathbb{E}u_2 \big\rangle+2\big\langle \boldsymbol{B}_2^\top\mathbb{E}\bar{\eta}_1^{u_2}
						+\boldsymbol{D}_2^\top\mathbb{E}\bar{\zeta}_1^{u_2}+\boldsymbol{D}_2^\top P_1\mathbb{E}\sigma+\mathbb{E}\rho^1_2,\mathbb{E}u_2 \big\rangle\\
					\end{aligned}
				\end{equation*}
				\begin{equation*}
					\begin{aligned}	&\quad +2\big\langle \bar{\eta}_1^{u_2}-\mathbb{E}\bar{\eta}_1^{u_2},b-\mathbb{E}b\big\rangle
						+2\big\langle \mathbb{E}\bar{\eta}_1^{u_2},\mathbb{E}b\big\rangle+2\langle \bar{\zeta}_1^{u_2},\sigma \rangle+\big\langle P_1\sigma,\sigma\big\rangle\\
						&\quad -\big\langle \Sigma_1(\bar{v}_1-\mathbb{E}\bar{v}_1),\bar{v}_1-\mathbb{E}\bar{v}_1\big\rangle-\big\langle \hat{\Sigma}_1\mathbb{E}\bar{v}_1,\mathbb{E}\bar{v}_1\big\rangle\Big]ds \biggr\}\\
						&\quad +\mathbb{E}\int_t^T\Big[\big\langle \Sigma_1\big[u_1-\mathbb{E}u_1-\bar{\Theta}_1(x-\mathbb{E}x)-(\bar{v}_1-\mathbb{E}\bar{v}_1)\big],
						u_1-\mathbb{E}u_1-\bar{\Theta}_1(x-\mathbb{E}x)-(\bar{v}_1-\mathbb{E}\bar{v}_1)\big\rangle\\
						&\qquad\qquad\quad  +\big\langle \hat{\Sigma}_1(\mathbb{E}u_1-\bar{\hat{\Theta}}_1\mathbb{E}x-\mathbb{E}\bar{v}_1),\mathbb{E}u_1-\bar{\hat{\Theta}}_1\mathbb{E}x-\mathbb{E}\bar{v}_1\big\rangle\Big]ds\\
						&=J_1\big(t,\xi;\bar{\Theta}_1(\bar{x}^{u_2}-\mathbb{E}\bar{x}^{u_2})+\bar{\hat{\Theta}}_1\mathbb{E}\bar{x}^{u_2}+\bar{v}_1,u_2(\cdot)\big)\\
						&\quad +\mathbb{E}\int_t^T\Big[\big\langle \Sigma_1\big[u_1-\mathbb{E}u_1-\bar{\Theta}_1(x-\mathbb{E}x)-(\bar{v}_1-\mathbb{E}\bar{v}_1)\big],u_1-\mathbb{E}u_1
						-\bar{\Theta}_1\big(x-\mathbb{E}x\big)-\big(\bar{v}_1-\mathbb{E}\bar{v}_1\big) \big\rangle\\
						&\qquad \qquad +\big\langle \hat{\Sigma}_1\big(\mathbb{E}u_1-\bar{\hat{\Theta}}_1\mathbb{E}x-\mathbb{E}\bar{v}_1\big),\mathbb{E}u_1-\bar{\hat{\Theta}}_1\mathbb{E}x-\mathbb{E}\bar{v}_1\big\rangle\Big]ds.
					\end{aligned}
				\end{equation*}
				Hence, if $\Sigma_1 \geq 0,\ \hat{\Sigma}_1 \geq 0,\ a.e.$, we can see that
				\begin{equation*}
					J_1\big(t,\xi;\bar{\Theta}_1(\bar{x}^{u_2}-\mathbb{E}\bar{x}^{u_2})+\bar{\hat{\Theta}}_1\mathbb{E}\bar{x}^{u_2}+\bar{v}_1,u_2(\cdot)\big) \leq J_1(t,\xi;u_1(\cdot),u_2(\cdot)),\quad \forall u_1(\cdot) \in \mathcal{U}_1[t,T].
				\end{equation*}
				In this case, (\ref{value-1}) holds.
			\end{proof}
			
			For brevity, we assume that $\Sigma_1^{-1}$ and $\hat{\Sigma}_1^{-1}$ exist in the following. If $\Sigma_1$ and $\hat{\Sigma}_1$ are not invertible, the concept of generalized inverse could be introduced as before, and the expression becomes more complicated, but there is no other difficulty. Thus, for given $u_2(\cdot) \in \mathcal{U}_2[t,T]$, by (\ref{follower closed-loop strategy Theta1 hat{Theta}1}) and (\ref{follower closed-loop optimal}), the follower takes his/her optimal strategy:
			\begin{equation}\label{follower optimal strategy-invertible}
				\left\{\begin{aligned}
					\bar{\Theta}_1&=-\Sigma_1^{-1}(B^\top_1P_1+D^\top_1P_1C+S^1_1),\,\,\,\,\bar{\hat{\Theta}}_1=-\hat{\Sigma}_1^{-1}(\boldsymbol{B}_1^\top\Pi_1+\boldsymbol{D}_1^\top P_1\boldsymbol{C}+\boldsymbol{S}^1_1),\\
					\bar{v}_1&=-\hat{\Sigma}_1^{-1}\big[\boldsymbol{B}_1^\top\mathbb{E}\bar{\eta}_1^{u_2}+\boldsymbol{D}_1^\top\mathbb{E}\bar{\zeta}_1^{u_2}
					+\boldsymbol{D}_1^\top P_1\mathbb{E}\sigma+\mathbb{E}\rho_1^1+(\boldsymbol{R}^1_{21}+\boldsymbol{D}_1^\top P_1\boldsymbol{D}_2)\mathbb{E}u_2\big]\\
					&\quad-\Sigma_1^{-1}\big[B_1^\top(\bar{\eta}_1^{u_2}-\mathbb{E}\bar{\eta}_1^{u_2})+D^\top_1(\bar{\zeta}_1^{u_2}-\mathbb{E}\bar{\zeta}_1^{u_2})+D_1^\top P_1(\sigma-\mathbb{E}\sigma)+\rho_1^1-\mathbb{E}\rho_1^1\\
					&\qquad\qquad +(R^1_{21}+D_1^\top P_1D_2\big)\big(u_2-\mathbb{E}u_2)\big],\quad a.e.,\,\mathbb{P}\mbox{-}a.s..
				\end{aligned}\right.
			\end{equation}
			
			\section{Optimization problem of the leader}\label{leader's problem}
			
			The process triple $(\bar{x}^{u_2}(\cdot),\bar{\eta}_1^{u_2}(\cdot),\bar{\zeta}_1^{u_2}(\cdot))\in L^2_{\mathbb{F}}(t,T;\mathbb{R}^n)\times L^2_{\mathbb{F}}(t,T;\mathbb{R}^n)\times L^2_{\mathbb{F}}(t,T;\mathbb{R}^n)$ satisfies the following MF-FBSDE, which, now, is the ``state" equation of the leader:
			\begin{equation}\label{leader state}
				\left\{\begin{aligned}
					d\bar{x}^{u_2}(s)&=\big\{ \tilde{A}(\bar{x}^{u_2}-\mathbb{E}\bar{x}^{u_2})+\check{A}\mathbb{E}\bar{x}^{u_2}+\tilde{M}(\bar{\eta}_1^{u_2}-\mathbb{E}\bar{\eta}_1^{u_2})
					+\check{M}\mathbb{E}\bar{\eta}_1^{u_2} \\
					&\qquad +\tilde{F}(\bar{\zeta}_1^{u_2}-\mathbb{E}\bar{\zeta}_1^{u_2})+\check{F}\mathbb{E}\bar{\zeta}_1^{u_2}+\tilde{B}(u_2-\mathbb{E}u_2)+\check{B}\mathbb{E}u_2+\tilde{b}\big\}ds\\
					&\quad +\big\{ \tilde{C}(\bar{x}^{u_2}-\mathbb{E}\bar{x}^{u_2})+\check{C}\mathbb{E}\bar{x}^{u_2}+\tilde{F}^\top(\bar{\eta}_1^{u_2}-\mathbb{E}\bar{\eta}_1^{u_2})
					+\check{F}^\top\mathbb{E}\bar{\eta}_1^{u_2} \\
					&\qquad+\tilde{K}(\bar{\zeta}_1^{u_2}-\mathbb{E}\bar{\zeta}_1^{u_2})+\check{K}\mathbb{E}\bar{\zeta}_1^{u_2}+\tilde{D}(u_2-\mathbb{E}u_2)
					+\check{D}\mathbb{E}u_2+\tilde{\sigma}\big\}dW(s),\\
					-d\bar{\eta}_1^{u_2}(s)&=\big\{\tilde{A}^\top(\bar{\eta}_1^{u_2}-\mathbb{E}\bar{\eta}_1^{u_2})+\check{A}^\top\mathbb{E}\bar{\eta}_1^{u_2}+\tilde{C}^\top(\bar{\zeta}_1^{u_2}
					-\mathbb{E}\bar{\zeta}_1^{u_2})+\check{C}^\top\mathbb{E}\bar{\zeta}_1^{u_2}\\
					&\qquad +\tilde{N}(u_2-\mathbb{E}u_2)+\check{N}\mathbb{E}u_2+\tilde{f}\big\}ds-\bar{\zeta}_1^{u_2}dW(s),\quad s\in[t,T],\\
					\bar{x}^{u_2}(t)&=\xi,\,\,\,\,\,\bar{\eta}_1^{u_2}(T)=\boldsymbol{g}^1,
				\end{aligned}\right.
			\end{equation}
			and leader's cost functional is
			\begin{equation}
				\begin{aligned}
					&\hat{J}_2(t,\xi;u_2(\cdot)):=J_2(t,\xi;\bar{u}_1(\cdot),u_2(\cdot))=\mathbb{E}\Bigg\{\big\langle G^2\big(\bar{x}^{u_2}(T)-\mathbb{E}\bar{x}^{u_2}(T)\big),\bar{x}^{u_2}(T)-\mathbb{E}\bar{x}^{u_2}(T) \big\rangle\\
					&\quad +\big\langle \boldsymbol{G}^2\mathbb{E}\bar{x}^{u_2}(T),\mathbb{E}\bar{x}^{u_2}(T)\big\rangle
					+2\big\langle g^2,\bar{x}^{u_2}(T)-\mathbb{E}\bar{x}^{u_2}(T) \big\rangle+2\big\langle \boldsymbol{g}^2,\mathbb{E}\bar{x}^{u_2}(T)\big\rangle\\
					&\quad +\int_t^T \bigg[\bigg\langle
					\left( \begin{array}{cccc}
						\tilde{Q}_{11}  & \tilde{Q}_{12}^\top  & \tilde{Q}_{13}^\top  & \tilde{S}^\top_1 \\
						\tilde{Q}_{12}  & \tilde{Q}_{22}       & \tilde{Q}_{23}^\top  & \tilde{S}^\top_2 \\
						\tilde{Q}_{13}  & \tilde{Q}_{23}       & \tilde{Q}_{33}       & \tilde{S}^\top_3 \\
						\tilde{S}_1     & \tilde{S}_2          & \tilde{S}_3          & \tilde{R}       \\
					\end{array} \right)
					\left( \begin{array}{c} \bar{x}^{u_2}-\mathbb{E}\bar{x}^{u_2} \\ \bar{\eta}^{u_2}_1-\mathbb{E}\bar{\eta}^{u_2}_1 \\ \bar{\zeta}^{u_2}_1-\mathbb{E}\bar{\zeta}^{u_2}_1 \\ u_2-\mathbb{E}u_2 \end{array} \right),
					\left( \begin{array}{c} \bar{x}^{u_2}-\mathbb{E}\bar{x}^{u_2} \\ \bar{\eta}^{u_2}_1-\mathbb{E}\bar{\eta}^{u_2}_1 \\ \bar{\zeta}^{u_2}_1-\mathbb{E}\bar{\zeta}^{u_2}_1 \\ u_2-\mathbb{E}u_2 \end{array}\right)\bigg\rangle\\
					&\qquad\qquad +2\bigg\langle
					\left( \begin{array}{c} \tilde{q}_1 \\ \tilde{q}_2 \\ \tilde{q}_3 \\ \tilde{\rho} \end{array} \right),
					\left( \begin{array}{c} \bar{x}^{u_2}-\mathbb{E}\bar{x}^{u_2} \\ \bar{\eta}^{u_2}_1-\mathbb{E}\bar{\eta}^{u_2}_1 \\ \bar{\zeta}^{u_2}_1-\mathbb{E}\bar{\zeta}^{u_2}_1 \\u_2-\mathbb{E}u_2 \end{array}\right)\bigg\rangle
					+2\bigg\langle
					\left( \begin{array}{c} \check{q}_1 \\ \check{q}_2 \\ \check{q}_3 \\ \check{\rho} \end{array} \right),
					\left( \begin{array}{c} \mathbb{E}\bar{x}^{u_2} \\ \mathbb{E}\bar{\eta}^{u_2}_1 \\ \mathbb{E}\bar{\zeta}^{u_2}_1 \\\mathbb{E}u_2 \end{array}\right)\bigg\rangle\\
					&\qquad\qquad +\bigg\langle
					\left( \begin{array}{cccc}
						\check{Q}_{11}  & \check{Q}_{12}^\top  & \check{Q}_{13}^\top  & \check{S}^\top_1 \\
						\check{Q}_{12}  & \check{Q}_{22}       & \check{Q}_{23}^\top  & \check{S}^\top_2 \\
						\check{Q}_{13}  & \check{Q}_{23}       & \check{Q}_{33}       & \check{S}^\top_3 \\
						\check{S}_1     & \check{S}_2          & \check{S}_3          & \check{R}        \\
					\end{array} \right)
					\left( \begin{array}{c} \mathbb{E}\bar{x}^{u_2} \\ \mathbb{E}\bar{\eta}^{u_2}_1 \\ \mathbb{E}\bar{\zeta}^{u_2}_1 \\\mathbb{E}u_2          \end{array} \right),
					\left( \begin{array}{c} \mathbb{E}\bar{x}^{u_2} \\ \mathbb{E}\bar{\eta}^{u_2}_1 \\ \mathbb{E}\bar{\zeta}^{u_2}_1 \\ \mathbb{E}u_2 \end{array}\right)\bigg\rangle \bigg]ds+\mathcal L\Bigg\},
				\end{aligned}
			\end{equation}
			where we denote
			\begin{equation*}
				\begin{aligned}
					&\acute{R}:=R^2_{11}+\hat{R}^2_{11},\,\,\, \acute{S}:=S^2_1+\hat{S}^2_1,\,\,\, \acute{Q}:=Q^2+\hat{Q}^2,\,\,\,
					\breve{R}:=R^2_{22}+\hat{R}^2_{22},\,\,\,\acute{L}:=R^1_{12}+D^\top_2P_1D_1,\,\,\,\\
					&\grave{R}:=R^2_{12}+\hat{R}^2_{12},\,\,\, \grave{S}:=S^2_2+\hat{S}^2_2,\,\,\,
					\acute{\rho}:=D^\top_1P_1(\sigma-\mathbb{E}\sigma)+\rho_1^1-\mathbb{E}\rho_1^1,\,\,\,
					\grave{\rho}:=\boldsymbol{D}_1^\top P_1\mathbb{E}\sigma+\mathbb{E}\rho_1^1,\\
					&\acute{J}:=\boldsymbol{R}^1_{12}+\boldsymbol{D}_2^\top P_1\boldsymbol{D}_1,\,\,\,\grave{L}:=R^1_{21}+D^\top_1P_1D_2,
					\,\,\,\grave{J}:= \boldsymbol{R}^1_{21}+\boldsymbol{D}_1^\top P_1\boldsymbol{D}_2,\\
					&\tilde{A}:=A+B_1\bar{\Theta}_1,\,\,\,\check{A}:=\boldsymbol{A}+\boldsymbol{B}_1\bar{\hat{\Theta}}_1,\,\,\,
					\tilde{C}:=C+D_1\bar{\Theta}_1,\,\,\,\check{C}:=\boldsymbol{C}+\boldsymbol{D}_1\bar{\hat{\Theta}}_1,\\
					&\tilde{M}:=-B_1\Sigma_1^{-1}B_1^\top,\,\,\,\check{M}:=-\boldsymbol{B}_1\hat{\Sigma}^{-1}_1\boldsymbol{B}_1^\top,\,\,\,
					\tilde{F}:=-B_1\Sigma_1^{-1}D_1^\top,\,\,\,\check{F}:=-\boldsymbol{B}_1\hat{\Sigma}^{-1}_1\boldsymbol{D}_1^\top,\\
					&\tilde{K}:=-D_1\Sigma_1^{-1}D^\top_1,\,\,\, \check{K}:=-\boldsymbol{D}_1\hat{\Sigma}^{-1}_1\boldsymbol{D}_1^\top,\,\,\,
					\tilde{B}:=B_2-B_1\Sigma_1^{-1}\grave{L},\,\,\,
					\check{B}:=\boldsymbol{B}_2-\boldsymbol{B}_1\hat{\Sigma}^{-1}_1\grave{J},\\
					&\tilde{D}:=D_2-D_1\Sigma_1^{-1}\grave{L},\,\,\,
					\check{D}:=\boldsymbol{D}_2-\boldsymbol{D}_1\hat{\Sigma}^{-1}_1\grave{J},\,\,\,
					\check{N}:=\Pi_1\boldsymbol{B}_2+\boldsymbol{C}^\top P_1\boldsymbol{D}_2+{\boldsymbol{S}^1_2}^\top+\bar{\hat{\Theta}}^\top_1\grave{J},\\
					&\tilde{N}:=P_1B_2+C^\top P_1D_2+S^{1\top}_2+\bar{\Theta}^\top_1\grave{L},\\
					&\tilde{b}:=b-B_1\Sigma_1^{-1}\acute{\rho}-\boldsymbol{B}_1\hat{\Sigma}^{-1}_1\grave{\rho},\,\,\,
					\tilde{\sigma}:=\sigma-D_1\Sigma_1^{-1}\acute{\rho}-\boldsymbol{D}_1\hat{\Sigma}^{-1}_1\grave{\rho},\\
					&\tilde{f}:=P_1\big(b-\mathbb{E}b\big)+\Pi_1\mathbb{E}b+\tilde{C}^\top P_1\big(\sigma-\mathbb{E}\sigma\big)
					+\check{C}^\top P_1\mathbb{E}\sigma+\bar{\Theta}^\top_1\big(\rho_1^1-\mathbb{E}\rho_1^1\big)+\bar{\hat{\Theta}}^\top_1\mathbb{E}\rho_1^1+q^1,\\
					&\tilde{Q}_{11}:=Q^2+S^{2\top}_1\bar{\Theta}_1+\bar{\Theta}^\top_1S^2_1+\bar{\Theta}^\top_1R^2_{11}\bar{\Theta}_1,\,\,\,
					\tilde{Q}_{12}:=-B_1\Sigma_1^{-1}\big(S^2_1+R^2_{11}\bar{\Theta}_1\big),\\
					&\check{Q}_{22}:=\boldsymbol{B}_1\hat{\Sigma}^{-1}_1\acute{R}\hat{\Sigma}^{-1}_1\boldsymbol{B}_1^\top,\,\,\,
					\tilde{Q}_{13}:=-D_1\Sigma_1^{-1}\big(S^2_1+R^2_{11}\bar{\Theta}_1\big),\,\,\,
					\tilde{Q}_{22}:=B_1\Sigma_1^{-1}R^2_{11}\Sigma_1^{-1}B^\top_1, \\
					&\tilde{Q}_{23}:=D_1\Sigma_1^{-1}R^2_{11}\Sigma_1^{-1}B^\top_1,\,\,\,
					\tilde{Q}_{33}:=D_1\Sigma_1^{-1}R^2_{11}\Sigma_1^{-1}D^\top_1,\,\,\,
					\check{Q}_{11}:=\acute{Q}+\acute{S}^\top\bar{\hat{\Theta}}_1+\bar{\hat{\Theta}}^\top_1\acute{S}+\bar{\hat{\Theta}}^\top_1\acute{R}\bar{\hat{\Theta}}_1,\\
					&\check{Q}_{23}:=\boldsymbol{D}_1\hat{\Sigma}^{-1}_1\acute{R}\hat{\Sigma}^{-1}_1\boldsymbol{B}_1^\top,\,\,\,
					\check{Q}_{12}:=-\boldsymbol{B}_1\hat{\Sigma}^{-1}_1\big(\acute{S}+\acute{R}\bar{\hat{\Theta}}_1\big),\,\,\,
					\check{Q}_{13}:=-\boldsymbol{D}_1\hat{\Sigma}^{-1}_1\big(\acute{S}+\acute{R}\bar{\hat{\Theta}}_1\big),\\
					&\check{Q}_{33}:=\boldsymbol{D}_1\hat{\Sigma}^{-1}_1\acute{R}\hat{\Sigma}^{-1}_1\boldsymbol{D}_1^\top,\,\,\,
					\tilde{S}_1:=S^2_2+R^2_{12}\bar{\Theta}_1-\acute{L}\Sigma_1^{-1}\big(S^2_1+R^2_{11}\bar{\Theta}_1\big),\\
					&\tilde{S}_2:=\big(\acute{L}\Sigma_1^{-1}R^2_{11}-R^2_{12}\big)\Sigma_1^{-1}B^\top_1,\,\,\,
					\tilde{S}_3:=\big(\acute{L}\Sigma_1^{-1}R^2_{11}-R^2_{12}\big)\Sigma_1^{-1}D^\top_1,\\
				\end{aligned}
			\end{equation*}
			\begin{equation}\label{coefficients of leader}
				\begin{aligned}
					&\tilde{R}:=R^2_{22}-\acute{L}\Sigma_1^{-1}R^2_{21}-R^2_{12}\Sigma_1^{-1}\acute{L}^\top+\acute{L}\Sigma_1^{-1}R^2_{11}\Sigma_1^{-1}\acute{L}^\top,\\
					&\tilde{q}_1:=q^2+\bar{\Theta}^\top_1\big(\rho^2_1-\mathbb{E}\rho^2_1\big)-\big(S^2_1+R^2_{11}\bar{\Theta}_1\big)^\top\Sigma_1^{-1}\acute{\rho},\\
					&\tilde{q}_2:=B_1\Sigma_1^{-1}\bigl[R^2_{11}\Sigma_1^{-1}\acute{\rho}-\big(\rho_1^2-\mathbb{E}\rho_1^2\big)\bigr],\,\,\,
					\tilde{q}_3:=D_1\Sigma_1^{-1}\bigl[R^2_{11}\Sigma_1^{-1}\acute{\rho}-\big(\rho_1^2-\mathbb{E}\rho_1^2\big)\bigr],\\
					&\tilde{\rho}:=\big(\acute{L}\Sigma_1^{-1}R^2_{11}-R^2_{12}\big)\Sigma_1^{-1}\acute{\rho}+\rho^2_2-\mathbb{E}\rho^2_2-\acute{L}\Sigma_1^{-1}\big(\rho_1^2-\mathbb{E}\rho^2_1\big),\\
					&\check{S}_1:=\grave{S}+\grave{R}\bar{\hat{\Theta}}_1-\acute{J}\hat{\Sigma}_1^{-1}\big(\acute{S}+\acute{R}\bar{\hat{\Theta}}_1\big),\,\,\,
					\check{S}_2:=\big(\acute{J}\hat{\Sigma}^{-1}_1\acute{R}-\grave{R}\big)\hat{\Sigma}^{-1}_1\boldsymbol{B}_1^\top,\\
					&\check{S}_3:=\big(\acute{J}\hat{\Sigma}^{-1}_1\acute{R}-\grave{R}\big)\hat{\Sigma}^{-1}_1\boldsymbol{D}_1^\top,\,\,\,
					\check{R}:=\breve{R}-\grave{R}\hat{\Sigma}^{-1}_1\acute{J}^\top-\acute{J}\hat{\Sigma}^{-1}_1\grave{R}^\top
					+\acute{J}\hat{\Sigma}^{-1}_1\acute{R}\hat{\Sigma}^{-1}_1\acute{J}^\top,\\
					&\check{q}_1:=\mathbb{E}q^2+\bar{\hat{\Theta}}^\top_1\mathbb{E}\rho_1^2-\big(\acute{S}+\acute{R}\bar{\hat{\Theta}}_1\big)^\top\hat{\Sigma}_1^{-1}\grave{\rho},\,\,\,
					\check{q}_2:=\boldsymbol{B}_1\hat{\Sigma}^{-1}_1\big[\acute{R}\hat{\Sigma}_1^{-1}\grave{\rho}-\mathbb{E}\rho_1^2\big],\\
					&\check{q}_3:=\boldsymbol{D}_1\hat{\Sigma}^{-1}_1\big[\acute{R}\hat{\Sigma}_1^{-1}\grave{\rho}-\mathbb{E}\rho_1^2\big],\,\,\,
					\check{\rho}:=\big(\acute{J}\hat{\Sigma}^{-1}_1\acute{R}-\grave{R}\big)\hat{\Sigma}^{-1}_1\grave{\rho}+\mathbb{E}\rho^2_2-\acute{J}\hat{\Sigma}^{-1}_1\mathbb{E}\rho_1^2,\\
					&\mathcal L:=\int_t^T\Big[\big\langle R^2_{11}\Sigma_1^{-1}\acute{\rho},\Sigma_1^{-1}\acute{\rho} \big\rangle-2\big\langle \rho_1^2-\mathbb{E}\rho_1^2,\Sigma_1^{-1}\acute{\rho}\big\rangle
					-2\big\langle \mathbb{E}\rho_1^2,\hat{\Sigma}^{-1}_1\grave{\rho}\big\rangle+\big\langle \acute{R}\hat{\Sigma}^{-1}_1\grave{\rho},\hat{\Sigma}^{-1}_1\grave{\rho} \big\rangle\Big]ds.
				\end{aligned}
			\end{equation}
			
			The mean-field type SLQ problem of the leader can be stated as follows.
			
			\textbf{Problem (MF-SLQ)$_l$}. For any given initial pair $(t,\xi) \in [0,T] \times L^2_{\mathcal{F}_t}(\Omega;\mathbb{R}^n)$, find a $\bar{u}_2(\cdot) \in \mathcal{U}_2[t,T]$ such that
			\begin{equation}\label{LLQ}
				\hat{J}_2(t,\xi;\bar{u}_2(\cdot))=\underset{u_2(\cdot)\in \mathcal{U}_2[t,T]} {\min}\hat{J}_2(t,\xi;u_2(\cdot))\equiv V_2(t,\xi).
			\end{equation}
			Problem (MF-SLQ)$_l$ is a SLQ optimal control problem of MF-FBSDE systems. If there exists a (unique) $\bar{u}_2(\cdot) \in \mathcal{U}_2[t,T]$ satisfying (\ref{LLQ}), $\bar{u}_2(\cdot)$ is called an (unique) open-loop optimal control of Problem (MF-SLQ)$_l$ for $(t,\xi)$, the corresponding $(\bar{x}(\cdot),\bar{\eta}_1(\cdot),\bar{\zeta}_1(\cdot))\equiv (\bar{x}^{\bar{u}_2}(\cdot),\bar{\eta}^{\bar{u}_2}_1(\cdot),\bar{\zeta}^{\bar{u}_2}_1(\cdot))$ is called an open-loop optimal state process triple. The map $V_2(\cdot,\cdot)$ is called the value function of Problem (MF-SLQ)$_l$. In this time, Problem (MF-SLQ)$_l$ is called (uniquely) open-loop solvable.
			
			In particular, we denote leader's problem as \textbf{Problem (MF-SLQ)$_{l0}$} and the cost functional as $\hat{J}_2^0(t,\xi;u_2(\cdot))$, when $\tilde{b}=\tilde{\sigma}=\tilde{f}=g^2=\hat{g}^2=\tilde{q}_j=\check{q}_j=\tilde{\rho}=\check{\rho}=0$, for $j=1,2,3$.
			
			The following theorem gives sufficient and necessary conditions to the open-loop solvability of Problem (MF-SLQ)$_l$.
			\begin{mythm}\label{leader open-loop ns condition}
				Let (H1)-(H2) hold. For any given initial pair $(t,\xi) \in [0,T] \times L^2_{\mathcal{F}_t}(\Omega;\mathbb{R}^n)$, a quadruple $(\bar{x}(\cdot),\bar{\eta}_1(\cdot),\bar{\zeta}_1(\cdot),\bar{u}_2(\cdot))$ is an open-loop optimal quadruple of Problem (MF-SLQ)$_l$ if and only if the following stationarity condition holds:
				\begin{equation}\label{leader open stationarity}
					\begin{aligned}
						0&=\tilde{B}^\top(q_2^{\bar{u}_2}-\mathbb{E}q_2^{\bar{u}_2})+\check{B}^\top\mathbb{E}q^{\bar{u}_2}_2+\tilde{D}^\top(k^{\bar{u}_2}_2-\mathbb{E}k^{\bar{u}_2}_2)
						+\check{D}^\top\mathbb{E}k^{\bar{u}_2}_2+\tilde{N}^\top(p^{\bar{u}_2}_2-\mathbb{E}p^{\bar{u}_2}_2)\\	
						&\quad +\check{N}^\top\mathbb{E}p^{\bar{u}_2}_2+\tilde{S}_1(\bar{x}-\mathbb{E}\bar{x})+\check{S}_1\mathbb{E}\bar{x}+\tilde{S}_2(\bar{\eta}_1-\mathbb{E}\bar{\eta}_1)
						+\check{S}_2\mathbb{E}\bar{\eta}_1+\tilde{S}_3(\bar{\zeta}_1-\mathbb{E}\bar{\zeta}_1)\\
						&\quad +\check{S}_3\mathbb{E}\bar{\zeta}_1+\tilde{R}(\bar{u}_2-\mathbb{E}\bar{u}_2)+\check{R}\mathbb{E}\bar{u}_2+\tilde{\rho}-\mathbb{E}\tilde{\rho}+\check{\rho},\quad a.e.,\, \mathbb{P}\mbox{-}a.s.,
					\end{aligned}
				\end{equation}
				where $(p^{\bar{u}_2}_2(\cdot),q^{\bar{u}_2}_2(\cdot),k_2^{\bar{u}_2}(\cdot))\in L^2_{\mathbb{F}}(t,T;\mathbb{R}^n)\times L^2_{\mathbb{F}}(t,T;\mathbb{R}^n)\times L^2_{\mathbb{F}}(t,T;\mathbb{R}^n)$ is the adapted solution to the following MF-FBSDE:
				\begin{equation}\label{leader adjoint equation}
					\left\{\begin{aligned}
						dp^{\bar{u}_2}_2(s)&=\big\{ \tilde{A}(p^{\bar{u}_2}_2-\mathbb{E}p^{\bar{u}_2}_2)+\check{A}\mathbb{E}p^{\bar{u}_2}_2+\tilde{M}^\top(q^{\bar{u}_2}_2-\mathbb{E}q^{\bar{u}_2}_2)
						+\check{M}^\top\mathbb{E}q^{\bar{u}_2}_2+\tilde{F}(k^{\bar{u}_2}_2-\mathbb{E}k^{\bar{u}_2}_2)\\
						&\quad\ +\check{F}\mathbb{E}k^{\bar{u}_2}_2+\tilde{Q}_{12}(\bar{x}-\mathbb{E}\bar{x})+\check{Q}_{12}\mathbb{E}\bar{x}+\tilde{Q}_{22}(\bar{\eta}_1-\mathbb{E}\bar{\eta}_1)
						+\check{Q}_{22}\mathbb{E}\bar{\eta}_1\\
						&\quad\ +\tilde{Q}_{23}^\top(\bar{\zeta}_1-\mathbb{E}\bar{\zeta}_1)+\check{Q}^\top_{23}\mathbb{E}\bar{\zeta}_1+\tilde{S}^\top_2(\bar{u}_2-\mathbb{E}\bar{u}_2)
						+\check{S}^\top_2\mathbb{E}\bar{u}_2+\tilde{q}_2-\mathbb{E}\tilde{q}_2+\check{q}_2\big\}ds\\
						&\,\, +\big\{\tilde{C}(p^{\bar{u}_2}_2-\mathbb{E}p^{\bar{u}_2}_2)+\check{C}\mathbb{E}p^{\bar{u}_2}_2+\tilde{F}^\top(q^{\bar{u}_2}_2-\mathbb{E}q^{\bar{u}_2}_2)
						+\check{F}^\top\mathbb{E}q^{\bar{u}_2}_2+\tilde{K}^\top(k^{\bar{u}_2}_2-\mathbb{E}k^{\bar{u}_2}_2)\\	
						&\quad\ +\check{K}^\top\mathbb{E}k^{\bar{u}_2}_2+\tilde{Q}_{13}(\bar{x}-\mathbb{E}\bar{x})+\check{Q}_{13}\mathbb{E}\bar{x}+\tilde{Q}_{23}(\bar{\eta}_1-\mathbb{E}\bar{\eta}_1)
						+\check{Q}_{23}\mathbb{E}\bar{\eta}_1\\
						&\quad\ +\tilde{Q}_{33}(\bar{\zeta}_1-\mathbb{E}\bar{\zeta}_1)+\check{Q}_{33}\mathbb{E}\bar{\zeta}_1+\tilde{S}^\top_3(\bar{u}_2-\mathbb{E}\bar{u}_2)
						+\check{S}^\top_3\mathbb{E}\bar{u}_2+\tilde{q}_3-\mathbb{E}\tilde{q}_3+\check{q}_3\big\}dW(s),\\
						-dq^{\bar{u}_2}_2(s)&=\big\{ \tilde{A}^\top(q^{\bar{u}_2}_2-\mathbb{E}q^{\bar{u}_2}_2)+\check{A}^\top\mathbb{E}q^{\bar{u}_2}_2
						+\tilde{C}^\top(k^{\bar{u}_2}_2-\mathbb{E}k^{\bar{u}_2}_2)
						+\check{C}^\top\mathbb{E}k^{\bar{u}_2}_2+\tilde{Q}_{11}(\bar{x}-\mathbb{E}\bar{x})\\
						&\quad\ +\check{Q}_{11}\mathbb{E}\bar{x}+\tilde{Q}_{12}^\top(\bar{\eta}_1-\mathbb{E}\bar{\eta}_1)+\check{Q}_{12}^\top\mathbb{E}\bar{\eta}_1
						+\tilde{Q}_{13}^\top(\bar{\zeta}_1-\mathbb{E}\bar{\zeta}_1)\\
						&\quad\ +\check{Q}^\top_{13}\mathbb{E}\bar{\zeta}_1+\tilde{S}^\top_1(\bar{u}_2-\mathbb{E}\bar{u}_2)+\check{S}^\top_1\mathbb{E}\bar{u}_2+\tilde{q}_1-\mathbb{E}\tilde{q}_1
						+\check{q}_1\Big\}ds-k^{\bar{u}_2}_2dW(s),\\
						p^{\bar{u}_2}_2(t)&=0,\,\,\,q^{\bar{u}_2}_2(T)=G^2\big(\bar{x}(T)-\mathbb{E}\bar{x}(T)\big)+\boldsymbol{G}^2\mathbb{E}\bar{x}(T)+\boldsymbol{g}^2,
					\end{aligned}\right.
				\end{equation}
				and the following convexity condition holds:
				\begin{equation}\label{leader convexity condition}
					\begin{aligned}
						0&\leq \mathbb{E}\Bigg\{ \big\langle G^2\big(x_{0l}(T)-\mathbb{E}x_{0l}(T)\big),x_{0l}(T)-\mathbb{E}x_{0l}(T) \big\rangle +\big\langle \boldsymbol{G}^2\mathbb{E}x_{0l}(T),\mathbb{E}x_{0l}(T)\big\rangle\\
						&\qquad+\int_t^T \bigg[\bigg\langle
						\left( \begin{array}{cccc}
							\tilde{Q}_{11}  & \tilde{Q}_{12}^\top  & \tilde{Q}_{13}^\top  & \tilde{S}^\top_1 \\
							\tilde{Q}_{12}  & \tilde{Q}_{22}       & \tilde{Q}_{23}^\top  & \tilde{S}^\top_2 \\
							\tilde{Q}_{13}  & \tilde{Q}_{23}       & \tilde{Q}_{33}       & \tilde{S}^\top_3 \\
							\tilde{S}_1     & \tilde{S}_2          & \tilde{S}_3          & \tilde{R}       \\
						\end{array} \right)
						\left( \begin{array}{c} x_{0l}-\mathbb{E}x_{0l} \\ \eta_{0l}-\mathbb{E}\eta_{0l} \\ \zeta_{0l}-\mathbb{E}\zeta_{0l} \\ u_2-\mathbb{E}u_2          \end{array} \right),
						\left( \begin{array}{c} x_{0l}-\mathbb{E}x_{0l} \\ \eta_{0l}-\mathbb{E}\eta_{0l} \\ \zeta_{0l}-\mathbb{E}\zeta_{0l} \\ u_2-\mathbb{E}u_2 \end{array}\right)\bigg\rangle\\
						&\qquad\quad +\bigg\langle
						\left( \begin{array}{cccc}
							\check{Q}_{11}  & \check{Q}_{12}^\top  & \check{Q}_{13}^\top  & \check{S}^\top_1 \\
							\check{Q}_{12}  & \check{Q}_{22}       & \check{Q}_{23}^\top  & \check{S}^\top_2 \\
							\check{Q}_{13}  & \check{Q}_{23}       & \check{Q}_{33}       & \check{S}^\top_3 \\
							\check{S}_1     & \check{S}_2          & \check{S}_3          & \check{R}        \\
						\end{array} \right)
						\left( \begin{array}{c} \mathbb{E}x_{0l}\\ \mathbb{E}\eta_{0l} \\ \mathbb{E}\zeta_{0l} \\\mathbb{E}u_2          \end{array} \right),
						\left( \begin{array}{c} \mathbb{E}x_{0l} \\ \mathbb{E}\eta_{0l} \\ \mathbb{E}\zeta_{0l} \\ \mathbb{E}u_2 \end{array}\right)\bigg\rangle\bigg]ds\Bigg\} ,\quad \forall u_2 \in \mathcal{U}_2[t,T],
					\end{aligned}
				\end{equation}
				where $(x_{0l}(\cdot),\eta_{0l}(\cdot),\zeta_{0l}(\cdot))\in L^2_{\mathbb{F}}(t,T;\mathbb{R}^n)\times L^2_{\mathbb{F}}(t,T;\mathbb{R}^n)\times L^2_{\mathbb{F}}(t,T;\mathbb{R}^n)$ is the solution to the following equation:
				\begin{equation}\label{leader x_0l}
					\left\{\begin{aligned}
						dx_{0l}(s)&=\big\{ \tilde{A}(x_{0l}-\mathbb{E}x_{0l})+\check{A}\mathbb{E}x_{0l}+\tilde{M}(\eta_{0l}-\mathbb{E}\eta_{0l})+\check{M}\mathbb{E}\eta_{0l} \\
						&\qquad +\tilde{F}(\zeta_{0l}-\mathbb{E}\zeta_{0l})+\check{F}\mathbb{E}\zeta_{0l}+\tilde{B}(u_2-\mathbb{E}u_2)+\check{B}\mathbb{E}u_2\big\}ds\\
						&\quad +\big\{ \tilde{C}(x_{0l}-\mathbb{E}x_{0l})+\check{C}\mathbb{E}x_{0l}+\tilde{F}^\top(\eta_{0l}-\mathbb{E}\eta_{0l})+\check{F}^\top\mathbb{E}\eta_{0l} \\
						&\qquad+\tilde{K}(\zeta_{0l}-\mathbb{E}\zeta_{0l})+\check{K}\mathbb{E}\zeta_{0l}+\tilde{D}(u_2-\mathbb{E}u_2)+\check{D}\mathbb{E}u_2\big\}dW(s),\\
						-d\eta_{0l}(s)&=\big\{ \tilde{A}^\top(\eta_{0l}-\mathbb{E}\eta_{0l})+\check{A}^\top\mathbb{E}\eta_{0l}+\tilde{C}^\top(\zeta_{0l}-\mathbb{E}\zeta_{0l})+\check{C}^\top\mathbb{E}\zeta_{0l} \\
						&\qquad +\tilde{N}(u_2-\mathbb{E}u_2)+\check{N}\mathbb{E}u_2\big\}ds-\zeta_{0l}dW(s),\\
						x_{0l}(t)&=0,\,\,\,\,\,\eta_{0l}(T)=0.
					\end{aligned}\right.
				\end{equation}
			\end{mythm}
			
			\begin{proof}
				Suppose $(\bar{x}(\cdot),\bar{\eta}_1(\cdot),\bar{\zeta}_1(\cdot),\bar{u}_2(\cdot))$ is a state-control quadruple corresponding to the given initial pair $(t,\xi) \in [0,T] \times L^2_{\mathcal{F}_t}(\Omega;\mathbb{R}^n)$. For any $u_2(\cdot) \in \mathcal{U}_2[t,T]$ and $\epsilon \in \mathbb{R}$, let $u_2^\epsilon(\cdot)=\bar{u}_2(\cdot)+\epsilon u_2(\cdot)$ and triple $(\bar{x}^{\epsilon}(\cdot)\equiv x(\cdot;\bar{\Theta}_1,\bar{\hat{\Theta}}_1,\bar{v}_1,\bar{u}_2+\epsilon u_2(\cdot)),\bar{\eta}^\epsilon_1(\cdot),\bar{\zeta}^\epsilon_1(\cdot))$ be the corresponding state. Thus,
				\begin{equation*}
					x_{0l}(\cdot) \equiv \frac{\bar{x}^{\epsilon}(\cdot)-\bar{x}(\cdot)}{\epsilon},\quad
					\eta_{0l}(\cdot) \equiv \frac{\bar{\eta}^\epsilon_1(\cdot)-\bar{\eta}_1(\cdot)}{\epsilon},\quad
					\zeta_{0l}(\cdot) \equiv \frac{\bar{\zeta}^\epsilon_1(\cdot)-\bar{\zeta}_1(\cdot)}{\epsilon}
				\end{equation*}
				are independent of $\epsilon$ and satisfy (\ref{leader x_0l}). Applying It\^o's formula to $\big\langle q^{\bar{u}_2}_2(\cdot),x_{0l}(\cdot) \big\rangle-\big\langle p^{\bar{u}_2}_2(\cdot),\eta_{0l}(\cdot) \big\rangle$, we can get
				\begin{equation*}
					\begin{aligned}
						&\hat{J}_2(t,\xi;\bar{u}_2(\cdot)+\epsilon u_2(\cdot))-\hat{J}_2(t,\xi;\bar{u}_2(\cdot))\\
						&=2\epsilon\mathbb{E}\biggl\{  \int_t^T \big\langle \tilde{B}^\top(q_2^{\bar{u}_2}-\mathbb{E}q_2^{\bar{u}_2})+\check{B}^\top\mathbb{E}q^{\bar{u}_2}_2
						+\tilde{D}^\top(k^{\bar{u}_2}_2-\mathbb{E}k^{\bar{u}_2}_2)+\check{D}^\top\mathbb{E}k^{\bar{u}_2}_2\\
						&\qquad\qquad +\tilde{N}^\top(p^{\bar{u}_2}_2-\mathbb{E}p^{\bar{u}_2}_2)+\check{N}^\top\mathbb{E}p^{\bar{u}_2}_2+\tilde{S}_1(\bar{x}-\mathbb{E}\bar{x})
						+\check{S}_1\mathbb{E}\bar{x}+\tilde{S}_2(\bar{\eta}_1-\mathbb{E}\bar{\eta}_1)+\check{S}_2\mathbb{E}\bar{\eta}_1\\
						&\qquad\qquad +\tilde{S}_3(\bar{\zeta}_1-\mathbb{E}\bar{\zeta}_1)
						+\check{S}_3\mathbb{E}\bar{\zeta}_1 +\tilde{R}(\bar{u}_2-\mathbb{E}\bar{u}_2)+\check{R}\mathbb{E}\bar{u}_2+\tilde{\rho}-\mathbb{E}\tilde{\rho}
						+\check{\rho},u_2 \big\rangle ds \biggr\}\\
						&\quad +\epsilon^2\mathbb{E}\biggl\{ \big\langle G^2\big(x_{0l}(T)-\mathbb{E}x_{0l}(T)\big),x_{0l}(T)-\mathbb{E}x_{0l}(T)\big\rangle
						+\big\langle \boldsymbol{G}^2\mathbb{E}x_{0l}(T),\mathbb{E}x_{0l}(T) \big\rangle\\
						&\qquad\qquad +\int_t^T\Big[ \big\langle \tilde{Q}_{11}(x_{0l}-\mathbb{E}x_{0l}),x_{0l}-\mathbb{E}x_{0l} \big\rangle
						+2\big\langle \tilde{Q}_{12}(x_{0l}-\mathbb{E}x_{0l}),\eta_{0l}-\mathbb{E}\eta_{0l}\big\rangle\\
						&\qquad\qquad\qquad +2\big\langle \tilde{Q}_{13}(x_{0l}-\mathbb{E}x_{0l}),\zeta_{0l}-\mathbb{E}\zeta_{0l}\big\rangle
						+\big\langle \tilde{Q}_{22}(\eta_{0l}-\mathbb{E}\eta_{0l}),\eta_{0l}-\mathbb{E}\eta_{0l}\big\rangle\\
						&\qquad\qquad\qquad +2\big\langle \tilde{Q}_{23}(\eta_{0l}-\mathbb{E}\eta_{0l}),\zeta_{0l}-\mathbb{E}\zeta_{0l}\big\rangle
						+\big\langle \tilde{Q}_{33}(\zeta_{0l}-\mathbb{E}\zeta_{0l}),\zeta_{0l}-\mathbb{E}\zeta_{0l}\big\rangle \\
						&\qquad\qquad\qquad +2\big\langle \tilde{S}_1(x_{0l}-\mathbb{E}x_{0l}),u_2-\mathbb{E}u_2\big\rangle
						+2\big\langle \tilde{S}_2(\eta_{0l}-\mathbb{E}\eta_{0l}),u_2-\mathbb{E}u_2\big\rangle\\
						&\qquad\qquad\qquad +2\big\langle \tilde{S}_3(\zeta_{0l}-\mathbb{E}\zeta_{0l}),u_2-\mathbb{E}u_2\big\rangle
						+\big\langle \tilde{R}(u_2-\mathbb{E}u_2),u_2-\mathbb{E}u_2\big\rangle\\
						&\qquad\qquad\qquad +\big\langle \check{Q}_{11}\mathbb{E}x_{0l},\mathbb{E}x_{0l} \big\rangle
						+\big\langle \check{Q}_{22}\mathbb{E}\eta_{0l},\mathbb{E}\eta_{0l}\big\rangle+2\big\langle \check{Q}_{12}\mathbb{E}x_{0l},\mathbb{E}\eta_{0l}\big\rangle\\
						&\qquad\qquad\qquad +2\big\langle \check{Q}_{13}\mathbb{E}x_{0l},\mathbb{E}\zeta_{0l}\big\rangle
						+2\big\langle \check{Q}_{23}\mathbb{E}\eta_{0l},\mathbb{E}\zeta_{0l}\big\rangle+\big\langle \check{Q}_{33}\mathbb{E}\zeta_{0l},\mathbb{E}\zeta_{0l}\big\rangle\\
						&\qquad\qquad\qquad +2\big\langle \check{S}_1\mathbb{E}x_{0l},\mathbb{E}u_2\big\rangle+2\big\langle \check{S}_2\mathbb{E}\eta_{0l},\mathbb{E}u_2\big\rangle
						+2\big\langle \check{S}_3\mathbb{E}\zeta_{0l},\mathbb{E}u_2\big\rangle+\big\langle \check{R}\mathbb{E}u_2,\mathbb{E}u_2\big\rangle\Big]ds \biggr\}.
				\end{aligned}\end{equation*}
				Therefore, $(\bar{x}(\cdot),\bar{\eta}_1(\cdot),\bar{\zeta}_1(\cdot),\bar{u}_2(\cdot))$ is an open-loop optimal quadruple of Problem (MF-SLQ)$_{l}$ if and only if (\ref{leader open stationarity}) and (\ref{leader convexity condition}) hold.
			\end{proof}
			
			Now we introduce the closed-loop system of Problem (MF-SLQ)$_l$. Taking $\big(\Theta_2(\cdot),\hat{\Theta}_2(\cdot),\Delta_2(\cdot),\\\hat{\Delta}_2(\cdot)\big) \in \mathcal{Q}_2[t,T]\times \mathcal{Q}_2[t,T]\times \mathcal{Q}_2[t,T]\times \mathcal{Q}_2[t,T]$ and $v_2(\cdot)\in \mathcal{U}_2[t,T]$, for any initial pair $(t,\xi) \in [0,T] \times L^2_{\mathcal{F}_t}(\Omega;\mathbb{R}^n)$, let us consider the following MF-FBSDE on $[t,T]$:
			\begin{equation}\label{leader closedloop system}
				\left\{\begin{aligned}
					d\bar{x}^{\Theta,\Delta,v}(s)&=\big\{ (\tilde{A}+\tilde{B}\Theta_2)(\bar{x}^{\Theta,\Delta,v}-\mathbb{E}\bar{x}^{\Theta,\Delta,v})+\big(\check{A}+\check{B}\hat{\Theta}_2\big)\mathbb{E}\bar{x}^{\Theta,\Delta,v}\\
					&\quad\ +(\tilde{M}+\tilde{B}\Delta_2)(\bar{\eta}_1^{\Theta,\Delta,v}-\mathbb{E}\bar{\eta}_1^{\Theta,\Delta,v})
					+(\check{M}+\check{B}\hat{\Delta}_2)\mathbb{E}\bar{\eta}_1^{\Theta,\Delta,v} \\
					&\quad\ +\tilde{F}(\bar{\zeta}_1^{\Theta,\Delta,v}-\mathbb{E}\bar{\zeta}_1^{\Theta,\Delta,v})+\check{F}\mathbb{E}\bar{\zeta}_1^{\Theta,\Delta,v}
					+\tilde{B}(v_2-\mathbb{E}v_2)+\check{B}\mathbb{E}v_2+\tilde{b}\big\}ds\\
					&\ +\big\{ (\tilde{C}+\tilde{D}\Theta_2)(\bar{x}^{\Theta,\Delta,v}-\mathbb{E}\bar{x}^{\Theta,\Delta,v})+(\check{C}+\check{D}\hat{\Theta}_2)\mathbb{E}\bar{x}^{\Theta,\Delta,v}\\
					&\quad\ +(\tilde{F}^\top+\tilde{D}\Delta_2)(\bar{\eta}_1^{\Theta,\Delta,v}-\mathbb{E}\bar{\eta}_1^{\Theta,\Delta,v})+(\check{F}^\top
					+\check{D}\hat{\Delta}_2)\mathbb{E}\bar{\eta}_1^{\Theta,\Delta,v}+\tilde{K}(\bar{\zeta}_1^{\Theta,\Delta,v} \\
					&\quad\ -\mathbb{E}\bar{\zeta}_1^{\Theta,\Delta,v})+\check{K}\mathbb{E}\bar{\zeta}_1^{\Theta,\Delta,v}+\tilde{D}(v_2-\mathbb{E}v_2)+\check{D}\mathbb{E}v_2+\tilde{\sigma}\big\}dW(s),\\
					-d\bar{\eta}_1^{\Theta,\Delta,v}(s)&=\big\{ (\tilde{A}^\top+\tilde{N}\Delta_2)(\bar{\eta}_1^{\Theta,\Delta,v}-\mathbb{E}\bar{\eta}_1^{\Theta,\Delta,v})
					+(\check{A}^\top+\check{N}\hat{\Delta}_2)\mathbb{E}\bar{\eta}_1^{\Theta,\Delta,v}\\	
					&\quad\ +\tilde{C}^\top(\bar{\zeta}_1^{\Theta,\Delta,v}-\mathbb{E}\bar{\zeta}_1^{\Theta,\Delta,v})+\check{C}^\top\mathbb{E}\bar{\zeta}_1^{\Theta,\Delta,v}
					+\tilde{N}\Theta_2(\bar{x}^{\Theta,\Delta,v}-\mathbb{E}\bar{x}^{\Theta,\Delta,v})\\
					&\quad\ +\check{N}\hat{\Theta}_2\mathbb{E}\bar{x}^{\Theta,\Delta,v}+\tilde{N}(v_2-\mathbb{E}v_2)+\check{N}\mathbb{E}v_2+\tilde{f}\big\}ds-\bar{\zeta}_1^{\Theta,\Delta,v}dW(s),\\
					\bar{x}^{\Theta,\Delta,v}(t)&=\xi,\,\,\,\,\,\bar{\eta}_1^{\Theta,\Delta,v}(T)=\boldsymbol{g}^1.
				\end{aligned}\right.
			\end{equation}
			This is a MF-FBSDE admitting a unique solution $(\bar{x}^{\Theta,\Delta,v}(\cdot)\equiv \bar{x}^{\Theta_2,\hat{\Theta}_2,\Delta_2,\hat{\Delta}_2,v_2}(\cdot),\bar{\eta}_1^{\Theta,\Delta,v}(\cdot),\\\bar{\zeta}_1^{\Theta,\Delta,v}(\cdot))$, which is called leader's closed-loop system of the original state equation (\ref{leader state}) under leader's closed-loop strategy $\big(\Theta_2(\cdot),\hat{\Theta}_2(\cdot),\Delta_2(\cdot),\hat{\Delta}_2(\cdot),v_2(\cdot)\big)$ which is independent of the initial state $\xi$. We denote $(\bar{x}^{\Theta,\Delta,v}(\cdot),\bar{\eta}_1^{\Theta,\Delta,v}(\cdot),\bar{\zeta}_1^{\Theta,\Delta,v}(\cdot))$ as $(\check{x}(\cdot),\check{\eta}_1(\cdot),\check{\zeta}_1(\cdot))$ and define
			\begin{equation}\label{lclf}
				\begin{aligned}
					&\hat{J}_2\big(t,\xi;\Theta_2(\check{x}-\mathbb{E}\check{x})+\hat{\Theta}_2\mathbb{E}\check{x}+\Delta_2(\check{\eta}_1-\mathbb{E}\check{\eta}_1)
					+\hat{\Delta}_2\mathbb{E}\check{\eta}_1+v_2\big)\\
					&=\mathbb{E} \bigg\{ \big\langle G^2\big(\check{x}(T)-\mathbb{E}\check{x}(T)\big),\check{x}(T)-\mathbb{E}\check{x}(T)\big\rangle
					+\big\langle \boldsymbol{G}^2\mathbb{E}\check{x}(T),\mathbb{E}\check{x}(T) \big\rangle+2\big\langle g^2,\check{x}(T)-\mathbb{E}\check{x}(T)\big\rangle\\
					&\qquad +2\big\langle \boldsymbol{g}^2,\mathbb{E}\check{x}(T)\big\rangle
					+\int_t^T \Big[ \big\langle (\tilde{Q}_{11}+\Theta_2^\top\tilde{S}_1+\tilde{S}^\top_1\Theta_2+\Theta_2^\top\tilde{R}\Theta_2)(\check{x}-\mathbb{E}\check{x}),\check{x}-\mathbb{E}\check{x} \big\rangle\\
					&\qquad +2\big\langle (\tilde{Q}_{12}+\Delta_2^\top\tilde{S}_1+\tilde{S}^\top_2\Theta_2+\Delta_2^\top\tilde{R}\Theta_2)(\check{x}
					-\mathbb{E}\check{x}),\check{\eta}_1-\mathbb{E}\check{\eta}_1 \big\rangle\\
					&\qquad +2\big\langle (\tilde{Q}_{13}+\tilde{S}^\top_3\Theta_2)(\check{x}-\mathbb{E}\check{x}),\check{\zeta}_1-\mathbb{E}\check{\zeta}_1 \big\rangle
					+2\big\langle (\tilde{Q}_{23}+\tilde{S}^\top_3\Delta_2)(\check{\eta}_1-\mathbb{E}\check{\eta}_1),\check{\zeta}_1-\mathbb{E}\check{\zeta}_1 \big\rangle\\
					&\qquad +\big\langle (\tilde{Q}_{22}+\Delta_2^\top\tilde{S}_2+\tilde{S}^\top_2\Delta_2+\Delta_2^\top\tilde{R}\Delta_2)(\check{\eta}_1-\mathbb{E}\check{\eta}_1),\check{\eta}_1-\mathbb{E}\check{\eta}_1 \big\rangle
					+\big\langle \tilde{Q}_{33}(\check{\zeta}_1-\mathbb{E}\check{\zeta}_1),\check{\zeta}_1-\mathbb{E}\check{\zeta}_1 \big\rangle\\
					&\qquad +2\big\langle (\tilde{S}_1+\tilde{R}\Theta_2)(\check{x}-\mathbb{E}\check{x}),v_2-\mathbb{E}v_2 \big\rangle
					+2\big\langle (\tilde{S}_2+\tilde{R}\Delta_2)(\check{\eta}_1-\mathbb{E}\check{\eta}_1),v_2-\mathbb{E}v_2 \big\rangle\\
					&\qquad +2\big\langle \tilde{S}_3(\check{\zeta}_1-\mathbb{E}\check{\zeta}_1),v_2-\mathbb{E}v_2 \big\rangle+\big\langle\tilde{R}(v_2-\mathbb{E}v_2),v_2-\mathbb{E}v_2 \big\rangle\\
					&\qquad +2\big\langle \tilde{q}_1-\mathbb{E}\tilde{q}_1+\Theta_2^\top(\tilde{\rho}-\mathbb{E}\tilde{\rho}),\check{x}-\mathbb{E}\check{x}\big\rangle
					+2\big\langle \tilde{q}_2-\mathbb{E}\tilde{q}_2+\Delta_2^\top(\tilde{\rho}-\mathbb{E}\tilde{\rho}),\check{\eta}_1-\mathbb{E}\check{\eta}_1\big\rangle\\
					&\qquad +2\big\langle \tilde{q}_3-\mathbb{E}\tilde{q}_3,\check{\zeta}_1-\mathbb{E}\check{\zeta}_1 \big \rangle +2\big\langle \tilde{\rho}-\mathbb{E}\tilde{\rho},v_2-\mathbb{E}v_2 \big\rangle
					+\big\langle (\check{Q}_{11}+\hat{\Theta}_2^\top\check{S}_1+\check{S}^\top_1\hat{\Theta}_2+\hat{\Theta}_2^\top\check{R}\hat{\Theta}_2)\mathbb{E}\check{x},\mathbb{E}\check{x} \big\rangle\\
					&\qquad +2\big\langle (\check{Q}_{12}+\hat{\Delta}^\top_2\check{S}_1+\check{S}^\top_2\hat{\Theta}_2+\hat{\Delta}^\top_2\check{R}\hat{\Theta}_2)\mathbb{E}\check{x},\mathbb{E}\check{\eta}_1\big\rangle
					+2\big\langle (\check{Q}_{13}+\check{S}^\top_3\hat{\Theta}_2)\mathbb{E}\check{x},\mathbb{E}\check{\zeta}_1\big\rangle\\
					&\qquad +\big\langle (\check{Q}_{22}+\check{S}^\top_2\hat{\Delta}_2+\hat{\Delta}^\top_2\check{S}_2+\hat{\Delta}^\top_2\check{R}\hat{\Delta}_2)\mathbb{E}\check{\eta}_1,\mathbb{E}\check{\eta}_1 \big\rangle
					+2\big\langle (\check{Q}_{23}+\check{S}^\top_3\hat{\Delta}_2)\mathbb{E}\check{\eta}_1,\mathbb{E}\check{\zeta}_1\big\rangle\\
					&\qquad +\big\langle \check{Q}_{33}\mathbb{E}\check{\zeta}_1,\mathbb{E}\check{\zeta}_1 \big \rangle+2\big\langle (\check{S}_1+\check{R}\hat{\Theta}_2)\mathbb{E}\check{x},\mathbb{E}v_2 \big\rangle
					+2\big\langle (\check{S}_2+\check{R}\hat{\Delta}_2)\mathbb{E}\check{\eta}_1,\mathbb{E}v_2 \big\rangle +2\big\langle \check{S}_3\mathbb{E}\check{\zeta}_1,\mathbb{E}v_2 \big\rangle\\
					&\qquad  +\big\langle \check{R}\mathbb{E}v_2,\mathbb{E}v_2 \big\rangle+2\big\langle \check{q}_1+\hat{\Theta}^\top_2\check{\rho},\mathbb{E}\check{x}\big\rangle
					+2\big\langle \check{q}_2+\hat{\Delta}^\top_2\check{\rho},\mathbb{E}\check{\eta}_1 \big\rangle
					+2\big\langle \check{q}_3,\mathbb{E}\check{\zeta}_1 \big\rangle+2\big\langle \check{\rho},\mathbb{E}v_2 \big\rangle\Big]ds+\mathcal{L}\bigg\}.
				\end{aligned}
			\end{equation}
			
			\begin{mydef}\label{def4.1}
				A quintuple $(\bar{\Theta}_2(\cdot),\bar{\hat{\Theta}}_2(\cdot),\bar{\Delta}_2(\cdot),\bar{\hat{\Delta}}_2(\cdot),\bar{v}_2(\cdot)) \in \mathcal{Q}_2[t,T] \times \mathcal{Q}_2[t,T]\times \mathcal{Q}_2[t,T]\times \mathcal{Q}_2[t,T] \times \mathcal{U}_2[t,T]$ is called a (unique) leader's closed-loop optimal strategy of Problem (MF-SLQ)$_l$ on $[t,T]$ if
				\begin{equation}
					\begin{aligned}
						&\hat{J}_2\big(t,\xi;\bar{\Theta}_2(\bar{x}-\mathbb{E}\bar{x})+\bar{\hat{\Theta}}_2\mathbb{E}\bar{x}+\bar{\Delta}_2(\bar{\eta}_1-\mathbb{E}\bar{\eta}_1)
						+\bar{\hat{\Delta}}_2\mathbb{E}\bar{\eta}_1+\bar{v}_2\big)\\
						&\leq \hat{J}_2\big(t,\xi;\Theta_2(\check{x}-\mathbb{E}\check{x})+\hat{\Theta}_2\mathbb{E}\check{x}+\Delta_2(\check{\eta}_1-\mathbb{E}\check{\eta}_1)
						+\hat{\Delta}_2\mathbb{E}\check{\eta}_1+v_2\big),\\
						&\forall (\Theta_2(\cdot),\hat{\Theta}_2(\cdot),\Delta_2(\cdot),\hat{\Delta}_2(\cdot),v_2(\cdot)) \in
						\mathcal{Q}_2[t,T] \times \mathcal{Q}_2[t,T]\times \mathcal{Q}_2[t,T]\times \mathcal{Q}_2[t,T] \times \mathcal{U}_2[t,T],\\
						&\hspace{5cm}\forall (t,\xi) \in [0,T] \times L^2_{\mathcal{F}_t}(\Omega;\mathbb{R}^n),
					\end{aligned}
				\end{equation}
				where $\bar{x}(\cdot)\equiv \bar{x}^{\bar{\Theta},\bar{\Delta},\bar{v}}(\cdot)$ together with $\bar{\eta}_1(\cdot)$ and $\bar{\zeta}_1(\cdot)$ satisfying (\ref{leader closedloop system}). In this time, Problem (MF-SLQ)$_l$ is said to be (uniquely) closed-loop solvable.
			\end{mydef}
			
			Similar to Lemma \ref{relation}, we have the following result.
			\begin{mypro}\label{pro4.3}
				Let (H1)-(H2) hold. Then the following are equivalent:
				\par (\romannumeral 1) $(\bar{\Theta}_2(\cdot),\bar{\hat{\Theta}}_2(\cdot),\bar{\Delta}_2(\cdot),\bar{\hat{\Delta}}_2(\cdot),\bar{v}_2(\cdot)) \in \mathcal{Q}_2[t,T] \times \mathcal{Q}_2[t,T]\times \mathcal{Q}_2[t,T]\times \mathcal{Q}_2[t,T] \times \mathcal{U}_2[t,T]$ is leader's closed-loop optimal strategy of Problem (MF-SLQ)$_l$.
				
				\par (\romannumeral 2) The following holds:
				\begin{equation*}
					\begin{aligned}
						&\hat{J}_2\big(t,\xi;\bar{\Theta}_2(\bar{x}-\mathbb{E}\bar{x})+\bar{\hat{\Theta}}_2\mathbb{E}\bar{x}+\bar{\Delta}_2(\bar{\eta}_1-\mathbb{E}\bar{\eta}_1)
						+\bar{\hat{\Delta}}_2\mathbb{E}\bar{\eta}_1+\bar{v}_2\big)\\
						&\quad \leq \hat{J}_2\big(t,\xi;\bar{\Theta}_2(\bar{x}^{\bar{\Theta},\bar{\Delta},v}-\mathbb{E}\bar{x}^{\bar{\Theta},\bar{\Delta},v})
						+\bar{\hat{\Theta}}_2\mathbb{E}\bar{x}^{\bar{\Theta},\bar{\Delta},v}+\bar{\Delta}_2(\bar{\eta}^{\bar{\Theta},\bar{\Delta},v}_1-\mathbb{E}\bar{\eta}^{\bar{\Theta},\bar{\Delta},v}_1)
						+\bar{\hat{\Delta}}_2\mathbb{E}\bar{\eta}^{\bar{\Theta},\bar{\Delta},v}_1+v_2\big),\\
						&\hspace{2cm}\forall (t,\xi) \in [0,T] \times L^2_{\mathcal{F}_t}(\Omega;\mathbb{R}^n),\,\, \forall v_2(\cdot) \in \mathcal{U}_2[t,T],
					\end{aligned}
				\end{equation*}
				where $(\bar{x}^{\bar{\Theta},\bar{\Delta},v}(\cdot)\equiv \bar{x}^{\bar{\Theta}_2,\bar{\hat{\Theta}}_2,\bar{\Delta}_2,\bar{\hat{\Delta}}_2,v_2}(\cdot),\bar{\eta}_1^{\bar{\Theta},\bar{\Delta},v}(\cdot),\bar{\zeta}_1^{\bar{\Theta},
					\bar{\Delta},v}(\cdot))$ is the solution of (\ref{leader closedloop system}).
				
				\par (\romannumeral 3) The following holds:
				\begin{equation*}
					\begin{aligned}
						&\hat{J}_2\big(t,\xi;\bar{\Theta}_2(\bar{x}-\mathbb{E}\bar{x})+\bar{\hat{\Theta}}_2\mathbb{E}\bar{x}+\bar{\Delta}_2(\bar{\eta}_1-\mathbb{E}\bar{\eta}_1)
						+\bar{\hat{\Delta}}_2\mathbb{E}\bar{\eta}_1+\bar{v}_2\big) \leq \hat{J}_2(t,\xi;u_2(\cdot)),\\
						&\hspace{2cm}\forall (t,\xi) \in [0,T] \times L^2_{\mathcal{F}_t}(\Omega;\mathbb{R}^n),\,\, \forall u_2(\cdot) \in \mathcal{U}_2[t,T].
					\end{aligned}
				\end{equation*}
			\end{mypro}
			
			By Proposition \ref{pro4.3} (\romannumeral 2) and Theorem \ref{leader open-loop ns condition}, we can obtain:
			\begin{equation}\label{leader optimal system0}
				\begin{aligned}
					0&=\tilde{N}^\top(\bar{p}-\mathbb{E}\bar{p})+\check{N}^\top\mathbb{E}\bar{p}+\tilde{B}^\top(\bar{q}-\mathbb{E}\bar{q})+\check{B}^\top\mathbb{E}\bar{q}
					+\tilde{D}^\top(\bar{k}-\mathbb{E}\bar{k})+\check{D}^\top\mathbb{E}\bar{k}\\
					&\quad +(\tilde{S}_1+\tilde{R}\bar{\Theta}_2)(\bar{x}-\mathbb{E}\bar{x})+(\check{S}_1+\check{R}\bar{\hat{\Theta}}_2)\mathbb{E}\bar{x}
					+(\tilde{S}_2+\tilde{R}\bar{\Delta}_2)(\bar{\eta}_1-\mathbb{E}\bar{\eta}_1)+(\check{S}_2+\check{R}\bar{\hat{\Delta}}_2)\mathbb{E}\bar{\eta}_1\\
					&\quad +\tilde{S}_3(\bar{\zeta}_1-\mathbb{E}\bar{\zeta}_1)+\check{S}_3\mathbb{E}\bar{\zeta}_1+\tilde{R}(\bar{v}_2-\mathbb{E}\bar{v}_2)
					+\check{R}\mathbb{E}\bar{v}_2+\tilde{\rho}-\mathbb{E}\tilde{\rho}+\check{\rho}, \quad  a.e.,\, \mathbb{P}\mbox{-}a.s.,
				\end{aligned}
			\end{equation}
			with
			\begin{equation}\label{leader optimal system}
				\left\{\begin{aligned}
					d\bar{x}(s)&=\big\{ (\tilde{A}+\tilde{B}\bar{\Theta}_2)(\bar{x}-\mathbb{E}\bar{x})+(\check{A}+\check{B}\bar{\hat{\Theta}}_2)\mathbb{E}\bar{x}
					+(\tilde{M}+\tilde{B}\bar{\Delta}_2)(\bar{\eta}_1-\mathbb{E}\bar{\eta}_1)\\
					&\quad\ +(\check{M}+\check{B}\bar{\hat{\Delta}}_2)\mathbb{E}\bar{\eta}_1 +\tilde{F}(\bar{\zeta}_1-\mathbb{E}\bar{\zeta}_1)+\check{F}\mathbb{E}\bar{\zeta}_1
					+\tilde{B}(\bar{v}_2-\mathbb{E}\bar{v}_2)+\check{B}\mathbb{E}\bar{v}_2+\tilde{b}\big\}ds\\
					&\ +\big\{ (\tilde{C}+\tilde{D}\bar{\Theta}_2)(\bar{x}-\mathbb{E}\bar{x})+(\check{C}+\check{D}\bar{\hat{\Theta}}_2)\mathbb{E}\bar{x}+(\tilde{F}^\top
					+\tilde{D}\bar{\Delta}_2)(\bar{\eta}_1-\mathbb{E}\bar{\eta}_1)\\
					&\quad\ +(\check{F}^\top+\check{D}\bar{\hat{\Delta}}_2)\mathbb{E}\bar{\eta}_1+\tilde{K}(\bar{\zeta}_1-\mathbb{E}\bar{\zeta}_1)
					+\check{K}\mathbb{E}\bar{\zeta}_1+\tilde{D}(\bar{v}_2-\mathbb{E}\bar{v}_2)+\check{D}\mathbb{E}\bar{v}_2+\tilde{\sigma}\big\}dW(s),\\
					d\bar{p}(s)&=\big\{(\tilde{A}^\top+\tilde{N}\bar{\Delta}_2)^\top(\bar{p}-\mathbb{E}\bar{p})+(\check{A}^\top+\check{N}\bar{\hat{\Delta}}_2)^\top\mathbb{E}\bar{p}
					+(\tilde{M}+\tilde{B}\bar{\Delta}_2)^\top(\bar{q}-\mathbb{E}\bar{q})\\
					&\quad\ +(\check{M}+\check{B}\bar{\hat{\Delta}}_2)^\top\mathbb{E}\bar{q}+(\tilde{F}^\top+\tilde{D}\bar{\Delta}_2)^\top(\bar{k}-\mathbb{E}\bar{k})+(\check{F}^\top
					+\check{D}\bar{\hat{\Delta}}_2)^\top\mathbb{E}\bar{k}\\
					&\quad\ +(\tilde{Q}_{12}+\bar{\Delta}_2^\top\tilde{S}_1+\tilde{S}^\top_2\bar{\Theta}_2+\bar{\Delta}_2^\top\tilde{R}\bar{\Theta}_2)(\bar{x}-\mathbb{E}\bar{x})
					+(\tilde{Q}_{23}+\tilde{S}^\top_3\bar{\Delta}_2)^\top(\bar{\zeta}_1-\mathbb{E}\bar{\zeta}_1)\\
					&\quad\ +(\check{Q}_{12}+\bar{\hat{\Delta}}^\top_2\check{S}_1+\check{S}^\top_2\bar{\hat{\Theta}}_2+\bar{\hat{\Delta}}^\top_2\check{R}\bar{\hat{\Theta}}_2)\mathbb{E}\bar{x}
					+(\check{Q}_{23}+\check{S}^\top_3\bar{\hat{\Delta}}_2)^\top\mathbb{E}\bar{\zeta}_1\\
					&\quad\ +(\tilde{Q}_{22}+\bar{\Delta}_2^\top\tilde{S}_2+\tilde{S}^\top_2\bar{\Delta}_2+\bar{\Delta}_2^\top\tilde{R}\bar{\Delta}_2)(\bar{\eta}_1-\mathbb{E}\bar{\eta}_1)
					+(\check{Q}_{22}+\check{S}^\top_2\bar{\hat{\Delta}}_2+\bar{\hat{\Delta}}^\top_2\check{S}_2\\
					&\quad\ +\bar{\hat{\Delta}}^\top_2\check{R}\bar{\hat{\Delta}}_2)\mathbb{E}\bar{\eta}_1+(\tilde{S}_2+\tilde{R}\bar{\Delta}_2)^\top(\bar{v}_2-\mathbb{E}\bar{v}_2)+(\check{S}_2
					+\check{R}\bar{\hat{\Delta}}_2)^\top\mathbb{E}\bar{v}_2\\
					&\quad\ +\tilde{q}_2-\mathbb{E}\tilde{q}_2+\check{q}_2+\bar{\Delta}_2^\top(\tilde{\rho}-\mathbb{E}\tilde{\rho})+\bar{\hat{\Delta}}^\top_2\check{\rho}\big\}ds
					+\big\{\tilde{C}(\bar{p}-\mathbb{E}\bar{p})+\check{C}\mathbb{E}\bar{p}+\tilde{F}^\top(\bar{q}-\mathbb{E}\bar{q})\\
					&\quad\ +\check{F}^\top\mathbb{E}\bar{q}+\tilde{K}^\top(\bar{k}-\mathbb{E}\bar{k})+\check{K}^\top\mathbb{E}\bar{k}+(\tilde{Q}_{13}
					+\tilde{S}^\top_3\bar{\Theta}_2)(\bar{x}-\mathbb{E}\bar{x})+(\check{Q}_{13}+\check{S}^\top_3\bar{\hat{\Theta}}_2)\mathbb{E}\bar{x}\\
					&\quad\ +(\tilde{Q}_{23}+\tilde{S}^\top_3\bar{\Delta}_2)(\bar{\eta}_1-\mathbb{E}\bar{\eta}_1)+\big(\check{Q}_{23}+\check{S}^\top_3\bar{\hat{\Delta}}_2\big)\mathbb{E}\bar{\eta}_1
					+\tilde{Q}_{33}\big(\bar{\zeta}_1-\mathbb{E}\bar{\zeta}_1\big)\\
					&\quad\ +\check{Q}_{33}\mathbb{E}\bar{\zeta}_1+\tilde{S}^\top_3(\bar{v}_2-\mathbb{E}\bar{v}_2)+\check{S}_3^\top\mathbb{E}\bar{v}_2+\tilde{q}_3-\mathbb{E}\tilde{q}_3
					+\check{q}_3\big\}dW(s),\\
					-d\bar{\eta}_1(s)&=\big\{ (\tilde{A}^\top+\tilde{N}\bar{\Delta}_2)(\bar{\eta}_1-\mathbb{E}\bar{\eta}_1)+(\check{A}^\top+\check{N}\bar{\hat{\Delta}}_2)\mathbb{E}\bar{\eta}_1
					+\tilde{C}^\top(\bar{\zeta}_1-\mathbb{E}\bar{\zeta}_1)+\check{C}^\top\mathbb{E}\bar{\zeta}_1\\
					&\quad\ +\tilde{N}\bar{\Theta}_2(\bar{x}-\mathbb{E}\bar{x})+\check{N}\bar{\hat{\Theta}}_2\mathbb{E}\bar{x}+\tilde{N}(\bar{v}_2-\mathbb{E}\bar{v}_2)
					+\check{N}\mathbb{E}\bar{v}_2+\tilde{f}\big\}ds-\bar{\zeta}_1dW(s),\\
					-d\bar{q}(s)&=\big\{\tilde{A}^\top(\bar{q}-\mathbb{E}\bar{q})+\check{A}^\top\mathbb{E}\bar{q}+\tilde{C}^\top(\bar{k}-\mathbb{E}\bar{k})+\check{C}^\top\mathbb{E}\bar{k}
					+(\tilde{Q}_{11}+\tilde{S}^\top_1\bar{\Theta}_2)(\bar{x}-\mathbb{E}\bar{x})\\
					&\quad\ +(\check{Q}_{11}+\check{S}^\top_1\bar{\hat{\Theta}}_2)\mathbb{E}\bar{x}+(\tilde{Q}_{12}^\top+\tilde{S}^\top_1\bar{\Delta}_2)(\bar{\eta}_1-\mathbb{E}\bar{\eta}_1)
					+(\check{Q}^\top_{12}+\check{S}^\top_1\bar{\hat{\Delta}}_2)\mathbb{E}\bar{\eta}_1+\tilde{Q}^\top_{13}(\bar{\zeta}_1\\
					&\quad\ -\mathbb{E}\bar{\zeta}_1)+\check{Q}^\top_{13}\mathbb{E}\bar{\zeta}_1+\tilde{S}^\top_1(\bar{v}_2-\mathbb{E}\bar{v}_2)
					+\check{S}^\top_1\mathbb{E}\bar{v}_2+\tilde{q}_1-\mathbb{E}\tilde{q}_1-\check{q}_1\big\}ds-\bar{k}dW(s),\\
					\bar{x}(t)&=\xi,\,\,\,\bar{p}(t)=0,\,\,\,\bar{\eta}_1(T)=\boldsymbol{g}^1,\,\,\,\bar{q}(T)=G^2\big(\bar{x}(T)-\mathbb{E}\bar{x}(T)\big)+\boldsymbol{G}^2\mathbb{E}\bar{x}(T)
					+\boldsymbol{g}^2.
				\end{aligned}\right.
			\end{equation}
			
			\begin{Remark}
				According to the equivalence of (\romannumeral 1) and (\romannumeral 2) in Proposition \ref{pro4.3}, $(\bar{\Theta}_2(\cdot),\bar{\hat{\Theta}}_2(\cdot),\bar{\Delta}_2(\cdot),\\\bar{\hat{\Delta}}_2(\cdot),\bar{v}_2(\cdot))$ is the closed-loop optimal strategy of the Problem (MF-SLQ)$_l$  if and only if $\bar{v}_2(\cdot)$ is the open-loop optimal control of the SLQ optimal control problem of (\ref{leader closedloop system}), (\ref{lclf}) with $\big(\Theta_2(\cdot),\\\hat{\Theta}_2(\cdot),\Delta_2(\cdot),\hat{\Delta}_2(\cdot)\big)=\big(\bar{\Theta}_2(\cdot),\bar{\hat{\Theta}}_2(\cdot),\bar{\Delta}_2(\cdot),
				\bar{\hat{\Delta}}_2(\cdot)\big)$. Let us label this problem for $v_2(\cdot)$ as \textbf{Problem (MF-SLQ)$^{\bar{\Theta},\bar{\Delta}}_l$}. The above MF-FBSDE is the optimality system of Problem (MF-SLQ)$^{\bar{\Theta},\bar{\Delta}}_l$. Similarly, this conclusion holds for Problem (MF-SLQ)$_{l0}$, and the corresponding problem about $v_2(\cdot)$ of Problem (MF-SLQ)$_{l0}$ is denoting as \textbf{Problem (MF-SLQ)$^{\bar{\Theta},\bar{\Delta}}_{l0}$}.
			\end{Remark}
			
			\begin{mypro}
				Let (H1)-(H2) hold. If $(\bar{\Theta}_2(\cdot),\bar{\hat{\Theta}}_2(\cdot),\bar{\Delta}_2(\cdot),\bar{\hat{\Delta}}_2(\cdot),\bar{v}_2(\cdot)) \in \mathcal{Q}_2[t,T] \times \mathcal{Q}_2[t,T]\times \mathcal{Q}_2[t,T]\times \mathcal{Q}_2[t,T] \times \mathcal{U}_2[t,T]$ is the closed-loop optimal strategy of Problem (MF-SLQ)$_l$, then $(\bar{\Theta}_2(\cdot),\bar{\hat{\Theta}}_2(\cdot),\bar{\Delta}_2(\cdot),\bar{\hat{\Delta}}_2(\cdot),0)$ is the closed-loop optimal strategy of Problem (MF-SLQ)$_{l0}$ on $[t,T]$.
			\end{mypro}
			
			\begin{proof}
				By the preceding discussion, we see that $(\bar{\Theta}_2(\cdot),\bar{\hat{\Theta}}_2(\cdot),\bar{\Delta}_2(\cdot),\bar{\hat{\Delta}}_2(\cdot),\bar{v}_2(\cdot))$ is the closed-loop optimal strategy of Problem (MF-SLQ)$_l$ on $[t,T]$ if and only if for any initial pair $(t,\xi) \in [0,T]\times L^2_{\mathcal{F}_t}(\Omega;\mathbb{R}^n)$, system (\ref{leader optimal system}) has the adapted solution $\big(\bar{x}(\cdot),\bar{\eta}_1(\cdot),\bar{\zeta}_1(\cdot),\bar{p}(\cdot),\bar{q}(\cdot),\bar{k}(\cdot)\big)$, the stationarity condition (\ref{leader optimal system0}) holds and
				\begin{equation*}
					\begin{aligned}
						&\mathbb{E}\bigg\{ \big\langle G^2\big(x^0(T)-\mathbb{E}x^0(T)\big),x^0(T)-\mathbb{E}x^0(T)\big\rangle+\big\langle \boldsymbol{G}^2\mathbb{E}x^0(T),\mathbb{E}x^0(T) \big\rangle\\
						&\quad +\int_t^T \Big[ \big\langle (\tilde{Q}_{11}+\bar{\Theta}_2^\top\tilde{S}_1+\tilde{S}^\top_1\bar{\Theta}_2+\bar{\Theta}_2^\top\tilde{R}\bar{\Theta}_2)(x^0-\mathbb{E}x^0),x^0-\mathbb{E}x^0 \big\rangle\\
						&\qquad +2\big\langle (\tilde{Q}_{12}+\bar{\Delta}_2^\top\tilde{S}_1+\tilde{S}^\top_2\bar{\Theta}_2+\bar{\Delta}_2^\top\tilde{R}\bar{\Theta}_2)(x^0-\mathbb{E}x^0),\eta^0-\mathbb{E}\eta^0 \big\rangle\\
						&\qquad +2\big\langle (\tilde{Q}_{13}+\tilde{S}^\top_3\bar{\Theta}_2)(x^0-\mathbb{E}x^0),\zeta^0-\mathbb{E}\zeta^0 \big\rangle
						+2\big\langle (\tilde{Q}_{23}+\tilde{S}^\top_3\bar{\Delta}_2)(\eta^0-\mathbb{E}\eta^0),\zeta^0-\mathbb{E}\zeta^0 \big\rangle\\
						&\qquad +\big\langle (\tilde{Q}_{22}+\bar{\Delta}_2^\top\tilde{S}_2+\tilde{S}^\top_2\bar{\Delta}_2+\bar{\Delta}_2^\top\tilde{R}\bar{\Delta}_2)(\eta^0-\mathbb{E}\eta^0),\eta^0-\mathbb{E}\eta^0 \big\rangle\\
						&\qquad +\big\langle \tilde{Q}_{33}(\zeta^0-\mathbb{E}\zeta^0),\zeta^0-\mathbb{E}\zeta^0 \big\rangle
						+2\big\langle (\tilde{S}_1+\tilde{R}\bar{\Theta}_2)(x^0-\mathbb{E}x^0),v_2-\mathbb{E}v_2 \big\rangle \\
						&\qquad +2\big\langle (\tilde{S}_2+\tilde{R}\bar{\Delta}_2)(\eta^0-\mathbb{E}\eta^0),v_2-\mathbb{E}v_2 \big\rangle+2\big\langle \tilde{S}_3(\zeta^0-\mathbb{E}\zeta^0),v_2-\mathbb{E}v_2 \big\rangle\\
						&\qquad +\big\langle\tilde{R}(v_2-\mathbb{E}v_2\big),v_2-\mathbb{E}v_2 \big\rangle+\big\langle(\check{Q}_{11}+\bar{\hat{\Theta}}_2^\top\check{S}_1+\check{S}^\top_1\bar{\hat{\Theta}}_2
						+\bar{\hat{\Theta}}_2^\top\check{R}\bar{\hat{\Theta}}_2)\mathbb{E}x^0,\mathbb{E}x^0 \big\rangle\\
						&\qquad +2\big\langle (\check{Q}_{12}+\bar{\hat{\Delta}}^\top_2\check{S}_1+\check{S}^\top_2\bar{\hat{\Theta}}_2
						+\bar{\hat{\Delta}}^\top_2\check{R}\bar{\hat{\Theta}}_2)\mathbb{E}x^0,\mathbb{E}\eta^0\big\rangle+2\big\langle (\check{Q}_{13}+\check{S}^\top_3\bar{\hat{\Theta}}_2)\mathbb{E}x^0,\mathbb{E}\zeta^0\big\rangle\\
						&\qquad+\big\langle (\check{Q}_{22}+\check{S}^\top_2\bar{\hat{\Delta}}_2+\bar{\hat{\Delta}}^\top_2\check{S}_2
						+\bar{\hat{\Delta}}^\top_2\check{R}\bar{\hat{\Delta}}_2)\mathbb{E}\eta^0,\mathbb{E}\eta^0 \big\rangle
						+2\big\langle(\check{Q}_{23}+\check{S}^\top_3\bar{\hat{\Delta}}_2)\mathbb{E}\eta^0,\mathbb{E}\zeta^0\big\rangle\\
						&\qquad+\big\langle \check{Q}_{33}\mathbb{E}\zeta^0,\mathbb{E}\zeta^0 \big \rangle+2\big\langle (\check{S}_1+\check{R}\bar{\hat{\Theta}}_2)\mathbb{E}x^0,\mathbb{E}v_2 \big\rangle
						+2\big\langle (\check{S}_2+\check{R}\bar{\hat{\Delta}}_2)\mathbb{E}\eta^0,\mathbb{E}v_2 \big\rangle\\
						&\qquad +2\big\langle \check{S}_3\mathbb{E}\zeta^0,\mathbb{E}v_2 \big\rangle+\big\langle \check{R}\mathbb{E}v_2,\mathbb{E}v_2 \big\rangle\Big]ds\bigg\} \geq 0,\quad \forall v_2(\cdot) \in \mathcal{U}_2[t,T],
					\end{aligned}
				\end{equation*}
				where $\big(x^0(\cdot),\eta^0(\cdot),\zeta^0(\cdot)\big)$ is the solution to
				\begin{equation*}
					\left\{\begin{aligned}
						dx^0(s)&=\big\{ (\tilde{A}+\tilde{B}\bar{\Theta}_2)(x^0-\mathbb{E}x^0)+(\check{A}+\check{B}\bar{\hat{\Theta}}_2)\mathbb{E}x^0+(\tilde{M}+\tilde{B}\bar{\Delta}_2)(\eta^0-\mathbb{E}\eta^0)\\
						&\qquad+(\check{M}+\check{B}\bar{\hat{\Delta}}_2)\mathbb{E}\eta^0+\tilde{F}(\zeta^0-\mathbb{E}\zeta^0)+\check{F}\mathbb{E}\zeta^0+\tilde{B}(v_2-\mathbb{E}v_2)
						+\check{B}\mathbb{E}v_2\big\}ds\\
						&\quad +\big\{ (\tilde{C}+\tilde{D}\bar{\Theta}_2)(x^0-\mathbb{E}x^0)+(\check{C}+\check{D}\bar{\hat{\Theta}}_2)\mathbb{E}x^0+(\tilde{F}^\top+\tilde{D}\bar{\Delta}_2)(\eta^0-\mathbb{E}\eta^0)\\
						&\qquad +(\check{F}^\top+\check{D}\bar{\hat{\Delta}}_2)\mathbb{E}\eta^0+\tilde{K}(\zeta^0-\mathbb{E}\zeta^0)+\check{K}\mathbb{E}\zeta^0+\tilde{D}(v_2-\mathbb{E}v_2)
						+\check{D}\mathbb{E}v_2\big\}dW(s),\\
						-d\eta^0(s)&=\big\{ (\tilde{A}^\top+\tilde{N}\bar{\Delta}_2)(\eta^0-\mathbb{E}\eta^0)+(\check{A}^\top+\check{N}\bar{\hat{\Delta}}_2)\mathbb{E}\eta^0
						+\tilde{C}^\top(\zeta^0-\mathbb{E}\zeta^0)+\check{C}^\top\mathbb{E}\zeta^0\\
						&\qquad +\tilde{N}\bar{\Theta}_2(x^0-\mathbb{E}x^0)+\check{N}\bar{\hat{\Theta}}_2\mathbb{E}x^0+\tilde{N}(v_2-\mathbb{E}v_2)+\check{N}\mathbb{E}v_2\big\}ds-\zeta^0dW(s),\\
						x^0(t)&=0,\,\,\,\,\,\eta^0(T)=0.
					\end{aligned}\right.
				\end{equation*}
				Since (\ref{leader optimal system}) admits a solution for each $\xi \in L^2_{\mathcal{F}_t}(\Omega;\mathbb{R}^n)$ and
				$(\bar{\Theta}_2(\cdot),\bar{\hat{\Theta}}_2(\cdot),\bar{\Delta}_2(\cdot),\bar{\hat{\Delta}}_2(\cdot),\bar{v}_2(\cdot))$ is independent of $\xi$, by subtracting solutions corresponding $\xi$ and $0$, the latter from the former, we see that for any $\xi \in L^2_{\mathcal{F}_t}(\Omega;\mathbb{R}^n)$, as long as $(\bar{x}_{l0}(\cdot),\bar{\eta}_{l0}(\cdot),\bar{\zeta}_{l0}(\cdot),\bar{p}_{l0}(\cdot),\bar{q}_{l0}(\cdot),\bar{k}_{l0}(\cdot))$ is the adapted solution to the following MF-FBSDE:
				\begin{equation}\label{problem(l0) optimal system-1}
					\left\{\begin{aligned}
						d\bar{x}_{l0}(s)&=\big\{ (\tilde{A}+\tilde{B}\bar{\Theta}_2)(\bar{x}_{l0}-\mathbb{E}\bar{x}_{l0})+(\check{A}+\check{B}\bar{\hat{\Theta}}_2)\mathbb{E}\bar{x}_{l0}
						+(\tilde{M}+\tilde{B}\bar{\Delta}_2)(\bar{\eta}_{l0}-\mathbb{E}\bar{\eta}_{l0})\\
						&\qquad +(\check{M}+\check{B}\bar{\hat{\Delta}}_2)\mathbb{E}\bar{\eta}_{l0}+\tilde{F}(\bar{\zeta}_{l0}-\mathbb{E}\bar{\zeta}_{l0})+\check{F}\mathbb{E}\bar{\zeta}_{l0}\big\}ds\\
						&\quad +\big\{(\tilde{C}+\tilde{D}\bar{\Theta}_2)(\bar{x}_{l0}-\mathbb{E}\bar{x}_{l0})+(\check{C}+\check{D}\bar{\hat{\Theta}}_2)\mathbb{E}\bar{x}_{l0}
						+(\tilde{F}^\top+\tilde{D}\bar{\Delta}_2)(\bar{\eta}_{l0}-\mathbb{E}\bar{\eta}_{l0})\\
						&\qquad +(\check{F}^\top+\check{D}\bar{\hat{\Delta}}_2)\mathbb{E}\bar{\eta}_{l0}+\tilde{K}(\bar{\zeta}_{l0}-\mathbb{E}\bar{\zeta}_{l0})+\check{K}\mathbb{E}\bar{\zeta}_{l0}\big\}dW(s),\\
						d\bar{p}_{l0}(s)&=\big\{(\tilde{A}^\top+\tilde{N}\bar{\Delta}_2)^\top(\bar{p}_{l0}-\mathbb{E}\bar{p}_{l0})+(\check{A}^\top+\check{N}\bar{\hat{\Delta}}_2)^\top\mathbb{E}\bar{p}_{l0}
						+(\tilde{M}+\tilde{B}\bar{\Delta}_2)^\top(\bar{q}_{l0}\\
						&\qquad -\mathbb{E}\bar{q}_{l0})+(\check{M}+\check{B}\bar{\hat{\Delta}}_2)^\top\mathbb{E}\bar{q}_{l0}
						+(\tilde{F}^\top+\tilde{D}\bar{\Delta}_2)^\top(\bar{k}_{l0}-\mathbb{E}\bar{k}_{l0})+(\check{F}^\top\\
						&\qquad +\check{D}\bar{\hat{\Delta}}_2)^\top\mathbb{E}\bar{k}_{l0}+(\tilde{Q}_{12}+\bar{\Delta}_2^\top\tilde{S}_1+\tilde{S}^\top_2\bar{\Theta}_2
						+\bar{\Delta}_2^\top\tilde{R}\bar{\Theta}_2)(\bar{x}_{l0}-\mathbb{E}\bar{x}_{l0})\\
						&\qquad +(\tilde{Q}_{23}+\tilde{S}^\top_3\bar{\Delta}_2)^\top(\bar{\zeta}_{l0} -\mathbb{E}\bar{\zeta}_{l0})+(\check{Q}_{12}+\bar{\hat{\Delta}}^\top_2\check{S}_1
						+\check{S}^\top_2\bar{\hat{\Theta}}_2+\bar{\hat{\Delta}}^\top_2\check{R}\bar{\hat{\Theta}}_2)\mathbb{E}\bar{x}_{l0}\\
						&\qquad +(\check{Q}_{22}+\check{S}^\top_2\bar{\hat{\Delta}}_2+\bar{\hat{\Delta}}^\top_2\check{S}_2
						+\bar{\hat{\Delta}}^\top_2\check{R}\bar{\hat{\Delta}}_2)\mathbb{E}\bar{\eta}_{l0}
						+(\check{Q}_{23}+\check{S}^\top_3\bar{\hat{\Delta}}_2)^\top\mathbb{E}\bar{\zeta}_{l0}\\
						&\qquad +(\tilde{Q}_{22}+\bar{\Delta}_2^\top\tilde{S}_2+\tilde{S}^\top_2\bar{\Delta}_2
						+\bar{\Delta}_2^\top\tilde{R}\bar{\Delta}_2)(\bar{\eta}_{l0}-\mathbb{E}\bar{\eta}_{l0})\big\}ds+\big\{\tilde{C}(\bar{p}_{l0}-\mathbb{E}\bar{p}_{l0})\\
						&\qquad +\check{C}\mathbb{E}\bar{p}_{l0}+\tilde{F}^\top(\bar{q}_{l0}-\mathbb{E}\bar{q}_{l0})+\check{F}^\top\mathbb{E}\bar{q}_{l0}
						+\tilde{K}^\top(\bar{k}_{l0}-\mathbb{E}\bar{k}_{l0})+\check{K}^\top\mathbb{E}\bar{k}_{l0}\\
						&\qquad +(\tilde{Q}_{13}+\tilde{S}^\top_3\bar{\Theta}_2)(\bar{x}_{l0}-\mathbb{E}\bar{x}_{l0})+(\check{Q}_{13}
						+\check{S}^\top_3\bar{\hat{\Theta}}_2)\mathbb{E}\bar{x}_{l0}+(\tilde{Q}_{23}+\tilde{S}^\top_3\bar{\Delta}_2)(\bar{\eta}_{l0}\\
						&\qquad -\mathbb{E}\bar{\eta}_{l0}) +(\check{Q}_{23}+\check{S}^\top_3\bar{\hat{\Delta}}_2)\mathbb{E}\bar{\eta}_{l0}+\tilde{Q}_{33}(\bar{\zeta}_{l0}-\mathbb{E}\bar{\zeta}_{l0})+\check{Q}_{33}\mathbb{E}\bar{\zeta}_{l0}\big\}dW(s),\\
						-d\bar{\eta}_{l0}(s)&=\big\{ (\tilde{A}^\top+\tilde{N}\bar{\Delta}_2)(\bar{\eta}_{l0}-\mathbb{E}\bar{\eta}_{l0}\big)+(\check{A}^\top
						+\check{N}\bar{\hat{\Delta}}_2)\mathbb{E}\bar{\eta}_{l0}+\tilde{C}^\top(\bar{\zeta}_{l0}-\mathbb{E}\bar{\zeta}_{l0})\\
						&\quad\ +\check{C}^\top\mathbb{E}\bar{\zeta}_{l0}+\tilde{N}\bar{\Theta}_2(\bar{x}_{l0}-\mathbb{E}\bar{x}_{l0})+\check{N}\bar{\hat{\Theta}}_2\mathbb{E}\bar{x}_{l0}\big\}ds
						-\bar{\zeta}_{l0}dW(s),\\
						-d\bar{q}_{l0}(s)&=\big\{\tilde{A}^\top(\bar{q}_{l0}-\mathbb{E}\bar{q}_{l0})+\check{A}^\top\mathbb{E}\bar{q}_{l0}+\tilde{C}^\top(\bar{k}_{l0}
						-\mathbb{E}\bar{k}_{l0})+\check{C}^\top\mathbb{E}\bar{k}_{l0}+(\tilde{Q}_{11}\\
						&\quad\ +\tilde{S}^\top_1\bar{\Theta}_2)(\bar{x}_{l0}-\mathbb{E}\bar{x}_{l0})+(\check{Q}_{11}+\check{S}^\top_1\bar{\hat{\Theta}}_2)\mathbb{E}\bar{x}_{l0}
						+(\tilde{Q}_{12}^\top+\tilde{S}^\top_1\bar{\Delta}_2)(\bar{\eta}_{l0}-\mathbb{E}\bar{\eta}_{l0})\\
						&\quad\ +(\check{Q}^\top_{12}+\check{S}^\top_1\bar{\hat{\Delta}}_2)\mathbb{E}\bar{\eta}_{l0}+\tilde{Q}^\top_{13}(\bar{\zeta}_{l0}-\mathbb{E}\bar{\zeta}_{l0})
						+\check{Q}^\top_{13}\mathbb{E}\bar{\zeta}_{l0}\big\}ds-\bar{k}_{l0}dW(s),\\
						\bar{x}_{l0}(t)&=\xi,\,\,\,\bar{p}_{l0}(t)=0,\,\,\,\bar{\eta}_{l0}(T)=0,\,\,\,
						\bar{q}_{l0}(T)=G^2\big(\bar{x}_{l0}(T)-\mathbb{E}\bar{x}_{l0}(T)\big)+ \boldsymbol{G}^2\mathbb{E}\bar{x}_{l0}(T),	
					\end{aligned}
					\right.
				\end{equation}
				it must have
				\begin{equation}\label{problem(l0) optimal system-2}
					\begin{aligned}
						0&=\tilde{N}^\top(\bar{p}_{l0}-\mathbb{E}\bar{p}_{l0})+\check{N}^\top\mathbb{E}\bar{p}_{l0}+\tilde{B}^\top(\bar{q}_{l0}-\mathbb{E}\bar{q}_{l0})
						+\check{B}^\top\mathbb{E}\bar{q}_{l0}+\tilde{D}^\top(\bar{k}_{l0}-\mathbb{E}\bar{k}_{l0})\\
						&\quad +\check{D}^\top\mathbb{E}\bar{k}_{l0}+(\tilde{S}_1+\tilde{R}\bar{\Theta}_2)(\bar{x}_{l0}-\mathbb{E}\bar{x}_{l0})+(\check{S}_1+\check{R}\bar{\hat{\Theta}}_2)\mathbb{E}\bar{x}_{l0}
						+(\tilde{S}_2+\tilde{R}\bar{\Delta}_2)(\bar{\eta}_{l0}-\mathbb{E}\bar{\eta}_{l0})\\
						&\quad +(\check{S}_2+\check{R}\bar{\hat{\Delta}}_2)\mathbb{E}\bar{\eta}_{l0}+\tilde{S}_3(\bar{\zeta}_{l0}-\mathbb{E}\bar{\zeta}_{l0})+\check{S}_3\mathbb{E}\bar{\zeta}_{l0},\quad a.e.,\, \mathbb{P}\mbox{-}a.s..
					\end{aligned}
				\end{equation}
				It follows, again from Theorem \ref{leader open-loop ns condition}, (\romannumeral 1) and (\romannumeral 2) in Proposition \ref{pro4.3}, that $\bar{v}_2(\cdot)=0$ is an open-loop optimal control of Problem (MF-SLQ)$_{l0}^{\bar{\Theta},\bar{\Delta}}$ and $(\bar{\Theta}_2(\cdot),\bar{\hat{\Theta}}_2(\cdot),\bar{\Delta}_2(\cdot),\bar{\hat{\Delta}}_2(\cdot),0)$ is the closed-loop optimal strategy of Problem (MF-SLQ)$_{l0}$ on $[t,T]$.
			\end{proof}
			
			It is clear that when we want to study necessary conditions for the closed-loop solvability of Problem (MF-SLQ)$_l$, we can transform the original problem into the open-loop solvability of Problem (MF-SLQ)$^{\bar{\Theta},\bar{\Delta}}_{l0}$. To summarize the relationship between Problem (MF-SLQ)$_l$, Problem (MF-SLQ)$_{l0}$, Problem (MF-SLQ)$^{\bar{\Theta},\bar{\Delta}}_l$ and Problem (MF-SLQ)$^{\bar{\Theta},\bar{\Delta}}_{l0}$, we plot the following diagram in the following Figure 1.
			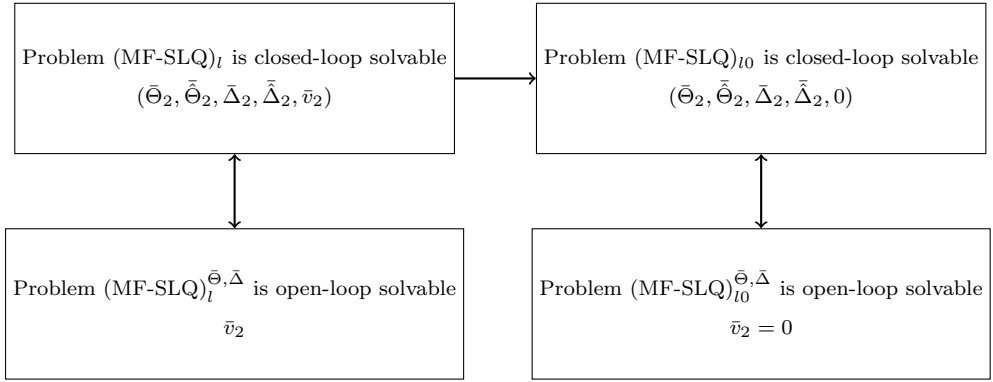
\begin{figure}[H]
				\scriptsize
				\tikzstyle{format} = [rectangle, minimum width = 3cm, minimum height=2cm, align=center, draw = black]
				\tikzstyle{test}=[diamond,aspect=10,draw,thin]
				\tikzstyle{point}=[coordinate,on grid,]
				\begin{tikzpicture}
					\node[format] (SLQP)
					{Problem (MF-SLQ)$_l$ is closed-loop solvable \\[1ex] $(\bar{\Theta}_2,\bar{\hat{\Theta}}_2,\bar{\Delta}_2,\bar{\hat{\Delta}}_2,\bar{v}_2)$};
					\node[format,below of=SLQP,node distance=3cm](SLQPT)
					{Problem (MF-SLQ)$^{\bar{\Theta},\bar{\Delta}}_l$ is open-loop solvable \\[1ex] $\bar{v}_2$};
					\node[format,right of=SLQP,node distance=7cm](SLQP0)
					{Problem (MF-SLQ)$_{l0}$ is closed-loop solvable \\[1ex] $(\bar{\Theta}_2,\bar{\hat{\Theta}}_2,\bar{\Delta}_2,\bar{\hat{\Delta}}_2,0)$};
					\node[format,below of=SLQP0,node distance=3cm](SLQPT0)
					{Problem (MF-SLQ)$^{\bar{\Theta},\bar{\Delta}}_{l0}$ is open-loop solvable \\[1ex] $\bar{v}_2=0$};
					\draw[->,thick](SLQP)--(SLQP0);
					\draw[<->,thick](SLQP)--(SLQPT);
					\draw[<->,thick](SLQP0)--(SLQPT0);
				\end{tikzpicture}
				\caption{Relationship}\label{fig:1}
			\end{figure}
			
			Now let us consider Problem (MF-SLQ)$_{l0}^{\bar{\Theta},\bar{\Delta}}$ whose optimality system is (\ref{problem(l0) optimal system-1}), (\ref{problem(l0) optimal system-2}) and the open-loop optimal control is $\bar{v}_2(\cdot)=0$. We set
			\begin{equation*}
				\begin{aligned}
					&X:=\left(\begin{matrix} \bar{x}_{l0}  \\ \bar{p}_{l0}  \end{matrix}\right),\quad
					Y:=\left(\begin{matrix} \bar{q}_{l0} \\ \bar{\eta}_{l0} \end{matrix}\right),\quad
					Z:=\left(\begin{matrix} \bar{k}_{l0} \\ \bar{\zeta}_{l0} \end{matrix}\right),\quad
					X_0:=\left(\begin{matrix} \xi  \\ 0  \end{matrix}\right),\\
					&\bar{\boldsymbol{\Theta}}_2:=\left(\begin{matrix} \bar{\Theta}_2 & 0 \end{matrix}\right),\quad
					\bar{\hat{\boldsymbol{\Theta}}}_2:=\left(\begin{matrix} \bar{\hat{\Theta}}_2 & 0 \end{matrix}\right),\quad
					\bar{\boldsymbol{\Delta}}_2:=\left(\begin{matrix} 0 & \bar{\Delta}_2 \end{matrix}\right),\quad
					\bar{\hat{\boldsymbol{\Delta}}}_2:=\left(\begin{matrix} 0 & \bar{\hat{\Delta}}_2 \end{matrix}\right),
				\end{aligned}
			\end{equation*}
			and
			\begin{equation*}
				\begin{aligned}
					&\tilde{\mathcal{A}}:=\left(\begin{matrix} \tilde{A} & 0 \\ \tilde{Q}_{12} & \tilde{A} \end{matrix}\right),\,\,\,
					\check{\mathcal{A}}:=\left(\begin{matrix} \check{A} & 0 \\ \check{Q}_{12} & \check{A} \end{matrix}\right),\,\,\,
					\tilde{\mathcal{M}}:=\left(\begin{matrix} 0 & \tilde{M} \\ \tilde{M}^\top & \tilde{Q}_{22} \end{matrix}\right),\,\,\,
					\check{\mathcal{M}}:=\left(\begin{matrix} 0 & \check{M} \\ \check{M}^\top & \check{Q}_{22} \end{matrix}\right),\\
					&\tilde{\mathcal{F}}:=\left(\begin{matrix} 0 & \tilde{F} \\ \tilde{F} & \tilde{Q}_{23}^\top \end{matrix}\right),\,\,\,
					\check{\mathcal{F}}:=\left(\begin{matrix} 0 & \check{F} \\ \check{F} & \check{Q}_{23}^\top \end{matrix}\right),\,\,\,
					\tilde{\mathcal{H}}:=\left(\begin{matrix} \tilde{Q}_{11} & 0  \\  0 & 0 \end{matrix}\right),\,\,\,
					\check{\mathcal{H}}:=\left(\begin{matrix} \check{Q}_{11} & 0  \\  0 & 0 \end{matrix}\right),\\
					&\tilde{\mathcal{C}}:=\left(\begin{matrix} \tilde{C} & 0  \\ \tilde{Q}_{13} & \tilde{C} \end{matrix}\right),\,\,\,
					\check{\mathcal{C}}:=\left(\begin{matrix} \check{C} & 0  \\ \check{Q}_{13} & \check{C} \end{matrix}\right),\,\,\,
					\tilde{\mathcal{K}}:=\left(\begin{matrix} 0 & \tilde{K}  \\ \tilde{K}^\top & \tilde{Q}_{33} \end{matrix}\right),\,\,\,
					\check{\mathcal{K}}:=\left(\begin{matrix} 0 & \check{K}  \\ \check{K}^\top & \check{Q}_{33} \end{matrix}\right),\\
					&\tilde{\mathcal{G}}:=\left(\begin{matrix} G^2 & 0  \\  0 & 0 \end{matrix}\right),\,\,\,
					\check{\mathcal{G}}:=\left(\begin{matrix} \hat{G}^2 & 0  \\  0 & 0 \end{matrix}\right),\,\,\,
					\tilde{\mathcal{B}}:=\left(\begin{matrix} \tilde{B} \\ \tilde{S}^\top_2 \end{matrix}\right),\,\,\,
					\check{\mathcal{B}}:=\left(\begin{matrix} \check{B} \\ \check{S}^\top_2 \end{matrix}\right),\,\,\,
					\tilde{\mathcal{D}}:=\left(\begin{matrix} \tilde{D} \\ \tilde{S}^\top_3 \end{matrix}\right),\,\,\,
					\check{\mathcal{D}}:=\left(\begin{matrix} \check{D} \\ \check{S}^\top_3 \end{matrix}\right),\\
					&\tilde{\mathcal{N}}:=\left(\begin{matrix} \tilde{S}^\top_1 \\ \tilde{N} \end{matrix}\right),\,\,\,
					\check{\mathcal{N}}:=\left(\begin{matrix}  \check{S}^\top_1 \\ \check{N} \end{matrix}\right),\,\,\,
					\tilde{\boldsymbol{b}}:=\left(\begin{matrix}\tilde{b}-\mathbb{E}\tilde{b} \\ \tilde{q}_2-\mathbb{E}\tilde{q}_2  \end{matrix}\right),\,\,\,
					\check{\boldsymbol{b}}:=\left(\begin{matrix} \mathbb{E}\tilde{b} \\ \check{q}_2 \end{matrix}\right),\,\,\,
					\tilde{\boldsymbol{\sigma}}:=\left(\begin{matrix}\tilde{\sigma}-\mathbb{E}\tilde{\sigma} \\ \tilde{q}_3-\mathbb{E}\tilde{q}_3  \end{matrix}\right),\\
					&\check{\boldsymbol{\sigma}}:=\left(\begin{matrix} \mathbb{E}\tilde{\sigma} \\ \check{q}_3 \end{matrix}\right),\,\,\,
					\tilde{\boldsymbol{f}}:=\left(\begin{matrix} \tilde{q}_1-\mathbb{E}\tilde{q}_1 \\ \tilde{f}-\mathbb{E}\tilde{f} \end{matrix}\right),\,\,\,
					\check{\boldsymbol{f}}:=\left(\begin{matrix}  \check{q}_1 \\ \mathbb{E}\tilde{f}  \end{matrix}\right),\,\,\,
					\tilde{\boldsymbol{g}}:=\left(\begin{matrix} g^2 \\ g^1  \end{matrix}\right),\,\,\,
					\check{\boldsymbol{g}}:=\left(\begin{matrix} \hat{g}^2 \\ \hat{g}^1  \end{matrix}\right).
				\end{aligned}
			\end{equation*}
			Then (\ref{problem(l0) optimal system-1}) is equivalent to the following MF-FBSDE:
			\begin{equation}\label{High dimension optimal system simple}
				\left\{\begin{aligned}
					dX(s)&=\big\{ \big(\tilde{\mathcal{A}}+\tilde{\mathcal{B}}\bar{\boldsymbol{\Theta}}_2+\bar{\boldsymbol{\Delta}}_2^\top\tilde{\mathcal{N}}^\top
					+\bar{\boldsymbol{\Delta}}_2^\top\tilde{R}\bar{\boldsymbol{\Theta}}_2\big)\big(X-\mathbb{E}X\big)+\big(\check{\mathcal{A}}+\check{\mathcal{B}}\bar{\hat{\boldsymbol{\Theta}}}_2
					+\bar{\hat{\boldsymbol{\Delta}}}_2^\top\check{\mathcal{N}}^\top\\
					&\qquad +\bar{\hat{\boldsymbol{\Delta}}}_2^\top\check{R}\bar{\hat{\boldsymbol{\Theta}}}_2\big)\mathbb{E}X+\big(\tilde{\mathcal{M}}+\tilde{\mathcal{B}}\bar{\boldsymbol{\Delta}}_2
					+\bar{\boldsymbol{\Delta}}^\top_2\tilde{\mathcal{B}}^\top+\bar{\boldsymbol{\Delta}}^\top_2\tilde{R}\bar{\boldsymbol{\Delta}}_2\big)\big(Y-\mathbb{E}Y\big)\\
					&\qquad +\big(\check{\mathcal{M}}+\check{\mathcal{B}}\bar{\hat{\boldsymbol{\Delta}}}_2+\bar{\hat{\boldsymbol{\Delta}}}^\top_2\check{\mathcal{B}}^\top
					+\bar{\hat{\boldsymbol{\Delta}}}^\top_2\check{R}\bar{\hat{\boldsymbol{\Delta}}}_2\big)\mathbb{E}Y\\
					&\qquad +\big(\tilde{\mathcal{F}}+\bar{\boldsymbol{\Delta}}^\top_2\tilde{\mathcal{D}}^\top\big)\big(Z-\mathbb{E}Z\big)+\big(\check{\mathcal{F}}
					+\bar{\hat{\boldsymbol{\Delta}}}^\top_2\check{\mathcal{D}}^\top\big)\mathbb{E}Z\big\} ds\\
					&\quad +\big\{ \big(\tilde{\mathcal{C}}+\tilde{\mathcal{D}}\bar{\boldsymbol{\Theta}}_2\big)\big(X-\mathbb{E}X\big)+\big(\check{\mathcal{C}}
					+\check{\mathcal{D}}\bar{\hat{\boldsymbol{\Theta}}}_2\big)\mathbb{E}X+\big(\tilde{\mathcal{F}}^\top+\tilde{\mathcal{D}}\bar{\boldsymbol{\Delta}}_2\big)\big(Y-\mathbb{E}Y\big)\\
					&\qquad +\big(\check{\mathcal{F}}^\top+\check{\mathcal{D}}\bar{\hat{\boldsymbol{\Delta}}}_2\big)\mathbb{E}Y+\tilde{\mathcal{K}}\big(Z-\mathbb{E}Z\big)+\check{\mathcal{K}}\mathbb{E}Z \big\}dW(s),\\
					-dY(s)&=\big\{ \big(\tilde{\mathcal{A}}^\top+\tilde{\mathcal{N}}\bar{\boldsymbol{\Delta}}_2\big)\big(Y-\mathbb{E}Y\big)+\big(\check{\mathcal{A}}^\top
					+\check{\mathcal{N}}\bar{\hat{\boldsymbol{\Delta}}}_2\big)\mathbb{E}Y+\tilde{\mathcal{C}}^\top\big(Z-\mathbb{E}Z\big)+\check{\mathcal{C}}^\top\mathbb{E}Z\\
					&\qquad+\big(\tilde{\mathcal{H}}+\tilde{\mathcal{N}}\bar{\boldsymbol{\Theta}}_2\big)\big(X-\mathbb{E}X\big)+\big(\check{\mathcal{H}}+\check{\mathcal{N}}\bar{\hat{\boldsymbol{\Theta}}}_2\big)\mathbb{E}X\big\} ds
					-ZdW(s),\quad s\in[t,T],\\
					X(t)&=X_0,\,\,\,Y(T)=\tilde{\mathcal{G}}\big(X(T)-\mathbb{E}X(T)\big)+\big(\tilde{\mathcal{G}}+\check{\mathcal{G}}\big)\mathbb{E}X(T),
				\end{aligned}\right.
			\end{equation}
			whose adapted solution is $(X(\cdot),Y(\cdot),Z(\cdot))\in L^2_{\mathbb{F}}(t,T;\mathbb{R}^{2n})\times L^2_{\mathbb{F}}(t,T;\mathbb{R}^{2n})\times L^2_{\mathbb{F}}(t,T;\mathbb{R}^{2n})$, with
			\begin{equation}\label{High dimension stationary}
				\begin{aligned}
					0&=\big(\tilde{\mathcal{N}}^\top+\tilde{R}\bar{\boldsymbol{\Theta}}_2\big)\big(X-\mathbb{E}X\big)+\big(\check{\mathcal{N}}^\top+\check{R}\bar{\hat{\boldsymbol{\Theta}}}_2\big)\mathbb{E}X
					+\big(\tilde{\mathcal{B}}^\top+\tilde{R}\bar{\boldsymbol{\Delta}}_2\big)\big(Y-\mathbb{E}Y\big)\\
					&\qquad +\big(\check{\mathcal{B}}^\top+\check{R}\bar{\hat{\boldsymbol{\Delta}}}_2\big)\mathbb{E}Y+\tilde{\mathcal{D}}^\top\big(Z-\mathbb{E}Z\big)+\check{\mathcal{D}}^\top\mathbb{E}Z,\quad a.e.,\, \mathbb{P}\mbox{-}a.s..
				\end{aligned}
			\end{equation}
			Notice that $Y(T)=\tilde{\mathcal{G}}\big(X(T)-\mathbb{E}X(T)\big)+\big(\tilde{\mathcal{G}}+\check{\mathcal{G}}\big)\mathbb{E}X(T)$, we assume that
			\begin{equation*}
				Y(\cdot)=P_2(\cdot)\big(X-\mathbb{E}X\big)(\cdot)+\Pi_2(\cdot)\mathbb{E}X(\cdot),
			\end{equation*}
			where $(P_2(\cdot),\Pi_2(\cdot)) \in C([t,T],\mathbb{R}^{2n \times 2n}) \times C([t,T],\mathbb{R}^{2n \times 2n})$, with $P_2(T)=\tilde{\mathcal{G}}$, $\Pi_2(T)=\tilde{\mathcal{G}}+\check{\mathcal{G}}$. Then we can obtain the following two relationships
			\begin{equation}\label{EY,Y-EY}
				\left\{\begin{aligned}
					\mathbb{E}Y(\cdot)&=\Pi_2(\cdot)\mathbb{E}X(\cdot),\\
					Y(\cdot)-\mathbb{E}Y(\cdot)&=P_2(\cdot)\big(X-\mathbb{E}X\big)(\cdot).
				\end{aligned}\right.
			\end{equation}
			Applying It\^o's formula to the second equation in (\ref{EY,Y-EY}), comparing coefficients of diffusion terms and assuming that $I-P_2\check{\mathcal{K}}$, $I-P_2\tilde{\mathcal{K}}$ are invertible, we can get
			\begin{equation}\label{EZ,Z-EZ}
				\hspace{-2mm}\left\{\begin{aligned}	
					\mathbb{E}Z(\cdot)&=\big(I-P_2\check{\mathcal{K}}\big)^{-1}\big[P_2\big(\check{\mathcal{C}}+\check{\mathcal{D}}\bar{\hat{\boldsymbol{\Theta}}}_2\big)+P_2\big(\check{\mathcal{F}}^\top
					+\check{\mathcal{D}}\bar{\hat{\boldsymbol{\Delta}}}_2\big)\Pi_2\big]\mathbb{E}X(\cdot),\\
					Z(\cdot)-\mathbb{E}Z(\cdot)&=\big(I-P_2\tilde{\mathcal{K}}\big)^{-1}\big[P_2\big(\tilde{\mathcal{C}}+\tilde{\mathcal{D}}\bar{\boldsymbol{\Theta}}_2\big)
					+P_2\big(\tilde{\mathcal{F}}^\top+\tilde{\mathcal{D}}\bar{\boldsymbol{\Delta}}_2\big)P_2\big]\big(X-\mathbb{E}X\big)(\cdot),\,\, \mathbb{P}\mbox{-}a.s..
				\end{aligned}\right.
			\end{equation}
			Then, comparing coefficients of drift terms, we can obtain
			\begin{equation}\label{P2(theta)}
				\begin{aligned}
					0&=\dot{P}_2+\tilde{\mathcal{A}}^\top P_2+P_2\tilde{\mathcal{A}}+\tilde{\mathcal{H}}+P_2\tilde{\mathcal{M}}P_2+\big(\tilde{\mathcal{C}}^\top
					+P_2\tilde{\mathcal{F}}\big)\big(I-P_2\tilde{\mathcal{K}}\big)^{-1}P_2\big(\tilde{\mathcal{C}}+\tilde{\mathcal{F}}^\top P_2\big)\\
					&\quad +\big[\tilde{\mathcal{N}}+P_2\tilde{\mathcal{B}}+\big(\tilde{\mathcal{C}}^\top+P_2\tilde{\mathcal{F}}\big)\big(I-P_2\tilde{\mathcal{K}}\big)^{-1}P_2\tilde{\mathcal{D}}\big]\big(\bar{\boldsymbol{\Theta}}_2
					+\bar{\boldsymbol{\Delta}}_2P_2\big)\\
					&\quad +P_2\bar{\boldsymbol{\Delta}}^\top_2\big\{\tilde{\mathcal{N}}^\top+\tilde{\mathcal{B}}^\top P_2+\big[\tilde{R}
					+\tilde{\mathcal{D}}^\top\big(I-P_2\tilde{\mathcal{K}}\big)^{-1}P_2\tilde{\mathcal{D}}\big]\big(\bar{\boldsymbol{\Theta}}_2+\bar{\boldsymbol{\Delta}}_2P_2\big)\\
					&\qquad\qquad\quad +\tilde{\mathcal{D}}^\top\big(I-P_2\tilde{\mathcal{K}}\big)^{-1}P_2\big(\tilde{\mathcal{C}}+\tilde{\mathcal{F}}^\top P_2\big)\big\}.
				\end{aligned}
			\end{equation}
			Similarly, applying It\^o's formula to the first equation in (\ref{EY,Y-EY}) and comparing coefficients of drift terms, we can obtain
			\begin{equation}\label{Pi2(theta)}
				\begin{aligned}
					0&=\dot{\Pi}_2+\check{\mathcal{A}}^\top \Pi_2+\Pi_2\check{\mathcal{A}}+\check{\mathcal{H}}+\Pi_2\check{\mathcal{M}}\Pi_2+\big(\check{\mathcal{C}}^\top
					+\Pi_2\check{\mathcal{F}}\big)\big(I-P_2\check{\mathcal{K}}\big)^{-1}P_2\big(\check{\mathcal{C}}+\check{\mathcal{F}}^\top\Pi_2\big)\\
					&\quad +\big[\check{\mathcal{N}}+\Pi_2\check{\mathcal{B}}+\big(\check{\mathcal{C}}^\top+\Pi_2\check{\mathcal{F}}\big)\big(I-P_2\check{\mathcal{K}}\big)^{-1}P_2\check{\mathcal{D}}\big]\big(\bar{\hat{\boldsymbol{\Theta}}}_2
					+\bar{\hat{\boldsymbol{\Delta}}}_2\Pi_2\big)+\Pi_2\bar{\hat{\boldsymbol{\Delta}}}^\top_2\big\{\check{\mathcal{N}}^\top+\check{\mathcal{B}}^\top \Pi_2\\
					&\quad +\big[\check{R}+\check{\mathcal{D}}^\top\big(I-P_2\check{\mathcal{K}}\big)^{-1}P_2\check{\mathcal{D}}\big]\big(\bar{\hat{\boldsymbol{\Theta}}}_2+\bar{\hat{\boldsymbol{\Delta}}}_2\Pi_2\big)
					+\check{\mathcal{D}}^\top\big(I-P_2\check{\mathcal{K}}\big)^{-1}P_2\big(\check{\mathcal{C}}+\check{\mathcal{F}}^\top\Pi_2\big)\big\}.
				\end{aligned}
			\end{equation}
			According to stationarity conditions in (\ref{High dimension stationary}) and noting (\ref{EY,Y-EY}), (\ref{EZ,Z-EZ}), we have
			\begin{equation*}
				\left\{\begin{aligned}
					0&=\tilde{\mathcal{N}}^\top+\tilde{\mathcal{B}}^\top P_2+\tilde{\mathcal{D}}^\top\big(I-P_2\tilde{\mathcal{K}}\big)^{-1}P_2\big(\tilde{\mathcal{C}}
					+\tilde{\mathcal{F}}^\top P_2\big)+\tilde{\Sigma}_2\big(\bar{\boldsymbol{\Theta}}_2+\bar{\boldsymbol{\Delta}}_2P_2\big),\\
					0&=\check{\mathcal{N}}^\top+\check{\mathcal{B}}^\top \Pi_2+\check{\mathcal{D}}^\top\big(I-P_2\check{\mathcal{K}}\big)^{-1}P_2\big(\check{\mathcal{C}}
					+\check{\mathcal{F}}^\top \Pi_2\big)+\check{\Sigma}_2\big(\bar{\hat{\boldsymbol{\Theta}}}_2+\bar{\hat{\boldsymbol{\Delta}}}_2\Pi_2\big),\quad a.e.,\, \mathbb{P}\mbox{-}a.s.,
				\end{aligned}\right.
			\end{equation*}
			where we denote the invertible
			\begin{equation}\label{Sigma2 and Sigma2}
				\tilde{\Sigma}_2\equiv\tilde{R}+\tilde{\mathcal{D}}^\top\big(I-P_2\tilde{\mathcal{K}}\big)^{-1}P_2\tilde{\mathcal{D}},\quad \check{\Sigma}_2\equiv\check{R}+\check{\mathcal{D}}^\top\big(I-P_2\check{\mathcal{K}}\big)^{-1}P_2\check{\mathcal{D}}.
			\end{equation}
			
			Let $\bar{\boldsymbol{\Theta}}\equiv\bar{\boldsymbol{\Theta}}_2+\bar{\boldsymbol{\Delta}}_2P_2$ and $\bar{\hat{\boldsymbol{\Theta}}}\equiv\bar{\hat{\boldsymbol{\Theta}}}_2+\bar{\hat{\boldsymbol{\Delta}}}_2\Pi_2$, then we have
			\begin{equation}\label{Theta,hat(Theta)}
				\left\{\begin{aligned}
					\bar{\boldsymbol{\Theta}}&=-\tilde{\Sigma}_2^{-1}\big[\tilde{\mathcal{N}}^\top+\tilde{\mathcal{B}}^\top P_2
					+\tilde{\mathcal{D}}^\top\big(I-P_2\tilde{\mathcal{K}}\big)^{-1}P_2\big(\tilde{\mathcal{C}}+\tilde{\mathcal{F}}^\top P_2\big)\big],\\
					\bar{\hat{\boldsymbol{\Theta}}}&=-\check{\Sigma}_2^{-1}\big[\check{\mathcal{N}}^\top
					+\check{\mathcal{B}}^\top \Pi_2+\check{\mathcal{D}}^\top\big(I-P_2\check{\mathcal{K}}\big)^{-1}P_2\big(\check{\mathcal{C}}+\check{\mathcal{F}}^\top \Pi_2\big)\big],\quad a.e.,\, \mathbb{P}\mbox{-}a.s..
				\end{aligned}\right.
			\end{equation}
			By substituting the above equation into (\ref{P2(theta)}) and (\ref{Pi2(theta)}), we obtain
			\begin{equation}\label{P2}
				\left\{\begin{aligned}
					&0=\dot{P}_2+\tilde{\mathcal{A}}^\top P_2+P_2\tilde{\mathcal{A}}+\tilde{\mathcal{H}}+P_2\tilde{\mathcal{M}}P_2+\big(\tilde{\mathcal{C}}^\top
					+P_2\tilde{\mathcal{F}}\big)\big(I-P_2\tilde{\mathcal{K}}\big)^{-1}P_2\big(\tilde{\mathcal{C}}+\tilde{\mathcal{F}}^\top P_2\big)\\
					&\qquad-\big[\tilde{\mathcal{N}}+P_2\tilde{\mathcal{B}}+\big(\tilde{\mathcal{C}}^\top+P_2\tilde{\mathcal{F}}\big)\big(I-P_2\tilde{\mathcal{K}}\big)^{-1}P_2\tilde{\mathcal{D}}\big]\\
					&\qquad\quad \times \tilde{\Sigma}_2^{-1}\big[\tilde{\mathcal{N}}^\top+\tilde{\mathcal{B}}^\top P_2
					+\tilde{\mathcal{D}}^\top\big(I-P_2\tilde{\mathcal{K}}\big)^{-1}P_2\big(\tilde{\mathcal{C}}+\tilde{\mathcal{F}}^\top P_2\big)\big],\quad s\in[t,T],\\
					&P_2(T)=\tilde{\mathcal{G}},\quad I-P_2\tilde{\mathcal{K}}>0,\quad \tilde{\Sigma}_2\equiv\tilde{R}+\tilde{\mathcal{D}}^\top\big(I-P_2\tilde{\mathcal{K}}\big)^{-1}P_2\tilde{\mathcal{D}}>0,
				\end{aligned}\right.
			\end{equation}
			and
			\begin{equation}\label{Pi2}
				\left\{\begin{aligned}
					&0=\dot{\Pi}_2+\check{\mathcal{A}}^\top \Pi_2+\Pi_2\check{\mathcal{A}}+\check{\mathcal{H}}+\Pi_2\check{\mathcal{M}}\Pi_2+\big(\check{\mathcal{C}}^\top
					+\Pi_2\check{\mathcal{F}}\big)\big(I-P_2\check{\mathcal{K}}\big)^{-1}P_2\big(\check{\mathcal{C}}+\check{\mathcal{F}}^\top\Pi_2\big)\\
					&\qquad-\big[\check{\mathcal{N}}+\Pi_2\check{\mathcal{B}}+\big(\check{\mathcal{C}}^\top+\Pi_2\check{\mathcal{F}}\big)\big(I-P_2\check{\mathcal{K}}\big)^{-1}P_2\check{\mathcal{D}}\big]\\
					&\qquad\quad \times \check{\Sigma}_2^{-1}\big[\check{\mathcal{N}}^\top+\check{\mathcal{B}}^\top \Pi_2+\check{\mathcal{D}}^\top\big(I-P_2\check{\mathcal{K}}\big)^{-1}P_2\big(\check{\mathcal{C}}
					+\check{\mathcal{F}}^\top \Pi_2\big)\big],\quad s\in[t,T],\\
					&\Pi_2(T)=\tilde{\mathcal{G}}+\check{\mathcal{G}},\quad I-P_2\check{\mathcal{K}}>0,\quad \check{\Sigma}_2\equiv\check{R}+\check{\mathcal{D}}^\top\big(I-P_2\check{\mathcal{K}}\big)^{-1}P_2\check{\mathcal{D}}>0.
				\end{aligned}\right.
			\end{equation}
			
			\begin{Remark}
				Riccati equations (\ref{P2}) and (\ref{Pi2}) are the same as (45) and (46) in Lin et al. \cite{LJZ2019}, respectively. There is no results about the solvability of them in their paper, since the general solvability is very difficult due to the appearance of terms like $[\tilde{R}+\tilde{\mathcal{D}}^\top(I-P_2\tilde{\mathcal{K}})^{-1}P_2\tilde{\mathcal{D}}]^{-1}$ and $[\check{R}+\check{\mathcal{D}}^\top(I-P_2\check{\mathcal{K}})^{-1}P_2\check{\mathcal{D}}]^{-1}$. However, in the end of this section of this paper, their solvability is studied under some coefficient assumptions in a special case.
			\end{Remark}
			
			Now, we back to Problem (MF-SLQ)$_l$ whose optimality system is (\ref{leader optimal system}) and we let
			\begin{equation*}
				\begin{aligned}
					&\bar{X}:=\left(\begin{matrix} \bar{x}  \\ \bar{p}  \end{matrix}\right),\,\,\,
					\bar{Y}:=\left(\begin{matrix} \bar{q} \\ \bar{\eta}_1 \end{matrix}\right),\,\,\,
					\bar{Z}:=\left(\begin{matrix} \bar{k} \\ \bar{\zeta}_1 \end{matrix}\right),\,\,\,
					\bar{X}_0:=\left(\begin{matrix} \xi  \\ 0  \end{matrix}\right).
				\end{aligned}
			\end{equation*}
			Thus, we have
			\begin{equation}\label{bar(X,Y,Z)}
				\left\{\begin{aligned}
					d\bar{X}(s)&=\big\{ \big(\tilde{\mathcal{A}}+\tilde{\mathcal{B}}\bar{\boldsymbol{\Theta}}_2+\bar{\boldsymbol{\Delta}}_2^\top\tilde{\mathcal{N}}^\top
					+\bar{\boldsymbol{\Delta}}_2^\top\tilde{R}\bar{\boldsymbol{\Theta}}_2\big)\big(\bar{X}-\mathbb{E}\bar{X}\big)+\big(\check{\mathcal{A}}+\check{\mathcal{B}}\bar{\hat{\boldsymbol{\Theta}}}_2
					+\bar{\hat{\boldsymbol{\Delta}}}_2^\top\check{\mathcal{N}}^\top\\
					&\qquad +\bar{\hat{\boldsymbol{\Delta}}}_2^\top\check{R}\bar{\hat{\boldsymbol{\Theta}}}_2\big)\mathbb{E}\bar{X}+\big(\tilde{\mathcal{M}}
					+\tilde{\mathcal{B}}\bar{\boldsymbol{\Delta}}_2+\bar{\boldsymbol{\Delta}}^\top_2\tilde{\mathcal{B}}^\top
					+\bar{\boldsymbol{\Delta}}^\top_2\tilde{R}\bar{\boldsymbol{\Delta}}_2\big)\big(\bar{Y}-\mathbb{E}\bar{Y}\big)\\
					&\qquad +\big(\check{\mathcal{M}}+\check{\mathcal{B}}\bar{\hat{\boldsymbol{\Delta}}}_2
					+\bar{\hat{\boldsymbol{\Delta}}}^\top_2\check{\mathcal{B}}^\top+\bar{\hat{\boldsymbol{\Delta}}}^\top_2\check{R}\bar{\hat{\boldsymbol{\Delta}}}_2\big)\mathbb{E}\bar{Y}
					+\big(\tilde{\mathcal{F}}+\bar{\boldsymbol{\Delta}}^\top_2\tilde{\mathcal{D}}^\top\big)\big(\bar{Z}-\mathbb{E}\bar{Z}\big)\\
					&\qquad +\big(\check{\mathcal{F}}+\bar{\hat{\boldsymbol{\Delta}}}^\top_2\check{\mathcal{D}}^\top\big)\mathbb{E}\bar{Z}
					+\big(\tilde{\mathcal{B}}+\bar{\boldsymbol{\Delta}}^\top_2\tilde{R}\big)\big(\bar{v}_2-\mathbb{E}\bar{v}_2\big)+\big(\check{\mathcal{B}}+\bar{\hat{\boldsymbol{\Delta}}}^\top_2\check{R}\big)\mathbb{E}\bar{v}_2\\
					&\qquad +\tilde{\boldsymbol{b}}+\check{\boldsymbol{b}}
					+\bar{\boldsymbol{\Delta}}^\top_2\big(\tilde{\rho}-\mathbb{E}\tilde{\rho}\big)+\bar{\hat{\boldsymbol{\Delta}}}^\top_2\check{\rho}\big\}ds
					+\big\{ \big(\tilde{\mathcal{C}}+\tilde{\mathcal{D}}\bar{\boldsymbol{\Theta}}_2\big)\big(\bar{X}-\mathbb{E}\bar{X}\big)\\
					&\qquad +\big(\check{\mathcal{C}}+\check{\mathcal{D}}\bar{\hat{\boldsymbol{\Theta}}}_2\big)\mathbb{E}\bar{X}
					+\big(\tilde{\mathcal{F}}^\top+\tilde{\mathcal{D}}\bar{\boldsymbol{\Delta}}_2\big)\big(\bar{Y}-\mathbb{E}\bar{Y}\big)+\big(\check{\mathcal{F}}^\top+\check{\mathcal{D}}\bar{\hat{\boldsymbol{\Delta}}}_2\big)\mathbb{E}\bar{Y}\\
					&\qquad +\tilde{\mathcal{K}}\big(\bar{Z}-\mathbb{E}\bar{Z}\big)+\check{\mathcal{K}}\mathbb{E}\bar{Z}+\tilde{\mathcal{D}}\big(\bar{v}_2-\mathbb{E}\bar{v}_2\big)+\check{\mathcal{D}}\mathbb{E}\bar{v}_2
					+\tilde{\boldsymbol{\sigma}}+\check{\boldsymbol{\sigma}}\big\} dW(s),\\
					-d\bar{Y}(s)&=\big\{ \big(\tilde{\mathcal{A}}^\top+\tilde{\mathcal{N}}\bar{\boldsymbol{\Delta}}_2\big)\big(\bar{Y}-\mathbb{E}\bar{Y}\big)+\big(\check{\mathcal{A}}^\top
					+\check{\mathcal{N}}\bar{\hat{\boldsymbol{\Delta}}}_2\big)\mathbb{E}\bar{Y}+\tilde{\mathcal{C}}^\top\big(\bar{Z}-\mathbb{E}\bar{Z}\big)+\check{\mathcal{C}}^\top\mathbb{E}\bar{Z}\\
					&\qquad +\big(\tilde{\mathcal{H}}+\tilde{\mathcal{N}}\bar{\boldsymbol{\Theta}}_2\big)\big(\bar{X}-\mathbb{E}\bar{X}\big)+\big(\check{\mathcal{H}}
					+\check{\mathcal{N}}\bar{\hat{\boldsymbol{\Theta}}}_2\big)\mathbb{E}\bar{X}+\tilde{\mathcal{N}}\big(\bar{v}_2-\mathbb{E}\bar{v}_2\big)+\check{\mathcal{N}}\mathbb{E}\bar{v}_2\\
					&\qquad +\tilde{\boldsymbol{f}}+\check{\boldsymbol{f}}\big\} ds-\bar{Z}dW(s),\quad s\in[t,T], \\\
					\bar{X}(t)&=\bar{X}_0,\,\,\,\bar{Y}(T)=\tilde{\mathcal{G}}\big(\bar{X}(T)-\mathbb{E}\bar{X}(T)\big)+\big(\tilde{\mathcal{G}}+\check{\mathcal{G}}\big)\mathbb{E}\bar{X}(T)+\tilde{\boldsymbol{g}}+\check{\boldsymbol{g}},\\
				\end{aligned}\right.
			\end{equation}
			with
			\begin{equation}\label{bar(X,Y,Z)---}
				\begin{aligned}		
					&0=\big(\tilde{\mathcal{N}}^\top+\tilde{R}\bar{\boldsymbol{\Theta}}_2\big)\big(\bar{X}-\mathbb{E}\bar{X}\big)+\big(\check{\mathcal{N}}^\top
					+\check{R}\bar{\hat{\boldsymbol{\Theta}}}_2\big)\mathbb{E}\bar{X}+\big(\tilde{\mathcal{B}}^\top+\tilde{R}\bar{\boldsymbol{\Delta}}_2\big)\big(\bar{Y}-\mathbb{E}\bar{Y}\big)\\
					&\qquad +\big(\check{\mathcal{B}}^\top+\check{R}\bar{\hat{\boldsymbol{\Delta}}}_2\big)\mathbb{E}\bar{Y}+\tilde{\mathcal{D}}^\top\big(\bar{Z}-\mathbb{E}\bar{Z}\big)+\check{\mathcal{D}}^\top\mathbb{E}\bar{Z}\\
					&\qquad +\tilde{R}\big(\bar{v}_2-\mathbb{E}\bar{v}_2\big)+\check{R}\mathbb{E}\bar{v}_2
					+\tilde{\rho}-\mathbb{E}\tilde{\rho}+\check{\rho},\quad a.e.,\, \mathbb{P}\mbox{-}a.s..
				\end{aligned}
			\end{equation}
			
			To determine $\bar{v}_2(\cdot)$, we define
			\begin{equation*}
				\left\{\begin{aligned}
					\bar{\eta}_2&=\bar{Y}-P_2\big(\bar{X}-\mathbb{E}\bar{X}\big)-\Pi_2\mathbb{E}\bar{X},\\
					\bar{\zeta}_2&=\bar{Z}-P_2\big[\big(\tilde{\mathcal{C}}+\tilde{\mathcal{D}}\bar{\boldsymbol{\Theta}}_2\big)\big(\bar{X}-\mathbb{E}\bar{X}\big)+\big(\check{\mathcal{C}}
					+\check{\mathcal{D}}\bar{\hat{\boldsymbol{\Theta}}}_2\big)\mathbb{E}\bar{X}+\big(\tilde{\mathcal{F}}^\top+\tilde{\mathcal{D}}\bar{\boldsymbol{\Delta}}_2\big)\big(\bar{Y}-\mathbb{E}\bar{Y}\big)\\
					&\quad +\big(\check{\mathcal{F}}^\top+\check{\mathcal{D}}\bar{\hat{\boldsymbol{\Delta}}}_2\big)\mathbb{E}\bar{Y}+\tilde{\mathcal{K}}\big(\bar{Z}-\mathbb{E}\bar{Z}\big)
					+\check{\mathcal{K}}\mathbb{E}\bar{Z}+\tilde{\mathcal{D}}\big(\bar{v}_2-\mathbb{E}\bar{v}_2\big)+\check{\mathcal{D}}\mathbb{E}\bar{v}_2+\tilde{\boldsymbol{\sigma}}+\check{\boldsymbol{\sigma}}\big],
				\end{aligned}\right.
			\end{equation*}
			with $\bar{\eta}_2(T)=\tilde{\boldsymbol{g}}+\check{\boldsymbol{g}}$. Then
			\begin{equation}\label{Eeta}
				\left\{\begin{aligned}
					\mathbb{E}\bar{Y}&=\mathbb{E}\bar{\eta}_2+\Pi_2\mathbb{E}\bar{X},\\
					\mathbb{E}\bar{Z}&=\big(I-P_2\check{\mathcal{K}}\big)^{-1}\big\{\mathbb{E}\bar{\zeta}_2+P_2\big[\big(\check{\mathcal{C}}
					+\check{\mathcal{D}}\bar{\hat{\boldsymbol{\Theta}}}_2\big)\mathbb{E}\bar{X}+\big(\check{\mathcal{F}}^\top+\check{\mathcal{D}}\bar{\hat{\boldsymbol{\Delta}}}_2\big)\mathbb{E}\bar{Y}
					+\check{\mathcal{D}}\mathbb{E}\bar{v}_2+\check{\boldsymbol{\sigma}}\big]\big\},
				\end{aligned}\right.
			\end{equation}
			and
			\begin{equation}\label{eta-Eeta}
				\left\{\begin{aligned}
					\bar{Y}-\mathbb{E}\bar{Y}&=\bar{\eta}_2-\mathbb{E}\bar{\eta}_2+P_2\big(\bar{X}-\mathbb{E}\bar{X}\big),\\
					\bar{Z}-\mathbb{E}\bar{Z}&=\big(I-P_2\tilde{\mathcal{K}}\big)^{-1}\big\{\bar{\zeta}_2-\mathbb{E}\bar{\zeta}_2+P_2\big[\big(\tilde{\mathcal{C}}
					+\tilde{\mathcal{D}}\bar{\boldsymbol{\Theta}}_2\big)\big(\bar{X}-\mathbb{E}\bar{X}\big)\\
					&\qquad+\big(\tilde{\mathcal{F}}^\top+\tilde{\mathcal{D}}\bar{\boldsymbol{\Delta}}_2\big)\big(\bar{Y}-\mathbb{E}\bar{Y}\big)
					+\tilde{\mathcal{D}}\big(\bar{v}_2-\mathbb{E}\bar{v}_2\big)+\tilde{\boldsymbol{\sigma}}\big]\big\}.
				\end{aligned}\right.
			\end{equation}
			Consequently,
			\begin{equation}\label{eta2}
				\begin{aligned}
					&-d\bar{\eta}_2(s)=\Big\{\Big[\dot{P}_2+\tilde{\mathcal{A}}^\top P_2+P_2\tilde{\mathcal{A}}+\tilde{\mathcal{H}}+P_2\tilde{\mathcal{M}}P_2+\big(\tilde{\mathcal{C}}^\top
					+P_2\tilde{\mathcal{F}}\big)\big(I-P_2\tilde{\mathcal{K}}\big)^{-1}P_2\big(\tilde{\mathcal{C}}+\tilde{\mathcal{F}}^\top P_2\big)\\
					&\quad +\big[\tilde{\mathcal{N}}+P_2\tilde{\mathcal{B}}+\big(\tilde{\mathcal{C}}^\top+P_2\tilde{\mathcal{F}}\big)\big(I-P_2\tilde{\mathcal{K}}\big)^{-1}P_2\tilde{\mathcal{D}}\big]\big(\bar{\boldsymbol{\Theta}}_2
					+\bar{\boldsymbol{\Delta}}_2P_2\big)+P_2\bar{\boldsymbol{\Delta}}^\top_2\big\{\tilde{\mathcal{N}}^\top+\tilde{\mathcal{B}}^\top P_2\\
					&\quad +\big[\tilde{R}+\tilde{\mathcal{D}}^\top\big(I-P_2\tilde{\mathcal{K}}\big)^{-1}P_2\tilde{\mathcal{D}}\big]\big(\bar{\boldsymbol{\Theta}}_2+\bar{\boldsymbol{\Delta}}_2P_2\big)
					+\tilde{\mathcal{D}}^\top\big(I-P_2\tilde{\mathcal{K}}\big)^{-1}P_2\big(\tilde{\mathcal{C}}+\tilde{\mathcal{F}}^\top P_2\big)\big\}\Big]\big(\bar{X}-\mathbb{E}\bar{X}\big)\\
					&\quad +\Big[\dot{\Pi}_2+\check{\mathcal{A}}^\top \Pi_2+\Pi_2\check{\mathcal{A}}+\check{\mathcal{H}}+\Pi_2\check{\mathcal{M}}\Pi_2+\big(\check{\mathcal{C}}^\top
					+\Pi_2\check{\mathcal{F}}\big)\big(I-P_2\check{\mathcal{K}}\big)^{-1}P_2\big(\check{\mathcal{C}}+\check{\mathcal{F}}^\top\Pi_2\big)\\
					&\quad +\big[\check{\mathcal{N}}+\Pi_2\check{\mathcal{B}}+\big(\check{\mathcal{C}}^\top+\Pi_2\check{\mathcal{F}}\big)
					\big(I-P_2\check{\mathcal{K}}\big)^{-1}P_2\check{\mathcal{D}}\big]\big(\bar{\hat{\boldsymbol{\Theta}}}_2+\bar{\hat{\boldsymbol{\Delta}}}_2\Pi_2\big)
					+\Pi_2\bar{\hat{\boldsymbol{\Delta}}}^\top_2\big\{\check{\mathcal{N}}^\top+\check{\mathcal{B}}^\top \Pi_2\\
					&\quad +\big[\check{R}+\check{\mathcal{D}}^\top\big(I-P_2\check{\mathcal{K}}\big)^{-1}P_2\check{\mathcal{D}}\big]\big(\bar{\hat{\boldsymbol{\Theta}}}_2+\bar{\hat{\boldsymbol{\Delta}}}_2\Pi_2\big)
					+\check{\mathcal{D}}^\top\big(I-P_2\check{\mathcal{K}}\big)^{-1}P_2\big(\check{\mathcal{C}}+\check{\mathcal{F}}^\top\Pi_2\big)\big\}\Big]\mathbb{E}\bar{X}\\
					&\quad +\big[\tilde{\mathcal{A}}^\top+\big(\tilde{\mathcal{C}}^\top+P_2\tilde{\mathcal{F}}\big)\big(I-P_2\tilde{\mathcal{K}}\big)^{-1}P_2\tilde{\mathcal{F}}^\top
					+P_2\tilde{\mathcal{M}}\big]\big(\bar{\eta}_2-\mathbb{E}\bar{\eta}_2\big)\\
					&\quad +\big[\check{\mathcal{A}}^\top+\big(\check{\mathcal{C}}^\top+\Pi_2\check{\mathcal{F}}\big)\big(I-P_2\check{\mathcal{K}}\big)^{-1}P_2\check{\mathcal{F}}^\top+\Pi_2\check{\mathcal{M}}\big]\mathbb{E}\bar{\eta}_2\\
					&\quad +\big(\tilde{\mathcal{C}}^\top+P_2\tilde{\mathcal{F}}\big)\big(I-P_2\tilde{\mathcal{K}}\big)^{-1}\big(\bar{\zeta}_2-\mathbb{E}\bar{\zeta}_2\big)
					+\big(\check{\mathcal{C}}^\top+\Pi_2\check{\mathcal{F}}\big)\big(I-P_2\check{\mathcal{K}}\big)^{-1}\mathbb{E}\bar{\zeta}_2\\
					&\quad +\big[\tilde{\mathcal{N}}+\big(\tilde{\mathcal{C}}^\top+P_2\tilde{\mathcal{F}}\big)\big(I-P_2\tilde{\mathcal{K}}\big)^{-1}P_2\tilde{\mathcal{D}}
					+P_2\tilde{\mathcal{B}}\big]\big[\bar{\boldsymbol{\Delta}}_2\big(\bar{\eta}_2-\mathbb{E}\bar{\eta}_2\big)+\bar{v}_2-\mathbb{E}\bar{v}_2\big]\\
					&\quad +\big[\check{\mathcal{N}}+\big(\check{\mathcal{C}}^\top+\Pi_2\check{\mathcal{F}}\big)\big(I-P_2\check{\mathcal{K}}\big)^{-1}P_2\check{\mathcal{D}}
					+\Pi_2\check{\mathcal{B}}\big]\big(\bar{\hat{\boldsymbol{\Delta}}}_2\mathbb{E}\bar{\eta}_2+\mathbb{E}\bar{v}_2\big)+P_2\tilde{\boldsymbol{b}}+\Pi_2\check{\boldsymbol{b}}\\
					&\quad +\tilde{\boldsymbol{f}}+\check{\boldsymbol{f}}+\big(\tilde{\mathcal{C}}^\top
					+P_2\tilde{\mathcal{F}}\big)\big(I-P_2\tilde{\mathcal{K}}\big)^{-1}P_2\tilde{\boldsymbol{\sigma}}
					+\big(\check{\mathcal{C}}^\top+\Pi_2\check{\mathcal{F}}\big)\big(I-P_2\check{\mathcal{K}}\big)^{-1}P_2\check{\boldsymbol{\sigma}}\\
					&\quad +P_2\bar{\boldsymbol{\Delta}}^\top_2\Big[ \big[\tilde{\mathcal{B}}^\top+\tilde{\mathcal{D}}^\top\big(I-P_2\tilde{\mathcal{K}}\big)^{-1}P_2\tilde{\mathcal{F}}^\top\big]\big(\bar{\eta}_2
					-\mathbb{E}\bar{\eta}_2\big)+\tilde{\mathcal{D}}^\top\big(I-P_2\tilde{\mathcal{K}}\big)^{-1}\big(\bar{\zeta}_2-\mathbb{E}\bar{\zeta}_2\big)\\
					&\quad +\big[\tilde{R}+\tilde{\mathcal{D}}^\top\big(I-P_2\tilde{\mathcal{K}}\big)^{-1}P_2\tilde{\mathcal{D}}\big]\big[\bar{\boldsymbol{\Delta}}_2\big(\bar{\eta}_2
					-\mathbb{E}\bar{\eta}_2\big)+\bar{v}_2-\mathbb{E}\bar{v}_2\big]+\tilde{\rho}-\mathbb{E}\tilde{\rho}+\tilde{\mathcal{D}}^\top\big(I-P_2\tilde{\mathcal{K}}\big)^{-1}P_2\tilde{\boldsymbol{\sigma}}\Big]\\
					&\quad +\Pi_2\bar{\hat{\boldsymbol{\Delta}}}^\top_2\Big[\big[\check{\mathcal{B}}^\top+\check{\mathcal{D}}^\top\big(I-P_2\check{\mathcal{K}}\big)^{-1}
					P_2\check{\mathcal{F}}^\top\big]\mathbb{E}\bar{\eta}_2+\check{\mathcal{D}}^\top\big(I-P_2\check{\mathcal{K}}\big)^{-1}\mathbb{E}\bar{\zeta}_2\\
					&\quad +\big[\check{R}+\check{\mathcal{D}}^\top\big(I-P_2\check{\mathcal{K}}\big)^{-1}P_2\check{\mathcal{D}}\big]\big[\bar{\hat{\boldsymbol{\Delta}}}_2\mathbb{E}\bar{\eta}_2
					+\mathbb{E}\bar{v}_2\big]+\check{\rho}+\check{\mathcal{D}}^\top\big(I-P_2\check{\mathcal{K}}\big)^{-1}P_2\check{\boldsymbol{\sigma}}\Big] \Big\}ds-\bar{\zeta}_2dW(s).
				\end{aligned}
			\end{equation}
			According to the stationarity condition (\ref{bar(X,Y,Z)---}) and noting (\ref{Sigma2 and Sigma2}), (\ref{Theta,hat(Theta)}), (\ref{Eeta}), (\ref{eta-Eeta}), we have
			\begin{equation*}
				\left\{\begin{aligned}
					0&=\big[\check{\mathcal{B}}^\top+\check{\mathcal{D}}^\top\big(I-P_2\check{\mathcal{K}}\big)^{-1}P_2\check{\mathcal{F}}^\top\big]\mathbb{E}\bar{\eta}_2
					+\check{\mathcal{D}}^\top\big(I-P_2\check{\mathcal{K}}\big)^{-1}\mathbb{E}\bar{\zeta}_2\\
					&\quad +\big[\check{R}+\check{\mathcal{D}}^\top\big(I-P_2\check{\mathcal{K}}\big)^{-1}P_2\check{\mathcal{D}}\big]\big(\bar{\hat{\boldsymbol{\Delta}}}_2\mathbb{E}\bar{\eta}_2
					+\mathbb{E}\bar{v}_2\big)+\check{\mathcal{D}}^\top\big(I-P_2\check{\mathcal{K}}\big)^{-1}P_2\check{\boldsymbol{\sigma}}+\check{\rho},\quad a.e.,\\
					0&=\big[\tilde{\mathcal{B}}^\top+\tilde{\mathcal{D}}^\top\big(I-P_2\tilde{\mathcal{K}}\big)^{-1}P_2\tilde{\mathcal{F}}^\top\big]\big(\bar{\eta}_2-\mathbb{E}\bar{\eta}_2\big)
					+\tilde{\mathcal{D}}^\top\big(I-P_2\tilde{\mathcal{K}}\big)^{-1}\big(\bar{\zeta}_2-\mathbb{E}\bar{\zeta}_2\big)+\tilde{\rho}-\mathbb{E}\tilde{\rho}+\big[\tilde{R}\\
					&\quad +\tilde{\mathcal{D}}^\top\big(I-P_2\tilde{\mathcal{K}}\big)^{-1}P_2\tilde{\mathcal{D}}\big]\big[\bar{\boldsymbol{\Delta}}_2\big(\bar{\eta}_2-\mathbb{E}\bar{\eta}_2\big)
					+\bar{v}_2-\mathbb{E}\bar{v}_2\big]+\tilde{\mathcal{D}}^\top\big(I-P_2\tilde{\mathcal{K}}\big)^{-1}P_2\tilde{\boldsymbol{\sigma}}, \quad a.e.,\,\, \mathbb{P}\mbox{-}a.s..
				\end{aligned}\right.
			\end{equation*}
			Denoting
			\begin{equation*}
				\tilde{\mathbb{V}}_2\equiv\bar{\boldsymbol{\Delta}}_2\big(\bar{\eta}_2-\mathbb{E}\bar{\eta}_2\big)+\bar{v}_2-\mathbb{E}\bar{v}_2,\quad
				\check{\mathbb{V}}_2\equiv\bar{\hat{\boldsymbol{\Delta}}}_2\mathbb{E}\bar{\eta}_2+\mathbb{E}\bar{v}_2,
			\end{equation*}
			we have
			\begin{equation}\label{tildeV2,checkV2}
				\left\{\begin{aligned}		
					\tilde{\mathbb{V}}_2&=-\tilde{\Sigma}_2^{-1}\big\{\big[\tilde{\mathcal{B}}^\top+\tilde{\mathcal{D}}^\top
					\big(I-P_2\tilde{\mathcal{K}}\big)^{-1}P_2\tilde{\mathcal{F}}^\top\big]\big(\bar{\eta}_2-\mathbb{E}\bar{\eta}_2\big)\\
					&\qquad\qquad +\tilde{\mathcal{D}}^\top\big(I-P_2\tilde{\mathcal{K}}\big)^{-1}\big(\bar{\zeta}_2-\mathbb{E}\bar{\zeta}_2\big)
					+\tilde{\rho}-\mathbb{E}\tilde{\rho}+\tilde{\mathcal{D}}^\top\big(I-P_2\tilde{\mathcal{K}}\big)^{-1}P_2\tilde{\boldsymbol{\sigma}}\big\},\\
					\check{\mathbb{V}}_2&=-\check{\Sigma}_2^{-1}\big\{\big[\check{\mathcal{B}}^\top+\check{\mathcal{D}}^\top\big(I-P_2\check{\mathcal{K}}\big)^{-1}P_2\check{\mathcal{F}}^\top\big]\mathbb{E}\bar{\eta}_2\\
					&\qquad\qquad +\check{\mathcal{D}}^\top\big(I-P_2\check{\mathcal{K}}\big)^{-1}\mathbb{E}\bar{\zeta}_2+\check{\mathcal{D}}^\top\big(I-P_2\check{\mathcal{K}}\big)^{-1}P_2\check{\boldsymbol{\sigma}}+\check{\rho}\big\}.
				\end{aligned}\right.
			\end{equation}
			Inserting the above into (\ref{eta2}), we achieve
			\begin{equation}\label{eta2,zeta2}
				\left\{\begin{aligned}
					-d\bar{\eta}_2(s)&=\big\{\big[\tilde{\mathcal{A}}^\top+\big(\tilde{\mathcal{C}}^\top+P_2\tilde{\mathcal{F}}\big)\big(I-P_2\tilde{\mathcal{K}}\big)^{-1}P_2\tilde{\mathcal{F}}^\top
					+P_2\tilde{\mathcal{M}}\big]\big(\bar{\eta}_2-\mathbb{E}\bar{\eta}_2\big)\\
					&\quad\ +\big[\check{\mathcal{A}}^\top+\big(\check{\mathcal{C}}^\top+\Pi_2\check{\mathcal{F}}\big)\big(I-P_2\check{\mathcal{K}}\big)^{-1}P_2\check{\mathcal{F}}^\top+\Pi_2\check{\mathcal{M}}\big]\mathbb{E}\bar{\eta}_2\\
					&\quad\ +\big(\tilde{\mathcal{C}}^\top+P_2\tilde{\mathcal{F}}\big)\big(I-P_2\tilde{\mathcal{K}}\big)^{-1}\big(\bar{\zeta}_2-\mathbb{E}\bar{\zeta}_2\big)
					+\big(\check{\mathcal{C}}^\top+\Pi_2\check{\mathcal{F}}\big)\big(I-P_2\check{\mathcal{K}}\big)^{-1}\mathbb{E}\bar{\zeta}_2\\
					&\quad\ +\big[\tilde{\mathcal{N}}+\big(\tilde{\mathcal{C}}^\top+P_2\tilde{\mathcal{F}}\big)\big(I-P_2\tilde{\mathcal{K}}\big)^{-1}P_2\tilde{\mathcal{D}}
					+P_2\tilde{\mathcal{B}}\big]\tilde{\mathbb{V}}_2+P_2\tilde{\boldsymbol{b}}+\Pi_2\check{\boldsymbol{b}}\\
					&\quad\ +\big[\check{\mathcal{N}}+\big(\check{\mathcal{C}}^\top+\Pi_2\check{\mathcal{F}}\big)\big(I-P_2\check{\mathcal{K}}\big)^{-1}P_2\check{\mathcal{D}}
					+\Pi_2\check{\mathcal{B}}\big]\check{\mathbb{V}}_2+\tilde{\boldsymbol{f}}+\check{\boldsymbol{f}}+\big(\tilde{\mathcal{C}}^\top\\
					&\quad\ +P_2\tilde{\mathcal{F}}\big)\big(I-P_2\tilde{\mathcal{K}}\big)^{-1}P_2\tilde{\boldsymbol{\sigma}}
					+\big(\check{\mathcal{C}}^\top+\Pi_2\check{\mathcal{F}}\big)\big(I-P_2\check{\mathcal{K}}\big)^{-1}P_2\check{\boldsymbol{\sigma}}\big\}ds-\bar{\zeta}_2dW(s),\\
					\bar{\eta}_2(T)&=\tilde{\boldsymbol{g}}+\check{\boldsymbol{g}}.
				\end{aligned}\right.
			\end{equation}
			
			It is worth to point out that, if we consider the outcome of the closed-loop strategy in (\ref{leader closedloop system})-(\ref{lclf}): $u_2=\Theta_2\big(\check{x}-\mathbb{E}\check{x}\big)+\hat{\Theta}_2\mathbb{E}\check{x}+\Delta_2\big(\check{\eta}_1-\mathbb{E}\check{\eta}_1\big)+\hat{\Delta}_2\mathbb{E}\check{\eta}_1+v_2$, it is anticipating. Therefore, inspired by Yong \cite{Yong2002}, we consider a non-anticipating closed-loop optimal control instead:
			\begin{equation*}
				\begin{aligned}
					\bar{u}_2&=\bar{\Theta}_2\big(\bar{x}-\mathbb{E}\bar{x}\big)+\bar{\hat{\Theta}}_2\mathbb{E}\bar{x}+\bar{\Delta}_2\big(\bar{\eta}_1-\mathbb{E}\bar{\eta}_1\big)+\bar{\hat{\Delta}}_2\mathbb{E}\bar{\eta}_1+\bar{v}_2\\
					&=\bar{\boldsymbol{\Theta}}_2\big(\bar{X}-\mathbb{E}\bar{X}\big)+\bar{\hat{\boldsymbol{\Theta}}}_2\mathbb{E}\bar{X}+\bar{\boldsymbol{\Delta}}_2\big(\bar{Y}-\mathbb{E}\bar{Y}\big)
					+\bar{\hat{\boldsymbol{\Delta}}}_2\mathbb{E}\bar{Y}+\bar{v}_2\\
					&=\bar{\boldsymbol{\Theta}}_2\big(\bar{X}-\mathbb{E}\bar{X}\big)+\bar{\hat{\boldsymbol{\Theta}}}_2\mathbb{E}\bar{X}+\bar{\boldsymbol{\Delta}}_2\big[\bar{\eta}_2
					-\mathbb{E}\bar{\eta}_2+P_2\big(\bar{X}-\mathbb{E}\bar{X}\big)\big]+\bar{\hat{\boldsymbol{\Delta}}}_2\big(\mathbb{E}\bar{\eta}_2+\Pi_2\mathbb{E}\bar{X}\big)+\bar{v}_2\\
					&=\big(\bar{\boldsymbol{\Theta}}_2+\bar{\boldsymbol{\Delta}}_2P_2\big)\big(\bar{X}-\mathbb{E}\bar{X}\big)+\big(\bar{\hat{\boldsymbol{\Theta}}}_2
					+\bar{\hat{\boldsymbol{\Delta}}}_2\Pi_2\big)\mathbb{E}\bar{X}+\bar{\boldsymbol{\Delta}}_2\big(\bar{\eta}_2-\mathbb{E}\bar{\eta}_2\big)+\bar{v}_2-\mathbb{E}\bar{v}_2
					+\bar{\hat{\boldsymbol{\Delta}}}_2\mathbb{E}\bar{\eta}_2+\mathbb{E}\bar{v}_2\\
					&=\bar{\boldsymbol{\Theta}}\big(\bar{X}-\mathbb{E}\bar{X}\big)+\bar{\hat{\boldsymbol{\Theta}}}\mathbb{E}\bar{X}+\tilde{\mathbb{V}}_2+\check{\mathbb{V}}_2.
				\end{aligned}
			\end{equation*}
			
			\begin{mydef}
				A triple $(\bar{\boldsymbol{\Theta}}(\cdot),\bar{\hat{\boldsymbol{\Theta}}}(\cdot),\tilde{\mathbb{V}}_2(\cdot)+\check{\mathbb{V}}_2(\cdot)) \in L^2(t,T;\mathbb{R}^{m_2\times 2n}) \times L^2(t,T;\mathbb{R}^{m_2\times 2n}) \times \mathcal{U}_2[t,T]$ is called a leader's non-anticipating closed-loop optimal strategy of Problem (MF-SLQ)$_l$ on $[t,T]$, if
				\begin{equation}
					\begin{aligned}
						&\hat{J}_2\big(t,\xi;\bar{\boldsymbol{\Theta}}\big(\bar{X}-\mathbb{E}\bar{X}\big)+\bar{\hat{\boldsymbol{\Theta}}}\mathbb{E}\bar{X}+\tilde{\mathbb{V}}_2+\check{\mathbb{V}}_2\big)\leq \hat{J}_2(t,\xi;u_2(\cdot)),\\
						&\hspace{2cm}\forall (t,\xi) \in [0,T] \times L^2_{\mathcal{F}_t}(\Omega;\mathbb{R}^n),\,\, \forall u_2(\cdot) \in \mathcal{U}_2[t,T].
					\end{aligned}
				\end{equation}
			\end{mydef}
			
			We have the following result.
			\begin{mythm}\label{Th-cl-l}
				Let (H1)-(H2) hold. Let Riccati equations (\ref{P2}) and (\ref{Pi2}) admit solutions $P_2(\cdot) \in C([t,T],\mathbb{R}^{2n \times 2n})$, $\Pi_2(\cdot) \in C([t,T],\mathbb{R}^{2n \times 2n})$, respectively. If Problem (MF-SLQ)$_l$ admits a non-anticipating closed-loop optimal strategy $(\bar{\boldsymbol{\Theta}}(\cdot),\bar{\hat{\boldsymbol{\Theta}}}(\cdot),\tilde{\mathbb{V}}_2(\cdot)+\check{\mathbb{V}}_2(\cdot))\in L^2(t,T;\mathbb{R}^{m_2\times 2n}) \times L^2(t,T;\mathbb{R}^{m_2\times 2n}) \times \mathcal{U}_2[t,T]$, it admits representation as (\ref{Theta,hat(Theta)}), (\ref{tildeV2,checkV2}) and the outcome is
				\begin{equation}\label{outcome leader}
					\bar{u}_2(\cdot)=\bar{\boldsymbol{\Theta}}(\cdot)\big(\bar{X}-\mathbb{E}\bar{X}\big)(\cdot)+\bar{\hat{\boldsymbol{\Theta}}}(\cdot)\mathbb{E}\bar{X}(\cdot)+\tilde{\mathbb{V}}_2(\cdot)+\check{\mathbb{V}}_2(\cdot),
				\end{equation}
				where $\bar{X}(\cdot)$ is the solution to the following MF-SDE:
				\begin{equation}\label{nonanticipating X}
					\left\{\begin{aligned}
						d\bar{X}(s)&=\big\{\tilde{\mathbb{A}}\big(\bar{X}-\mathbb{E}\bar{X}\big)+\tilde{\mathbb{B}}+\check{\mathbb{A}}\mathbb{E}\bar{X}+\check{\mathbb{B}}\big\}ds\\
						&\quad +\big\{\tilde{\mathbb{C}}\big(\bar{X}-\mathbb{E}\bar{X}\big)+\tilde{\mathbb{D}}+\check{\mathbb{C}}\mathbb{E}\bar{X}+\check{\mathbb{D}}\big\}dW(s),\quad s\in[t,T],\\
						\bar{X}(t)&=\bar{X}_0,
					\end{aligned}\right.
				\end{equation}
				where
				\begin{equation*}
					\begin{aligned}
						\tilde{\mathbb{A}}&:=\tilde{\mathcal{A}}+\tilde{\mathcal{M}}P_2+\tilde{\mathcal{F}}\big(I-P_2\tilde{\mathcal{K}}\big)^{-1}P_2\big(\tilde{\mathcal{C}}
						+\tilde{\mathcal{F}}^\top P_2\big)+\big[\tilde{\mathcal{B}}+\tilde{\mathcal{F}}\big(I-P_2\tilde{\mathcal{K}}\big)^{-1}P_2\tilde{\mathcal{D}}\big]\bar{\boldsymbol{\Theta}},\\
						\check{\mathbb{A}}&:=\check{\mathcal{A}}+\check{\mathcal{M}}\Pi_2+\check{\mathcal{F}}\big(I-P_2\check{\mathcal{K}}\big)^{-1}P_2\big(\check{\mathcal{C}}+\check{\mathcal{F}}^\top\Pi_2\big)
						+\big[\check{\mathcal{B}}+\check{\mathcal{F}}\big(I-P_2\check{\mathcal{K}}\big)^{-1}P_2\check{\mathcal{D}}\big]\bar{\hat{\boldsymbol{\Theta}}},\\
						\tilde{\mathbb{C}}&:=\tilde{\mathcal{C}}+\tilde{\mathcal{F}}^\top P_2+\tilde{\mathcal{K}}\big(I-P_2\tilde{\mathcal{K}}\big)^{-1}P_2\big(\tilde{\mathcal{C}}
						+\tilde{\mathcal{F}}^\top P_2\big)+\big[\tilde{\mathcal{D}}+\tilde{\mathcal{K}}\big(I-P_2\tilde{\mathcal{K}}\big)^{-1}P_2\tilde{\mathcal{D}}\big]\bar{\boldsymbol{\Theta}},\\
						\check{\mathbb{C}}&:=\check{\mathcal{C}}+\check{\mathcal{F}}^\top\Pi_2+\check{\mathcal{K}}\big(I-P_2\check{\mathcal{K}}\big)^{-1}P_2\big(\check{\mathcal{C}}
						+\check{\mathcal{F}}^\top\Pi_2\big)+\big[\check{\mathcal{D}}+\check{\mathcal{K}}\big(I-P_2\check{\mathcal{K}}\big)^{-1}P_2\check{\mathcal{D}}\big]\bar{\hat{\boldsymbol{\Theta}}},\\
						\tilde{\mathbb{B}}&:=\big[\tilde{\mathcal{M}}+\tilde{\mathcal{F}}\big(I-P_2\tilde{\mathcal{K}}\big)^{-1}P_2\tilde{\mathcal{F}}^\top\big]\big(\bar{\eta}_2
						-\mathbb{E}\bar{\eta}_2\big)+\tilde{\mathcal{F}}\big(I-P_2\tilde{\mathcal{K}}\big)^{-1}\big(\bar{\zeta}_2-\mathbb{E}\bar{\zeta}_2\big)\\
						&\qquad+\big[\tilde{\mathcal{B}}+\tilde{\mathcal{F}}\big(I-P_2\tilde{\mathcal{K}}\big)^{-1}P_2\tilde{\mathcal{D}}\big]\tilde{\mathbb{V}}_2+\tilde{\boldsymbol{b}}
						+\tilde{\mathcal{F}}\big(I-P_2\tilde{\mathcal{K}}\big)^{-1}P_2\tilde{\boldsymbol{\sigma}},\\
						\check{\mathbb{B}}&:=\big[\check{\mathcal{M}}+\check{\mathcal{F}}\big(I-P_2\check{\mathcal{K}}\big)^{-1}P_2\check{\mathcal{F}}^\top\big]\mathbb{E}\bar{\eta}_2
						+\check{\mathcal{F}}\big(I-P_2\check{\mathcal{K}}\big)^{-1}\mathbb{E}\bar{\zeta}_2\\
						&\qquad+\big[\check{\mathcal{B}}+\check{\mathcal{F}}\big(I-P_2\check{\mathcal{K}}\big)^{-1}P_2\check{\mathcal{D}}\big]\check{\mathbb{V}}_2
						+\check{\mathcal{F}}\big(I-P_2\check{\mathcal{K}}\big)^{-1}P_2\check{\boldsymbol{\sigma}}+\check{\boldsymbol{b}},\\
						\tilde{\mathbb{D}}&:=\big[\tilde{\mathcal{F}}^\top+\tilde{\mathcal{K}}\big(I-P_2\tilde{\mathcal{K}}\big)^{-1}P_2\tilde{\mathcal{F}}^\top\big]\big(\bar{\eta}_2-\mathbb{E}\bar{\eta}_2\big)
						+\tilde{\mathcal{K}}\big(I-P_2\tilde{\mathcal{K}}\big)^{-1}\big(\bar{\zeta}_2-\mathbb{E}\bar{\zeta}_2\big)\\
						&\qquad+\big[\tilde{\mathcal{D}}+\tilde{\mathcal{K}}\big(I-P_2\tilde{\mathcal{K}}\big)^{-1}P_2\tilde{\mathcal{D}}\big]\tilde{\mathbb{V}}_2
						+\big[I+\tilde{\mathcal{K}}\big(I-P_2\tilde{\mathcal{K}}\big)^{-1}P_2\big]\tilde{\boldsymbol{\sigma}},\\
						\check{\mathbb{D}}&:=\big[\check{\mathcal{F}}^\top+\check{\mathcal{K}}\big(I-P_2\check{\mathcal{K}}\big)^{-1}P_2\check{\mathcal{F}}^\top\big]\mathbb{E}\bar{\eta}_2
						+\check{\mathcal{K}}\big(I-P_2\check{\mathcal{K}}\big)^{-1}\mathbb{E}\bar{\zeta}_2\\
						&\qquad+\big[\check{\mathcal{D}}+\check{\mathcal{K}}\big(I-P_2\check{\mathcal{K}}\big)^{-1}P_2\check{\mathcal{D}}\big]\check{\mathbb{V}}_2
						+\big[I+\check{\mathcal{K}}\big(I-P_2\check{\mathcal{K}}\big)^{-1}P_2\big]\check{\boldsymbol{\sigma}}.
					\end{aligned}
				\end{equation*}
			\end{mythm}
			
			\begin{proof}
				We only need to give the equation that $\bar{X}(\cdot)$ satisfies. Putting (\ref{Theta,hat(Theta)}), (\ref{tildeV2,checkV2}) into the equation of $\bar{X}(\cdot)$ in (\ref{bar(X,Y,Z)}) and noticing (\ref{Eeta}), (\ref{eta-Eeta}), we get (\ref{nonanticipating X}).
			\end{proof}
			
			The following theorem gives an expression of the value function of the leader.
			\begin{mythm}\label{leader value function}
				Let (H1)-(H2) hold. If Problem (MF-SLQ)$_l$ is closed-loop solvable, the value function of the leader is
				\begin{equation}\label{value functin of the leader}
					\begin{aligned}
						&V_2(t,\xi)=\mathbb{E}\biggl\{\big\langle \Gamma_1(t)\big(\bar{X}_0-\mathbb{E}\bar{X}_0\big),\bar{X}_0-\mathbb{E}\bar{X}_0 \big\rangle+\big\langle \Gamma_2(t)\mathbb{E}\bar{X}_0,\mathbb{E}\bar{X}_0 \big\rangle
						+2\big\langle \alpha_1(t),\bar{X}_0-\mathbb{E}\bar{X}_0 \big\rangle\\
						& +2\big\langle \alpha_2(t),\mathbb{E}\bar{X}_0 \big\rangle+\mathcal{L}+\int_t^T\big[ \tilde{\mathbb{D}}^\top\Gamma_1\tilde{\mathbb{D}}
						+\check{\mathbb{D}}^\top\Gamma_1\check{\mathbb{D}}+2\tilde{\mathbb{B}}^\top\alpha_1+2\big(\tilde{\mathbb{D}}
						+\check{\mathbb{D}}\big)^\top\beta_1+2\check{\mathbb{B}}^\top\alpha_2+\tilde{\mathbb{L}}+\check{\mathbb{L}}\big]ds\biggr\},
					\end{aligned}
				\end{equation}
				where $\big(\Gamma_1(\cdot),\Gamma_2(\cdot)\big) \in C([t,T],\mathbb{S}^{2n}) \times C([t,T],\mathbb{S}^{2n})$ is the solution of the following Lyapunov equations:
				\begin{equation}\label{Lyapunov equations}
					\left\{\begin{aligned}
						&0=\dot{\Gamma}_1+\tilde{\mathbb{A}}^\top\Gamma_1+\Gamma_1\tilde{\mathbb{A}}+\tilde{\mathbb{C}}^\top\Gamma_1\tilde{\mathbb{C}}+\tilde{\mathbb{Q}},\quad \Gamma_1(T)=\tilde{\mathbb{H}},\\
						&0=\dot{\Gamma}_2+\check{\mathbb{A}}^\top\Gamma_2+\Gamma_2\check{\mathbb{A}}+\check{\mathbb{C}}^\top\Gamma_1\check{\mathbb{C}}+\check{\mathbb{Q}},\quad \Gamma_2(T)=\check{\mathbb{H}},
					\end{aligned}\right.
				\end{equation}
				$(\alpha_1(\cdot),\beta_1(\cdot)) \in L^2_{\mathbb{F}}(t,T;\mathbb{R}^{2n})\times L^2_{\mathbb{F}}(t,T;\mathbb{R}^{2n})$ and $(\alpha_2(\cdot),\beta_2(\cdot))\in L^2_{\mathbb{F}}(t,T;\mathbb{R}^{2n})\times L^2_{\mathbb{F}}(t,T;\mathbb{R}^{2n})$ are adapted solutions to following two BSDEs:
				\begin{equation}\label{BSDE of the leader-1}
					\left\{\begin{aligned}
						&-d\alpha_1(s)=\big[\tilde{\mathbb{A}}^\top\alpha_1+\tilde{\mathbb{C}}^\top\beta_1+\tilde{\mathbb{S}}+\Gamma_1\tilde{\mathbb{B}}+\tilde{\mathbb{C}}^\top\Gamma_1\tilde{\mathbb{D}}\big]ds-\beta_1dW(s),\\
						&\alpha_1(T)=\tilde{h},
					\end{aligned}\right.
				\end{equation}
				\begin{equation}\label{BSDE of the leader-2}
					\left\{\begin{aligned}
						&-d\alpha_2(s)=\big[\check{\mathbb{A}}^\top\alpha_2+\check{\mathbb{S}}+\Gamma_2\check{\mathbb{B}}+\check{\mathbb{C}}^\top\Gamma_1\check{\mathbb{D}}+\check{\mathbb{C}}^\top\beta_1\big]ds-\beta_2dW(s),\\
						&\alpha_2(T)=\check{h},
					\end{aligned}\right.
				\end{equation}
				respectively, where
				\begin{equation*}
					\begin{aligned}
						&\tilde{\mathbb{H}}:=\mathbb{M}_1^\top G^2\mathbb{M}_1,\,\,\,\check{\mathbb{H}}:=\mathbb{M}_1^\top\big(G^2+\hat{G}^2\big)\mathbb{M}_1,\,\,\,\tilde{h}:=\mathbb{M}_1^\top g^2,\,\,\,
						\check{h}:=\mathbb{M}_1^\top\big(g^2+\hat{g}^2\big),\\
						&\tilde{\mathbb{Q}}:=\mathbb{M}^\top_1\tilde{Q}_{11}\mathbb{M}_1+P^\top_2\mathbb{M}^\top_2\tilde{Q}_{12}\mathbb{M}_1+\mathbb{M}^\top_1\tilde{Q}^\top_{12}\mathbb{M}_2P_2
						+\mathbb{M}^\top_1\tilde{Q}^\top_{13}\mathbb{M}_2\big(I-P_2\tilde{\mathcal{K}}\big)^{-1}P_2\big(\tilde{\mathcal{C}}+\tilde{\mathcal{F}}^\top P_2\\
						&\qquad +\tilde{\mathcal{D}}\bar{\boldsymbol{\Theta}}\big)+\big(\tilde{\mathcal{C}}+\tilde{\mathcal{F}}^\top P_2+\tilde{\mathcal{D}}\bar{\boldsymbol{\Theta}}\big)^\top P_2^\top
						\big(I-\tilde{\mathcal{K}}^\top P_2^\top \big)^{-1}\mathbb{M}^\top_2\tilde{Q}_{13}\mathbb{M}_1+P_2^\top\mathbb{M}^\top_2\tilde{Q}_{22}\mathbb{M}_2P_2\\
						&\qquad+P^\top_2\mathbb{M}^\top_2\tilde{Q}^\top_{23}\mathbb{M}_2\big(I-P_2\tilde{\mathcal{K}}\big)^{-1}P_2\big(\tilde{\mathcal{C}}
						+\tilde{\mathcal{F}}^\top P_2+\tilde{\mathcal{D}}\bar{\boldsymbol{\Theta}}\big)+\bar{\boldsymbol{\Theta}}^\top\tilde{S}_2\mathbb{M}_2P_2+P^\top_2\mathbb{M}^\top_2\tilde{S}^\top_2\bar{\boldsymbol{\Theta}}\\
						&\qquad+\big(\tilde{\mathcal{C}}+\tilde{\mathcal{F}}^\top P_2+\tilde{\mathcal{D}}\bar{\boldsymbol{\Theta}}\big)^\top P_2^\top
						\big(I-\tilde{\mathcal{K}}^\top P_2^\top \big)^{-1}\mathbb{M}^\top_2\tilde{Q}_{23}\mathbb{M}_2P_2+\bar{\boldsymbol{\Theta}}^\top\tilde{S}_1\mathbb{M}_1+\mathbb{M}^\top_1\tilde{S}^\top_1\bar{\boldsymbol{\Theta}}\\
						&\qquad+ \big(\tilde{\mathcal{C}}+\tilde{\mathcal{F}}^\top P_2+\tilde{\mathcal{D}}\bar{\boldsymbol{\Theta}}\big)^\top P_2^\top
						\big(I-\tilde{\mathcal{K}}^\top P_2^\top \big)^{-1}\mathbb{M}^\top_2\tilde{Q}_{33}\mathbb{M}_2\big(I-P_2\tilde{\mathcal{K}}\big)^{-1}P_2\big(\tilde{\mathcal{C}}
						+\tilde{\mathcal{F}}^\top P_2+\tilde{\mathcal{D}}\bar{\boldsymbol{\Theta}}\big)\\
						&\qquad +\big(\tilde{\mathcal{C}}+\tilde{\mathcal{F}}^\top P_2+\tilde{\mathcal{D}}\bar{\boldsymbol{\Theta}}\big)^\top P_2^\top
						\big(I-\tilde{\mathcal{K}}^\top P_2^\top \big)^{-1}\mathbb{M}^\top_2\tilde{S}^\top_3\bar{\boldsymbol{\Theta}}\\
						&\qquad+\bar{\boldsymbol{\Theta}}^\top\tilde{S}_3\mathbb{M}_2\big(I-P_2\tilde{\mathcal{K}}\big)^{-1}P_2\big(\tilde{\mathcal{C}}
						+\tilde{\mathcal{F}}^\top P_2+\tilde{\mathcal{D}}\bar{\boldsymbol{\Theta}}\big)+\bar{\boldsymbol{\Theta}}^\top\tilde{R}\bar{\boldsymbol{\Theta}},\\
						&\tilde{\mathbb{S}}:=\big[\mathbb{M}^\top_1\tilde{Q}^\top_{12}\mathbb{M}_2+P^\top_2\mathbb{M}^\top_2\tilde{Q}_{22}\mathbb{M}_2+\bar{\boldsymbol{\Theta}}^\top\tilde{S}_2\mathbb{M}_2\\
						&\qquad\qquad +\big(\tilde{\mathcal{C}}+\tilde{\mathcal{F}}^\top P_2+\tilde{\mathcal{D}}\bar{\boldsymbol{\Theta}}\big)^\top P_2^\top
						\big(I-\tilde{\mathcal{K}}^\top P_2^\top \big)^{-1}\mathbb{M}^\top_2\tilde{Q}_{23}\mathbb{M}_2\big]\big(\bar{\eta}_2-\mathbb{E}\bar{\eta}_2\big)\\
						&\qquad +\big[P^\top_2\mathbb{M}^\top_2\tilde{Q}^\top_{23}+\mathbb{M}^\top_1\tilde{Q}^\top_{13}+\big(\tilde{\mathcal{C}}
						+\tilde{\mathcal{F}}^\top P_2+\tilde{\mathcal{D}}\bar{\boldsymbol{\Theta}}\big)^\top P_2^\top\big(I-\tilde{\mathcal{K}}^\top P_2^\top \big)^{-1}\mathbb{M}^\top_2\tilde{Q}_{33}\\
						&\qquad\qquad+\bar{\boldsymbol{\Theta}}^\top\tilde{S}_3\big]\mathbb{M}_2\big(I-P_2\tilde{\mathcal{K}}\big)^{-1}
						\big[ P_2\tilde{\mathcal{F}}^\top\big(\bar{\eta}_2-\mathbb{E}\bar{\eta}_2\big)+\bar{\zeta}_2-\mathbb{E}\bar{\zeta}_2+P_2\tilde{\mathcal{D}}\tilde{\mathbb{V}}_2+P_2\tilde{\boldsymbol{\sigma}}\big]\\
						&\qquad +\big[\mathbb{M}^\top_1\tilde{S}^\top_1+P^\top_2\mathbb{M}^\top_2\tilde{S}^\top_2+\bar{\boldsymbol{\Theta}}^\top\tilde{R}+\big(\tilde{\mathcal{C}}
						+\tilde{\mathcal{F}}^\top P_2+\tilde{\mathcal{D}}\bar{\boldsymbol{\Theta}}\big)^\top P_2^\top\big(I-\tilde{\mathcal{K}}^\top P_2^\top \big)^{-1}\mathbb{M}^\top_2\tilde{S}^\top_3\big]\tilde{\mathbb{V}}_2\\
						&\qquad +\mathbb{M}^\top_1\big(\tilde{q}_1-\mathbb{E}\tilde{q}_1\big)+P^\top_2\mathbb{M}^\top_2\big(\tilde{q}_2-\mathbb{E}\tilde{q}_2\big)+\bar{\boldsymbol{\Theta}}^\top\big(\tilde{\rho}-\mathbb{E}\tilde{\rho}\big)\\
						&\qquad +\big(\tilde{\mathcal{C}}+\tilde{\mathcal{F}}^\top P_2+\tilde{\mathcal{D}}\bar{\boldsymbol{\Theta}}\big)^\top P_2^\top
						\big(I-\tilde{\mathcal{K}}^\top P_2^\top \big)^{-1}\mathbb{M}^\top_2\big(\tilde{q}_3-\mathbb{E}\tilde{q}_3\big),\\
						&\tilde{\mathbb{L}}:=2\big\langle \mathbb{M}^\top_2\tilde{Q}_{23}\mathbb{M}_2\big(\bar{\eta}_2-\mathbb{E}\bar{\eta}_2\big)+\mathbb{M}_2^\top\big(\tilde{q}_3-\mathbb{E}\tilde{q}_3\big)
						+\mathbb{M}^\top_2\tilde{S}^\top_3\tilde{\mathbb{V}}_2,\\
						&\qquad\qquad\qquad \big(I-P_2\tilde{\mathcal{K}}\big)^{-1}\big[ P_2\tilde{\mathcal{F}}^\top\big(\bar{\eta}_2-\mathbb{E}\bar{\eta}_2\big)
						+\bar{\zeta}_2-\mathbb{E}\bar{\zeta}_2+P_2\tilde{\mathcal{D}}\tilde{\mathbb{V}}_2+P_2\tilde{\boldsymbol{\sigma}}\big] \big\rangle\\
						&\qquad +\big\langle \mathbb{M}^\top_2\tilde{Q}_{33}\mathbb{M}_2\big(I-P_2\tilde{\mathcal{K}}\big)^{-1}
						\big[ P_2\tilde{\mathcal{F}}^\top\big(\bar{\eta}_2-\mathbb{E}\bar{\eta}_2\big)+\bar{\zeta}_2-\mathbb{E}\bar{\zeta}_2+P_2\tilde{\mathcal{D}}\tilde{\mathbb{V}}_2+P_2\tilde{\boldsymbol{\sigma}}\big],\\
						&\qquad\qquad\qquad \big(I-P_2\tilde{\mathcal{K}}\big)^{-1}\big[ P_2\tilde{\mathcal{F}}^\top\big(\bar{\eta}_2-\mathbb{E}\bar{\eta}_2\big)+\bar{\zeta}_2
						-\mathbb{E}\bar{\zeta}_2+P_2\tilde{\mathcal{D}}\tilde{\mathbb{V}}_2+P_2\tilde{\boldsymbol{\sigma}}\big] \big\rangle\\
						&\qquad+\big\langle \mathbb{M}^\top_2\tilde{Q}_{22}\mathbb{M}_2\big(\bar{\eta}_2-\mathbb{E}\bar{\eta}_2\big),\bar{\eta}_2-\mathbb{E}\bar{\eta}_2\big\rangle
						+2\big\langle\tilde{S}_2\mathbb{M}_2\big(\bar{\eta}_2-\mathbb{E}\bar{\eta}_2\big),\tilde{\mathbb{V}}_2 \big\rangle+\big\langle \tilde{R}\tilde{\mathbb{V}}_2,\tilde{\mathbb{V}}_2 \big\rangle\\
						&\qquad +2\big\langle \mathbb{M}^\top_2\big(\tilde{q}_2-\mathbb{E}\tilde{q}_2\big),\bar{\eta}_2-\mathbb{E}\bar{\eta}_2 \big\rangle
						+2\big\langle \tilde{\rho}-\mathbb{E}\tilde{\rho},\tilde{\mathbb{V}}_2 \big\rangle,\\
						&\check{\mathbb{L}}:=2\big\langle \mathbb{M}^\top_2\check{Q}_{23}\mathbb{M}_2\mathbb{E}\bar{\eta}_2+\mathbb{M}^\top_2\check{S}^\top_3\check{\mathbb{V}}_2
						+\mathbb{M}^\top_2\check{q}_3 ,\big(I-P_2\check{\mathcal{K}}\big)^{-1}\big[ P_2\check{\mathcal{F}}^\top\mathbb{E}\bar{\eta}_2+\mathbb{E}\bar{\zeta}_2+P_2\check{\mathcal{D}}\check{\mathbb{V}}_2
						+P_2\check{\boldsymbol{\sigma}}\big] \big\rangle\\
						&\qquad +\big\langle \mathbb{M}^\top_2\check{Q}_{33}\mathbb{M}_2\big(I-P_2\check{\mathcal{K}}\big)^{-1}\big[ P_2\check{\mathcal{F}}^\top\mathbb{E}\bar{\eta}_2
						+\mathbb{E}\bar{\zeta}_2+P_2\check{\mathcal{D}}\check{\mathbb{V}}_2+P_2\check{\boldsymbol{\sigma}}\big],\\
						&\qquad\qquad\quad \big(I-P_2\check{\mathcal{K}}\big)^{-1}\big[ P_2\check{\mathcal{F}}^\top\mathbb{E}\bar{\eta}_2+\mathbb{E}\bar{\zeta}_2
						+P_2\check{\mathcal{D}}\check{\mathbb{V}}_2+P_2\check{\boldsymbol{\sigma}}\big] \big\rangle\\
						&\qquad +\big\langle \mathbb{M}^\top_2\check{Q}_{22}\mathbb{M}_2\mathbb{E}\bar{\eta}_2,\mathbb{E}\bar{\eta}_2 \big\rangle
						+2\big\langle \check{S}_2\mathbb{M}_2\mathbb{E}\bar{\eta}_2,\check{\mathbb{V}}_2 \big\rangle+\big\langle \check{R}\check{\mathbb{V}}_2,\check{\mathbb{V}}_2 \big\rangle
						+2\big\langle \mathbb{M}^\top_2\check{q}_2,\mathbb{E}\bar{\eta}_2 \big\rangle+2\big\langle \check{\rho},\check{\mathbb{V}}_2 \big\rangle,\\
						&\check{\mathbb{Q}}:=\mathbb{M}^\top_1\check{Q}_{11}\mathbb{M}_1+\Pi_2^\top\mathbb{M}^\top_2\check{Q}_{12}\mathbb{M}_1+\mathbb{M}^\top_1\check{Q}^\top_{12}\mathbb{M}_2\Pi_2\\
						&\qquad +\big(\check{\mathcal{C}}+\check{\mathcal{F}}^\top\Pi_2+\check{\mathcal{D}}\bar{\hat{\boldsymbol{\Theta}}}\big)^\top P_2^\top
						\big(I-\check{\mathcal{K}}^\top P_2^\top \big)^{-1}\mathbb{M}^\top_2\check{Q}_{13}\mathbb{M}_1\\
						&\qquad +\mathbb{M}^\top_1\check{Q}^\top_{13}\mathbb{M}_2\big(I-P_2\check{\mathcal{K}}\big)^{-1}P_2\big(\check{\mathcal{C}}+\check{\mathcal{F}}^\top\Pi_2
						+\check{\mathcal{D}}\bar{\hat{\boldsymbol{\Theta}}}\big)+\Pi_2^\top\mathbb{M}^\top_2\check{Q}_{22}\mathbb{M}_2\Pi_2\\
						&\qquad +\Pi^\top_2\mathbb{M}^\top_2\check{Q}^\top_{23}\mathbb{M}_2\big(I-P_2\check{\mathcal{K}}\big)^{-1}P_2\big(\check{\mathcal{C}}+\check{\mathcal{F}}^\top\Pi_2
						+\check{\mathcal{D}}\bar{\hat{\boldsymbol{\Theta}}}\big)+\bar{\hat{\boldsymbol{\Theta}}}^\top\check{S}_2\mathbb{M}_2\Pi_2+\Pi_2^\top\mathbb{M}^\top_2\check{S}^\top_2\bar{\hat{\boldsymbol{\Theta}}}\\
						&\qquad +\big(\check{\mathcal{C}}+\check{\mathcal{F}}^\top\Pi_2+\check{\mathcal{D}}\bar{\hat{\boldsymbol{\Theta}}}\big)^\top P_2^\top
						\big(I-\check{\mathcal{K}}^\top P_2^\top \big)^{-1}\mathbb{M}^\top_2\check{Q}_{23}\mathbb{M}_2\Pi_2+\bar{\hat{\boldsymbol{\Theta}}}^\top\check{S}_1\mathbb{M}_1
						+\mathbb{M}^\top_1\check{S}^\top_1\bar{\hat{\boldsymbol{\Theta}}}\\
						&\qquad +\big(\check{\mathcal{C}}+\check{\mathcal{F}}^\top\Pi_2+\check{\mathcal{D}}\bar{\hat{\boldsymbol{\Theta}}}\big)^\top P_2^\top
						\big(I-\check{\mathcal{K}}^\top P_2^\top \big)^{-1}\mathbb{M}^\top_2\check{Q}_{33}\mathbb{M}_2\big(I-P_2\check{\mathcal{K}}\big)^{-1}P_2\big(\check{\mathcal{C}}
						+\check{\mathcal{F}}^\top\Pi_2+\check{\mathcal{D}}\bar{\hat{\boldsymbol{\Theta}}}\big)\\	
					\end{aligned}
				\end{equation*}
				\begin{equation*}
					\begin{aligned}
						&\qquad+\bar{\hat{\boldsymbol{\Theta}}}^\top\check{S}_3\mathbb{M}_2\big(I-P_2\check{\mathcal{K}}\big)^{-1}P_2\big(\check{\mathcal{C}}+\check{\mathcal{F}}^\top\Pi_2+\check{\mathcal{D}}\bar{\hat{\boldsymbol{\Theta}}}\big)\\
						&\qquad +\big(\check{\mathcal{C}}+\check{\mathcal{F}}^\top\Pi_2+\check{\mathcal{D}}\bar{\hat{\boldsymbol{\Theta}}}\big)^\top P_2^\top
						\big(I-\check{\mathcal{K}}^\top P_2^\top \big)^{-1}\mathbb{M}^\top_2\check{S}^\top_3\bar{\hat{\boldsymbol{\Theta}}}+\bar{\hat{\boldsymbol{\Theta}}}^\top\check{R}\bar{\hat{\boldsymbol{\Theta}}},\\
						&\check{\mathbb{S}}:=\big[\mathbb{M}^\top_1\check{Q}^\top_{12}\mathbb{M}_2+\Pi_2^\top\mathbb{M}^\top_2\check{Q}_{22}\mathbb{M}_2+\big(\check{\mathcal{C}}
						+\check{\mathcal{F}}^\top\Pi_2+\check{\mathcal{D}}\bar{\hat{\boldsymbol{\Theta}}}\big)^\top P_2^\top\big(I-\check{\mathcal{K}}^\top P_2^\top \big)^{-1} \mathbb{M}^\top_2\check{Q}_{23}\mathbb{M}_2\\
						&\qquad +\bar{\hat{\boldsymbol{\Theta}}}^\top\check{S}_2\mathbb{M}_2\big]\mathbb{E}\bar{\eta}_2
						+\big[\mathbb{M}_1\check{Q}^\top_{13}+\Pi_2^\top\mathbb{M}_2^\top\check{Q}^\top_{23}+\bar{\hat{\boldsymbol{\Theta}}}^\top\check{S}_3\big]\mathbb{M}_2\big(I-P_2\check{\mathcal{K}}\big)^{-1}\\
						&\qquad \times \big[ P_2\check{\mathcal{F}}^\top\mathbb{E}\bar{\eta}_2+\mathbb{E}\bar{\zeta}_2+P_2\check{\mathcal{D}}\check{\mathbb{V}}_2+P_2\check{\boldsymbol{\sigma}}\big]
						+\big(\check{\mathcal{C}}+\check{\mathcal{F}}^\top\Pi_2+\check{\mathcal{D}}\bar{\hat{\boldsymbol{\Theta}}}\big)^\top P_2^\top\big(I-\check{\mathcal{K}}^\top P_2^\top \big)^{-1}\\
						&\qquad \times\mathbb{M}^\top_2\check{Q}_{33}\mathbb{M}_2\big(I-P_2\check{\mathcal{K}}\big)^{-1}
						\big[ P_2\check{\mathcal{F}}^\top\mathbb{E}\bar{\eta}_2+\mathbb{E}\bar{\zeta}_2+P_2\check{\mathcal{D}}\check{\mathbb{V}}_2+P_2\check{\boldsymbol{\sigma}}\big]\\
						&\qquad +\big[\mathbb{M}^\top_1\check{S}^\top_1+\Pi_2^\top\mathbb{M}^\top_2\check{S}^\top_2+\bar{\hat{\boldsymbol{\Theta}}}^\top\check{R}
						+\big(\check{\mathcal{C}}+\check{\mathcal{F}}^\top\Pi_2+\check{\mathcal{D}}\bar{\hat{\boldsymbol{\Theta}}}\big)^\top P_2^\top
						\big(I-\check{\mathcal{K}}^\top P_2^\top \big)^{-1}\mathbb{M}^\top_2\check{S}^\top_3\big]\check{\mathbb{V}}_2\\
						&\qquad +\mathbb{M}^\top_1\check{q}_1+\Pi^\top_2\mathbb{M}^\top_2\check{q}_2+\big(\check{\mathcal{C}}+\check{\mathcal{F}}^\top\Pi_2
						+\check{\mathcal{D}}\bar{\hat{\boldsymbol{\Theta}}}\big)^\top P_2^\top\big(I-\check{\mathcal{K}}^\top P_2^\top \big)^{-1}\mathbb{M}^\top_2\check{q}_3+\bar{\hat{\boldsymbol{\Theta}}}^\top\check{\rho}.\\
					\end{aligned}
				\end{equation*}
			\end{mythm}
			
			\begin{proof}
				By the definition of leader's value function, we can get
				\begin{equation*}
					\begin{aligned}
						&V_2(t,\xi)=\mathbb{E}\bigg\{ \big\langle G^2\big(\bar{x}(T)-\mathbb{E}\bar{x}(T)\big),\bar{x}(T)-\mathbb{E}\bar{x}(T) \big\rangle
						+\big\langle \boldsymbol{G}^2\mathbb{E}\bar{x}(T),\mathbb{E}\bar{x}(T)\big\rangle+2\big\langle g^2,\bar{x}(T)-\mathbb{E}\bar{x}(T) \big\rangle\\
						&\quad +2\big\langle \boldsymbol{g}^2,\mathbb{E}\bar{x}(T)\big\rangle
						+\int_t^T \Big[\big\langle \tilde{Q}_{11}\big(\bar{x}-\mathbb{E}\bar{x}\big),\bar{x}-\mathbb{E}\bar{x}\big\rangle
						+2\big\langle\tilde{Q}_{12}\big(\bar{x}-\mathbb{E}\bar{x}\big),\bar{\eta}_1-\mathbb{E}\bar{\eta}_1 \big\rangle\\
						&\quad +2\big\langle \tilde{Q}_{13}\big(\bar{x}-\mathbb{E}\bar{x}\big),\bar{\zeta}_1-\mathbb{E}\bar{\zeta}_1 \big\rangle
						+\big\langle \tilde{Q}_{22}\big(\bar{\eta}_1-\mathbb{E}\bar{\eta}_1\big),\bar{\eta}_1-\mathbb{E}\bar{\eta}_1\big\rangle
						+2\big\langle \tilde{Q}_{23}\big(\bar{\eta}_1-\mathbb{E}\bar{\eta}_1\big),\bar{\zeta}_1-\mathbb{E}\bar{\zeta}_1 \big\rangle\\
						&\quad +\big\langle \tilde{Q}_{33}\big(\bar{\zeta}_1-\mathbb{E}\bar{\zeta}_1\big),\bar{\zeta}_1-\mathbb{E}\bar{\zeta}_1 \big\rangle
						+2\big\langle \tilde{S}_1\big(\bar{x}-\mathbb{E}\bar{x}\big),\bar{u}_2-\mathbb{E}\bar{u}_2 \big\rangle
						+2\big\langle \tilde{S}_2\big(\bar{\eta}_1-\mathbb{E}\bar{\eta}_1\big),\bar{u}_2-\mathbb{E}\bar{u}_2 \big\rangle\\
						&\quad +2\big\langle \tilde{S}_3\big(\bar{\zeta}_1-\mathbb{E}\bar{\zeta}_1\big),\bar{u}_2-\mathbb{E}\bar{u}_2 \big\rangle
						+\big\langle \tilde{R}\big(\bar{u}_2-\mathbb{E}\bar{u}_2\big),\bar{u}_2-\mathbb{E}\bar{u}_2 \big\rangle
						+2\big\langle \tilde{q}_1-\mathbb{E}\tilde{q}_1,\bar{x}-\mathbb{E}\bar{x}  \big\rangle\\
						&\quad +2\big\langle \tilde{q}_2-\mathbb{E}\tilde{q}_2,\bar{\eta}_1-\mathbb{E}\bar{\eta}_1 \big\rangle
						+2\big\langle \tilde{q}_3-\mathbb{E}\tilde{q}_3,\bar{\zeta}_1-\mathbb{E}\bar{\zeta}_1\big\rangle
						+2\big\langle \tilde{\rho}-\mathbb{E}\tilde{\rho},\bar{u}_2-\mathbb{E}\bar{u}_2 \big\rangle+\big\langle \check{Q}_{11}\mathbb{E}\bar{x},\mathbb{E}\bar{x} \big\rangle\\
						&\quad +2\big\langle \check{Q}_{12}\mathbb{E}\bar{x},\mathbb{E}\bar{\eta}_1 \big\rangle+2\big\langle \check{Q}_{13}\mathbb{E}\bar{x},\mathbb{E}\bar{\zeta}_1 \big\rangle
						+\big\langle \check{Q}_{22}\mathbb{E}\bar{\eta}_1,\mathbb{E}\bar{\eta}_1 \big\rangle
						+2\big\langle \check{Q}_{23}\mathbb{E}\bar{\eta}_1,\mathbb{E}\bar{\zeta}_1 \big\rangle
						+\big\langle \check{Q}_{33}\mathbb{E}\bar{\zeta}_1,\mathbb{E}\bar{\zeta}_1 \big\rangle\\
						&\quad+2\big\langle \check{S}_1\mathbb{E}\bar{x},\mathbb{E}\bar{u}_2\big\rangle+2\big\langle \check{S}_2\mathbb{E}\bar{\eta}_1,\mathbb{E}\bar{u}_2\big\rangle
						+2\big\langle \check{S}_3\mathbb{E}\bar{\zeta}_1,\mathbb{E}\bar{u}_2\big\rangle+\big\langle \check{R}\mathbb{E}\bar{u}_2,,\mathbb{E}\bar{u}_2 \big\rangle
						+2\big\langle \check{q}_1,\mathbb{E}\bar{x} \big\rangle\\
						&\quad +2\big\langle \check{q}_2,\mathbb{E}\bar{\eta}_1 \big\rangle+2\big\langle \check{q}_3,\mathbb{E}\bar{\zeta}_1\big\rangle+2\big\langle \check{\rho},\mathbb{E}\bar{u}_2 \big\rangle \Big]ds+\mathcal{L}\bigg\}.
					\end{aligned}
				\end{equation*}
				We set
				\begin{equation*}
					\mathbb{M}_1:=\left(\begin{matrix} 1 & 0 \end{matrix}\right),\quad
					\mathbb{M}_2:=\left(\begin{matrix} 0 & 1 \end{matrix}\right),
				\end{equation*}
				then
				\begin{equation*}
					\bar{x}=\mathbb{M}_1\bar{X},\quad \bar{\eta}_1=\mathbb{M}_2\bar{Y},\quad \bar{\zeta}_1=\mathbb{M}_2\bar{Z}.
				\end{equation*}
				
				Noticing (\ref{outcome leader}), (\ref{Eeta}) and (\ref{eta-Eeta}),we have
				\begin{equation*}
					\begin{aligned}
						&V_2(t,\xi)=\mathbb{E}\bigg\{ \big\langle \mathbb{M}_1^\top G^2\mathbb{M}_1\big(\bar{X}(T)-\mathbb{E}\bar{X}(T)\big),\bar{x}(T)-\mathbb{E}\bar{x}(T) \big\rangle\\
						&\quad +\big\langle \mathbb{M}_1^\top\boldsymbol{G}^2\mathbb{M}_1\mathbb{E}\bar{X}(T),\mathbb{E}\bar{X}(T)\big\rangle
						+2\big\langle \mathbb{M}_1^\top g^2,\bar{X}(T)-\mathbb{E}\bar{X}(T) \big\rangle+2\big\langle \mathbb{M}_1^\top\boldsymbol{g}^2,\mathbb{E}\bar{X}(T)\big\rangle\\
						&\quad +\int_t^T \Big[ \big\langle \mathbb{M}_1^\top\tilde{Q}_{11}\mathbb{M}_1\big(\bar{X}-\mathbb{E}\bar{X}\big),\bar{X}-\mathbb{E}\bar{X}\big\rangle
						+2\big\langle\mathbb{M}_2^\top\tilde{Q}_{12}\mathbb{M}_1\big(\bar{X}-\mathbb{E}\bar{X}\big),\bar{Y}-\mathbb{E}\bar{Y} \big\rangle\\
						&\quad +2\big\langle \mathbb{M}_2^\top\tilde{Q}_{13}\mathbb{M}_1\big(\bar{X}-\mathbb{E}\bar{X}\big),\bar{Z}-\mathbb{E}\bar{Z} \big\rangle
						+\big\langle \mathbb{M}_2^\top\tilde{Q}_{22}\mathbb{M}_2\big(\bar{Y}-\mathbb{E}\bar{Y}\big),\bar{Y}-\mathbb{E}\bar{Y}\big\rangle\\
						&\quad +2\big\langle \mathbb{M}_2^\top\tilde{Q}_{23}\mathbb{M}_2\big(\bar{Y}-\mathbb{E}\bar{Y}\big),\bar{Z}-\mathbb{E}\bar{Z} \big\rangle
						+\big\langle \mathbb{M}_2^\top\tilde{Q}_{33}\mathbb{M}_2\big(\bar{Z}-\mathbb{E}\bar{Z}\big),\bar{Z}-\mathbb{E}\bar{Z}\big\rangle\\
						&\quad +2\big\langle \tilde{S}_1\mathbb{M}_1\big(\bar{X}-\mathbb{E}\bar{X}\big),\bar{\boldsymbol{\Theta}}\big(\bar{X}-\mathbb{E}\bar{X}\big)+\tilde{\mathbb{V}}_2 \big\rangle
						+2\big\langle \tilde{S}_2\mathbb{M}_2\big(\bar{Y}-\mathbb{E}\bar{Y}\big),\bar{\boldsymbol{\Theta}}\big(\bar{X}-\mathbb{E}\bar{X}\big)+\tilde{\mathbb{V}}_2 \big\rangle\\
					\end{aligned}
				\end{equation*}
				\begin{equation}\label{the leader's value function}
					\begin{aligned}
						&\quad +2\big\langle \tilde{S}_3\mathbb{M}_2\big(\bar{Z}-\mathbb{E}\bar{Z}\big),\bar{\boldsymbol{\Theta}}\big(\bar{X}-\mathbb{E}\bar{X}\big)+\tilde{\mathbb{V}}_2 \big\rangle
						+\big\langle \tilde{R}\big[\bar{\boldsymbol{\Theta}}\big(\bar{X}-\mathbb{E}\bar{X}\big)+\tilde{\mathbb{V}}_2\big],\bar{\boldsymbol{\Theta}}\big(\bar{X}-\mathbb{E}\bar{X}\big)+\tilde{\mathbb{V}}_2 \big\rangle\\
						&\quad +2\big\langle \mathbb{M}_1^\top\big(\tilde{q}_1-\mathbb{E}\tilde{q}_1\big),\bar{X}-\mathbb{E}\bar{X}  \big\rangle
						+2\big\langle \mathbb{M}_2^\top\big(\tilde{q}_2-\mathbb{E}\tilde{q}_2\big),\bar{Y}-\mathbb{E}\bar{Y} \big\rangle
						+2\big\langle \mathbb{M}_2^\top\big(\tilde{q}_3-\mathbb{E}\tilde{q}_3\big),\bar{Z}-\mathbb{E}\bar{Z}\big\rangle\\
						&\quad +2\big\langle \tilde{\rho}-\mathbb{E}\tilde{\rho},\bar{\boldsymbol{\Theta}}\big(\bar{X}-\mathbb{E}\bar{X}\big)+\tilde{\mathbb{V}}_2 \big\rangle
						+\big\langle \mathbb{M}_1^\top\check{Q}_{11}\mathbb{M}_1\mathbb{E}\bar{X},\mathbb{E}\bar{X} \big\rangle
						+2\big\langle \mathbb{M}_2^\top\check{Q}_{12}\mathbb{M}_1\mathbb{E}\bar{X},\mathbb{E}\bar{Y} \big\rangle\\
						&\quad +2\big\langle \mathbb{M}_2^\top\check{Q}_{13}\mathbb{M}_1\mathbb{E}\bar{X},\mathbb{E}\bar{Z} \big\rangle
						+\big\langle \mathbb{M}_2^\top\check{Q}_{22}\mathbb{M}_2\mathbb{E}\bar{Y},\mathbb{E}\bar{Y} \big\rangle
						+2\big\langle \mathbb{M}_2^\top\check{Q}_{23}\mathbb{M}_2\mathbb{E}\bar{Y},\mathbb{E}\bar{Z} \big\rangle\\
						&\quad +\big\langle \mathbb{M}_2^\top\check{Q}_{33}\mathbb{M}_2\mathbb{E}\bar{Z},\mathbb{E}\bar{Z} \big\rangle
						+2\big\langle \check{S}_1\mathbb{M}_1\mathbb{E}\bar{X},\bar{\hat{\boldsymbol{\Theta}}}\mathbb{E}\bar{X}+\check{\mathbb{V}}_2\big\rangle
						+2\big\langle \check{S}_2\mathbb{M}_2\mathbb{E}\bar{Y},\bar{\hat{\boldsymbol{\Theta}}}\mathbb{E}\bar{X}+\check{\mathbb{V}}_2\big\rangle\\
						&\quad +2\big\langle \check{S}_3\mathbb{M}_2\mathbb{E}\bar{Z},\bar{\hat{\boldsymbol{\Theta}}}\mathbb{E}\bar{X}+\check{\mathbb{V}}_2\big\rangle
						+\big\langle \check{R}\big(\bar{\hat{\boldsymbol{\Theta}}}\mathbb{E}\bar{X}+\check{\mathbb{V}}_2\big),\bar{\hat{\boldsymbol{\Theta}}}\mathbb{E}\bar{X}+\check{\mathbb{V}}_2 \big\rangle
						+2\big\langle \mathbb{M}_1^\top\check{q}_1,\mathbb{E}\bar{X} \big\rangle\\
						&\quad +2\big\langle \mathbb{M}_2^\top\check{q}_2,\mathbb{E}\bar{Y} \big\rangle
						+2\big\langle \mathbb{M}_2^\top\check{q}_3,\mathbb{E}\bar{Z}\big\rangle+2\big\langle \check{\rho},\bar{\hat{\boldsymbol{\Theta}}}\mathbb{E}\bar{X}+\check{\mathbb{V}}_2\big\rangle \bigg]ds+\mathcal{L}\bigg\}\\
						&=\mathbb{E}\bigg\{ \big\langle \tilde{\mathbb{H}}\big(\bar{X}(T)-\mathbb{E}\bar{X}(T)\big),\bar{x}(T)-\mathbb{E}\bar{x}(T) \big\rangle +\big\langle \check{\mathbb{H}}\mathbb{E}\bar{X}(T),\mathbb{E}\bar{X}(T)\big\rangle\\
						&\qquad +2\big\langle \tilde{h},\bar{X}(T)-\mathbb{E}\bar{X}(T) \big\rangle+2\big\langle \check{h},\mathbb{E}\bar{X}(T)\big\rangle+\mathcal{L}
						+\int_t^T\Big[\big\langle \tilde{\mathbb{Q}}\big(\bar{X}-\mathbb{E}\bar{X}\big),\bar{X}-\mathbb{E}\bar{X} \big\rangle\\
						&\qquad +\big\langle \check{\mathbb{Q}}\mathbb{E}\bar{X},\mathbb{E}\bar{X} \big\rangle
						+2\big\langle \bar{X}-\mathbb{E}\bar{X},\tilde{\mathbb{S}}\big\rangle+2\big\langle \mathbb{E}\bar{X},\check{\mathbb{S}} \big\rangle+\tilde{\mathbb{L}}+\check{\mathbb{L}}\Big]ds \bigg\}.
					\end{aligned}
				\end{equation}
				
				Applying It\^o's formula to $\big\langle \Gamma_1\big(\bar{X}-\mathbb{E}\bar{X}\big), \bar{X}-\mathbb{E}\bar{X}\big\rangle$ and $\big\langle \alpha_1,\bar{X}-\mathbb{E}\bar{X} \big\rangle$, noting (\ref{Lyapunov equations}), (\ref{BSDE of the leader-1}), (\ref{BSDE of the leader-2}), we have
				\begin{equation*}
					\begin{aligned}
						&\mathbb{E}\big\{\big\langle \tilde{\mathbb{H}}\big(\bar{X}(T)-\mathbb{E}\bar{X}(T)\big), \bar{X}(T)-\mathbb{E}\bar{X}(T)\big\rangle+2\big\langle \tilde{h},\bar{X}(T)-\mathbb{E}\bar{X}(T)\big\rangle\big\}\\
						&=\mathbb{E}\big\{\big\langle \Gamma_1(t)\big(\bar{X}_0-\mathbb{E}\bar{X}_0\big),\bar{X}_0-\mathbb{E}\bar{X}_0 \big\rangle+2\big\langle \alpha_1(t),\bar{X}_0-\mathbb{E}\bar{X}_0\big\rangle\big\}
						+\mathbb{E}\int_t^T\Big[\big\langle \big[\dot{\Gamma}_1+\Gamma_1\tilde{\mathbb{A}}\\
						&\qquad +\tilde{\mathbb{A}}^\top\Gamma_1+\tilde{\mathbb{C}}^\top\Gamma_1\tilde{\mathbb{C}}\big]
						\big(\bar{X}-\mathbb{E}\bar{X}\big),\bar{X}-\mathbb{E}\bar{X} \big\rangle-2\big\langle \tilde{\mathbb{S}},\bar{X}-\mathbb{E}\bar{X} \big\rangle
						+\big\langle \check{\mathbb{C}}^\top\Gamma_1\check{\mathbb{C}}\mathbb{E}\bar{X},\mathbb{E}\bar{X} \big\rangle\\
						&\qquad +2\big\langle \check{\mathbb{C}}^\top\Gamma_1\check{\mathbb{D}}+\check{\mathbb{C}}\beta_1,\mathbb{E}\bar{X} \big\rangle
						+\tilde{\mathbb{D}}^\top\Gamma_1\tilde{\mathbb{D}}+\check{\mathbb{D}}^\top\Gamma_1\check{\mathbb{D}}+2\tilde{\mathbb{B}}^\top\alpha_1+2\big(\tilde{\mathbb{D}}+\check{\mathbb{D}}\big)^\top\beta_1\Big]ds,
					\end{aligned}
				\end{equation*}
				and similarly,
				\begin{equation*}
					\begin{aligned}
						&\mathbb{E}\big\{\big\langle \check{\mathbb{H}}\mathbb{E}\bar{X}(T),\mathbb{E}\bar{X}(T) \big\rangle+2\big\langle \check{h},\mathbb{E}\bar{X}(T) \big\rangle-\big\langle \Gamma_2(t)\mathbb{E}\bar{X}_0,\mathbb{E}\bar{X}_0 \big\rangle-2\big\langle \alpha_2(t),\mathbb{E}\bar{X}_0 \big\rangle\big\}\\
						&\hspace{-4mm}=\mathbb{E}\int_t^T\Big[ \big\langle \big[\dot{\Gamma}_2+\check{\mathbb{A}}^\top\Gamma_2+\Gamma_2\check{\mathbb{A}}\big]\mathbb{E}\bar{X},\mathbb{E}\bar{X} \big\rangle
						-2\big\langle \check{\mathbb{S}}+\check{\mathbb{C}}^\top\Gamma_1\check{\mathbb{D}}+\check{\mathbb{C}}^\top\beta_1,\mathbb{E}\bar{X} \big\rangle+2\check{\mathbb{B}}^\top\alpha_2 \Big]ds.
					\end{aligned}
				\end{equation*}
				Plugging the above two equations into (\ref{the leader's value function}), we have (\ref{value functin of the leader}).
			\end{proof}
			
			Going back to the follower's problem, by (\ref{follower optimal strategy-invertible}), the optimal control $\bar{u}_1(\cdot)$ can also be represented in the following way:
			\begin{equation}\label{folower non-anticipating}
				\begin{aligned}
					\bar{u}_1&=-\Sigma_1^{-1}(B^\top_1P_1+D^\top_1P_1C+S^1_1)(\bar{x}-\mathbb{E}\bar{x})-\hat{\Sigma}_1^{-1}(\boldsymbol{B}_1^\top\Pi_1+\boldsymbol{D}_1^\top P_1\boldsymbol{C}+\boldsymbol{S}^1_1)\mathbb{E}\bar{x}\\
					&\quad-\Sigma_1^{-1}(R^1_{21}+D_1^\top P_1D_2)\big[\bar{\boldsymbol{\Theta}}\big(\bar{X}-\mathbb{E}\bar{X}\big)+\tilde{\mathbb{V}}_2\big]-\hat{\Sigma}_1^{-1}(\boldsymbol{R}^1_{21}+\boldsymbol{D}_1^\top P_1\boldsymbol{D}_2)\big(\bar{\hat{\boldsymbol{\Theta}}}\mathbb{E}\bar{X}+\check{\mathbb{V}}_2\big)\\
					&\quad -\Sigma_1^{-1}\big[B_1^\top(\bar{\eta}_1-\mathbb{E}\bar{\eta}_1)+D^\top_1(\bar{\zeta}_1-\mathbb{E}\bar{\zeta}_1)+D_1^\top P_1(\sigma-\mathbb{E}\sigma)+\rho_1^1-\mathbb{E}\rho_1^1\big]\\
					&\quad -\hat{\Sigma}_1^{-1}(\boldsymbol{B}_1^\top\mathbb{E}\bar{\eta}_1+\boldsymbol{D}_1^\top\mathbb{E}\bar{\zeta}_1+\boldsymbol{D}_1^\top P_1\mathbb{E}\sigma+\mathbb{E}\rho_1^1)\\
					\hspace{-4cm}    &=\bar{\boldsymbol{\Theta}}_1\big(\bar{X}-\mathbb{E}\bar{X}\big)+\bar{\hat{\boldsymbol{\Theta}}}_1\mathbb{E}\bar{X}+\tilde{\mathbb{V}}_1+\check{\mathbb{V}}_1,
				\end{aligned}
			\end{equation}
			where
			\begin{equation}\label{nonanticipating-follower}
				\begin{aligned}
					\bar{\boldsymbol{\Theta}}_1&=\left(\begin{matrix} \bar{\Theta}_1 &0
					\end{matrix}\right)+\left(\begin{matrix} 0 & -\Sigma_1^{-1}B^\top_1
					\end{matrix}\right)P_2+\left(\begin{matrix} 0 & -\Sigma_1^{-1}D^\top_1
					\end{matrix}\right)\big(I-P_2\tilde{\mathcal{K}}\big)^{-1}P_2\big(\tilde{\mathcal{C}}+\tilde{\mathcal{F}}^\top P_2\big)\\
					&\quad +\Big[\left(\begin{matrix} 0 & -\Sigma_1^{-1}D^\top_1
					\end{matrix}\right)\big(I-P_2\tilde{\mathcal{K}}\big)^{-1}P_2\tilde{\mathcal{D}}-\Sigma_1^{-1}(R^1_{21}+D^\top_1P_1D_2)\Big]\bar{\boldsymbol{\Theta}},\\
					\bar{\hat{\boldsymbol{\Theta}}}_1&=\left(\begin{matrix} \bar{\hat{\Theta}}_1 &0
					\end{matrix}\right)+\left(\begin{matrix} 0 & -\hat{\Sigma}_1^{-1}\boldsymbol{B}_1^\top
					\end{matrix}\right)\Pi_2+\left(\begin{matrix} 0 & -\hat{\Sigma}_1^{-1}\boldsymbol{D}_1^\top
					\end{matrix}\right)\big(I-P_2\check{\mathcal{K}}\big)^{-1}P_2\big(\check{\mathcal{C}}+\check{\mathcal{F}}^\top \Pi_2\big)\\
					&\quad +\Big[\left(\begin{matrix} 0 & -\hat{\Sigma}_1^{-1}\boldsymbol{D}_1^\top
					\end{matrix}\right)\big(I-P_2\check{\mathcal{K}}\big)^{-1}P_2\check{\mathcal{D}}
					-\hat{\Sigma}_1^{-1}(\boldsymbol{R}^1_{21}+\boldsymbol{D}_1^\top P_1\boldsymbol{D}_2)\Big]\bar{\hat{\boldsymbol{\Theta}}},\\
					\tilde{\mathbb{V}}_1&=\Big[\left(\begin{matrix} 0 & -\Sigma_1^{-1}D^\top_1
					\end{matrix}\right)\big(I-P_2\tilde{\mathcal{K}}\big)^{-1}P_2\tilde{\mathcal{F}}^\top+\left(\begin{matrix} 0 & -\Sigma_1^{-1}B_1^\top
					\end{matrix}\right)\Big](\bar{\eta}_2-\mathbb{E}\bar{\eta}_2)\\
					&\quad +\left(\begin{matrix} 0 & -\Sigma_1^{-1}D^\top_1
					\end{matrix}\right)\big(I-P_2\tilde{\mathcal{K}}\big)^{-1}(\bar{\zeta}_2-\mathbb{E}\bar{\zeta}_2)+\left(\begin{matrix} 0 & -\Sigma_1^{-1}D^\top_1
					\end{matrix}\right)\big(I-P_2\tilde{\mathcal{K}}\big)^{-1}P_2\tilde{\boldsymbol{\sigma}}\\
					&\quad +\Big[\left(\begin{matrix} 0 & -\Sigma_1^{-1}D^\top_1
					\end{matrix}\right)\big(I-P_2\tilde{\mathcal{K}}\big)^{-1}P_2\tilde{\mathcal{D}}-\Sigma_1^{-1}(R^1_{21}+D^\top_1P_1D_2)\Big]\tilde{\mathbb{V}}_2\\
					&\quad-\Sigma_1^{-1}\big[D^\top_1P_1(\sigma-\mathbb{E}\sigma)+\rho_1^1-\mathbb{E}\rho_1^1\big],\\
					\check{\mathbb{V}}_1&=\Big[\left(\begin{matrix} 0 & -\hat{\Sigma}_1^{-1}\boldsymbol{D}_1^\top
					\end{matrix}\right)\big(I-P_2\check{\mathcal{K}}\big)^{-1}P_2\check{\mathcal{F}}^\top+\left(\begin{matrix} 0 & -\hat{\Sigma}_1^{-1}\boldsymbol{B}_1^\top
					\end{matrix}\right)\Big]\mathbb{E}\bar{\eta}_2\\
					&\quad +\left(\begin{matrix} 0 & -\hat{\Sigma}_1^{-1}\boldsymbol{D}_1^\top
					\end{matrix}\right)\big(I-P_2\check{\mathcal{K}}\big)^{-1}\mathbb{E}\bar{\zeta}_2+\left(\begin{matrix} 0 & -\hat{\Sigma}_1^{-1}\boldsymbol{D}_1^\top
					\end{matrix}\right)\big(I-P_2\check{\mathcal{K}}\big)^{-1}P_2\check{\boldsymbol{\sigma}}\\
					&\quad +\Big[\left(\begin{matrix} 0 & -\hat{\Sigma}_1^{-1}\boldsymbol{D}_1^\top
					\end{matrix}\right)\big(I-P_2\check{\mathcal{K}}\big)^{-1}P_2\check{\mathcal{D}}-\hat{\Sigma}_1^{-1}(\boldsymbol{R}^1_{21} +\boldsymbol{D}_1^\top P_1\boldsymbol{D}_2)\Big]\check{\mathbb{V}}_2\\
					&\quad -\hat{\Sigma}_1^{-1}(\boldsymbol{D}_1^\top P_1\mathbb{E}\sigma+\mathbb{E}\rho_1^1).
				\end{aligned}
			\end{equation}
			
			Now, the 6-tuple $\big(\bar{\boldsymbol{\Theta}}_1(\cdot),\bar{\hat{\boldsymbol{\Theta}}}_1(\cdot),\tilde{\mathbb{V}}_1(\cdot)+\check{\mathbb{V}}_1(\cdot),\bar{\boldsymbol{\Theta}}(\cdot),
			\bar{\hat{\boldsymbol{\Theta}}}(\cdot),\tilde{\mathbb{V}}_2(\cdot)+\check{\mathbb{V}}_2(\cdot)\big)$ is our non-anticipating closed-loop Stackelberg equilibrium of the mean-field type LQ Stackelberg stochastic differential game on $[t,T]$.
			
			\begin{Remark}
				Riccati equations (\ref{Riccati---1}), (\ref{Riccati---2}), (\ref{P2}) and (\ref{Pi2}) are the same as (19), (20), (45) and (46) in \cite{LJZ2019}, respectively. This means that if both the open-loop Stackelberg equilibrium and the closed-loop Stackelberg equilibrium exist, the feedback representation of the open-loop Stackelberg equilibrium is consistent with the outcome of the closed-loop Stackelberg equilibrium.
			\end{Remark}
			
			Before the ending of this section, let us consider the solvability of Riccati equations (\ref{P2}) and (\ref{Pi2}) of $(P_2(\cdot), \Pi_2(\cdot))$, in the case
			\begin{equation}\label{special case of Riccati}
				D_1=\hat{D}_1=0,\,\,D_2=-\hat{D}_2.
			\end{equation}
			The general case is very challenging and we will consider it in the future. In the present case,  (\ref{P2}) and (\ref{Pi2}) take the following form:
			\begin{equation}\label{P2_D1=0}
				\left\{\begin{aligned}
					&0=\dot{P}_2+\tilde{\mathcal{A}}^\top P_2+P_2\tilde{\mathcal{A}}+\tilde{\mathcal{H}}+P_2\tilde{\mathcal{M}}P_2+\tilde{\mathcal{C}}^\top P_2\tilde{\mathcal{C}}\\
					&\qquad-(\tilde{\mathcal{N}}+P_2\tilde{\mathcal{B}}+\tilde{\mathcal{C}}^\top P_2\tilde{\mathcal{D}})(\tilde{R}+\tilde{\mathcal{D}}^\top P_2\tilde{\mathcal{D}})^{-1}(\tilde{\mathcal{N}}^\top+\tilde{\mathcal{B}}^\top P_2
					+\tilde{\mathcal{D}}^\top P_2\tilde{\mathcal{C}}),\quad s\in[t,T],\\
					&P_2(T)=\tilde{\mathcal{G}},\quad \tilde{R}+\tilde{\mathcal{D}}^\top P_2\tilde{\mathcal{D}}>0,
				\end{aligned}\right.
			\end{equation}
			and
			\begin{equation}\label{Pi2_D1=0}
				\left\{\begin{aligned}
					&0=\dot{\Pi}_2+\check{\mathcal{A}}^\top \Pi_2+\Pi_2\check{\mathcal{A}}+\check{\mathcal{H}}+\Pi_2\check{\mathcal{M}}\Pi_2+\check{\mathcal{C}}^\top P_2\check{\mathcal{C}}\\
					&\qquad -(\check{\mathcal{N}}+\Pi_2\check{\mathcal{B}})\check{R}^{-1}(\check{\mathcal{N}}^\top+\check{\mathcal{B}}^\top \Pi_2),\quad s\in[t,T],\\
					&\Pi_2(T)=\tilde{\mathcal{G}}+\check{\mathcal{G}},\quad  \check{R}>0.
				\end{aligned}\right.
			\end{equation}
			
			We introduce the following assumptions.
			\begin{itemize}
				\item [\textbf{(H3)}] Coefficients of (\ref{state}) and weighting coefficients of (\ref{cf}) satisfy
				\begin{equation*}
					\begin{cases}
						A, \hat{A}, C, \hat{C} \in L^{\infty}(0,T;\mathbb{R}^{n \times n}),\quad B_i, \hat{B}_i, D_i, \hat{D}_i \in L^{\infty}(0,T;\mathbb{R}^{n \times m_i}),\\
						Q^i, \hat{Q}^i \in L^{\infty}(0,T;\mathbb{S}^n),\quad S^i_j, \hat{S}^i_j \in L^{\infty}(0,T;\mathbb{R}^{m_j \times n}),\quad R^i_{jj}, \hat{R}^i_{jj} \in L^{\infty}(0,T;\mathbb{S}^{m_j}),\\
						R^i_{12}, \hat{R}^i_{12} \in L^{\infty}(0,T;\mathbb{R}^{m_2 \times m_1}),\quad G^i, \hat{G}^i \in \mathbb{S}^n, \quad i, j=1, 2.\\
					\end{cases}
				\end{equation*}
				\item [\textbf{(H4)}] Let $\tilde{\mathcal{M}}, \check{\mathcal{M}}$ be invertible, and
				\begin{equation*}
					\tilde{R},\,\check{R}>0, \quad \tilde{\mathcal{H}}-\tilde{\mathcal{N}}\tilde{R}^{-1}\tilde{\mathcal{N}}^\top\geq0,\quad\check{\mathcal{H}}-\check{\mathcal{N}}\check{R}^{-1}\check{\mathcal{N}}^\top\geq0,\quad
					\tilde{\mathcal{M}},\,\check{\mathcal{M}}<0, \quad\tilde{\mathcal{G}},\,\check{\mathcal{G}} \geq0.
				\end{equation*}
			\end{itemize}
			
			\begin{Remark}
				It is worth noting that Riccati equations $P_2(\cdot)$ and $\Pi_2(\cdot)$ are no longer standard ones like $(6.6)$ of Chapter 6 in Yong and Zhou \cite{YongZhou1999} due to additional terms $P_2\tilde{\mathcal{M}}P_2$ and $\Pi_2\check{\mathcal{M}}\Pi_2$. In order to deal with these two terms, we need to make some assumptions to $\tilde{\mathcal{M}}$ and $ \check{\mathcal{M}}$. On the other hand, assumption (H3) means that coefficients of them are essentially bounded.
			\end{Remark}
			
			\begin{mythm}\label{Riccati equation}
				Let (H3)-(H4) hold. Then Riccati equations (\ref{P2_D1=0}) and (\ref{Pi2_D1=0}) admit a unique solution $P_2(\cdot) \in C([t,T],\mathbb{S}^{2n}_+)$ and $\Pi_2(\cdot) \in C([t,T],\mathbb{S}^{2n}_+)$, respectively.
			\end{mythm}
			
			\begin{proof}
				Firstly, we prove (\ref{P2_D1=0}) admits at most one solution $P_2(\cdot) \in C([t,T],\mathbb{S}^{2n}_+)$. Suppose that $P_2(\cdot)$ and $\tilde{P}_2(\cdot)$ are two solutions of (\ref{P2_D1=0}). Setting $\hat{P}=P_2-\tilde{P}_2$, then $\hat{P}(\cdot)$ satisfies
				\begin{equation*}
					\left\{\begin{aligned}
						&0=\dot{\hat{P}}+\tilde{\mathcal{A}}^\top\hat{P}+\hat{P}\tilde{\mathcal{A}}+\hat{P}\tilde{\mathcal{M}}P_2+\tilde{P}_2\tilde{\mathcal{M}}\hat{P}+\tilde{\mathcal{C}}^\top\hat{P}\tilde{\mathcal{C}}\\
						&\quad -\big(\hat{P}\tilde{\mathcal{B}}+\tilde{\mathcal{C}}^\top\hat{P}\tilde{\mathcal{D}}\big)\Sigma^{-1}\big(\tilde{\mathcal{N}}^\top+\tilde{\mathcal{B}}^\top P_2+\tilde{\mathcal{D}}^\top P_2\tilde{\mathcal{C}}\big)\\
						&\quad -\big(\tilde{\mathcal{N}}+\tilde{P}_2\tilde{\mathcal{B}}+\tilde{\mathcal{C}}^\top \tilde{P}_2\tilde{\mathcal{D}}\big)
						\Sigma^{-1}\tilde{\mathcal{D}}^\top \hat{P}\tilde{\mathcal{D}}\hat{\Sigma}^{-1}\big(\tilde{\mathcal{N}}^\top+\tilde{\mathcal{B}}^\top P_2+\tilde{\mathcal{D}}^\top P_2\tilde{\mathcal{C}}\big)\\
						&\quad - \big(\tilde{\mathcal{N}}+\tilde{P}_2\tilde{\mathcal{B}}+\tilde{\mathcal{C}}^\top \tilde{P}_2\tilde{\mathcal{D}}\big)\hat{\Sigma}^{-1}
						\big(\tilde{\mathcal{B}}^\top\hat{P}+\tilde{\mathcal{D}}^\top\hat{P}\tilde{\mathcal{C}}\big),\quad s\in[t,T],\\
						&\hat{P}(T)=0,
					\end{aligned}\right.
				\end{equation*}
				where $\Sigma:=\tilde{R}+\tilde{\mathcal{D}}^\top P_2\tilde{\mathcal{D}} >0$ and $\hat{\Sigma}:=\tilde{R}+\tilde{\mathcal{D}}^\top \tilde{P}_2\tilde{\mathcal{D}}>0$. Since $|\Sigma(\cdot)^{-1}|$ and $|\hat{\Sigma}(\cdot)^{-1}|$ are uniformly bounded due to their continuity, we can apply Gronwall's inequality to get $\hat{P}(\cdot)\equiv 0$. This proves the uniqueness of the solution to (\ref{P2_D1=0}).
				
				Next, let us focus on the existence of solution to (\ref{P2_D1=0}). Inspired by \cite{YongZhou1999}, we set
				\begin{equation*}
					\left\{\begin{aligned}
						\Phi_1&:=\big(\tilde{R}+\tilde{\mathcal{D}}^\top P_2\tilde{\mathcal{D}}\big)^{-1}\big(\tilde{\mathcal{N}}^\top+\tilde{\mathcal{B}}^\top P_2
						+\tilde{\mathcal{D}}^\top P_2\tilde{\mathcal{C}}\big),\quad \Phi_2:=\tilde{\mathcal{M}}P_2,\\
						\hat{\mathcal{A}}&:=\tilde{\mathcal{A}}-\tilde{\mathcal{B}}\Phi_1+\Phi_2,\quad \hat{\mathcal{C}}:=\tilde{\mathcal{C}}-\tilde{\mathcal{D}}\Phi_1,\\
						\hat{\mathcal{Q}}&:=\big(\Phi_1-\tilde{R}^{-1}\tilde{\mathcal{N}}^\top\big)^\top\tilde{R}\big(\Phi_1-\tilde{R}^{-1}\tilde{\mathcal{N}}^\top\big)
						-\Phi^\top_2\tilde{\mathcal{M}}^{-1}\Phi_2+\tilde{\mathcal{H}}-\tilde{\mathcal{N}}\tilde{R}^{-1}\tilde{\mathcal{N}}^\top.
					\end{aligned}\right.
				\end{equation*}
				It is easy to verify that (\ref{P2_D1=0}) is equivalent to
				\begin{equation}\label{P2-simple}
					\left\{\begin{aligned}
						&0=\dot{P}_2+\hat{\mathcal{A}}^\top P_2+P_2\hat{\mathcal{A}}+\hat{\mathcal{C}}^\top P_2\hat{\mathcal{C}}+\hat{\mathcal{Q}},\quad s\in[t,T],\\
						&P_2(T)=\tilde{\mathcal{G}}.
					\end{aligned}\right.
				\end{equation}
				Then, we construct the followinng iterative scheme. For $i=0, 1, 2,\cdots$, setting
				\begin{equation}\label{iterative}
					\left\{\begin{aligned}
						\Sigma_0&=\tilde{\mathcal{G}},\\
						\Phi_1^i&=	\big(\tilde{R}+\tilde{\mathcal{D}}^\top \Sigma_i\tilde{\mathcal{D}}\big)^{-1}\big(\tilde{\mathcal{N}}^\top+\tilde{\mathcal{B}}^\top \Sigma_i
						+\tilde{\mathcal{D}}^\top \Sigma_i\tilde{\mathcal{C}}\big),\quad \Phi_2^i=\tilde{\mathcal{M}}\Sigma_i,\\
						\hat{\mathcal{A}}_i&=\tilde{\mathcal{A}}-\tilde{\mathcal{B}}\Phi^i_1+\Phi^i_2,\quad \hat{\mathcal{C}}_i=\tilde{\mathcal{C}}-\tilde{\mathcal{D}}\Phi_1^i,\\
						\hat{\mathcal{Q}}_i&=\big(\Phi^i_1-\tilde{R}^{-1}\tilde{\mathcal{N}}^\top\big)^\top\tilde{R}\big(\Phi^i_1-\tilde{R}^{-1}\tilde{\mathcal{N}}^\top\big)
						-{\Phi^i_2}^\top\tilde{\mathcal{M}}^{-1}\Phi_2^i+\tilde{\mathcal{H}}-\tilde{\mathcal{N}}\tilde{R}^{-1}\tilde{\mathcal{N}}^\top,
					\end{aligned}\right.
				\end{equation}
				and let $\Sigma_{i+1}(\cdot)$ be the solution to
				\begin{equation}\label{Sigma_i}
					\left\{\begin{aligned}
						&0=\dot{\Sigma}_{i+1}+\hat{\mathcal{A}}_i^\top \Sigma_{i+1}+\Sigma_{i+1}\hat{\mathcal{A}}_i+{\hat{\mathcal{C}}_i}^\top \Sigma_{i+1}\hat{\mathcal{C}}_i+\hat{\mathcal{Q}}_i,\quad s\in[t,T],\\
						&\Sigma_{i+1}(T)=\tilde{\mathcal{G}}.
					\end{aligned}\right.
				\end{equation}
				By using Lemma 7.3 of Chapter 6 in \cite{YongZhou1999} and noting (H3) and (H4), we see that for $i \geq 0 $, $\Sigma_i(\cdot)$ is well defined and moreover $\Sigma_i(\cdot) \geq 0$. We claim that $\{\Sigma_i(\cdot)\}$, for $i \geq 0 $, is a decreasing sequence in $C([t,T],\mathbb{S}^{2n}_+)$. To show this, we define $\Delta_i:=\Sigma_i-\Sigma_{i+1}$, $\Lambda_i:=\Phi^{i-1}_1-\Phi_1^i$ and $\Upsilon_i:=\Phi^{i-1}_2-\Phi_2^i$. Subtracting the $(i+1)$th equation from the $i$th equation for (\ref{Sigma_i}), we get
				\begin{equation*}
					\begin{aligned}
						-\dot{\Delta}_i&=\Delta_i\hat{\mathcal{A}}_i+\hat{\mathcal{A}}_i^\top\Delta_i+\hat{\mathcal{C}}_i^\top\Delta_i\hat{\mathcal{C}}_i+\Sigma_i\big(\hat{\mathcal{A}}_{i-1}
						-\hat{\mathcal{A}}_i\big)+\big(\hat{\mathcal{A}}_{i-1}-\hat{\mathcal{A}}_i\big)^\top\Sigma_i\\
						&\qquad +\hat{\mathcal{C}}_{i-1}^\top\Sigma_i\hat{\mathcal{C}}_{i-1}-\hat{\mathcal{C}}_i^\top\Sigma_i\hat{\mathcal{C}}_i+\hat{\mathcal{Q}}_{i-1}-\hat{\mathcal{Q}}_i.
					\end{aligned}
				\end{equation*}
				According to (\ref{iterative}), we have
				\begin{equation*}
					\begin{aligned}
						&\hat{\mathcal{A}}_{i-1}-\hat{\mathcal{A}}_i=-\tilde{\mathcal{B}}\Lambda_i+\Upsilon_i, \quad 	\hat{\mathcal{C}}_{i-1}-\hat{\mathcal{C}}_i=-\tilde{\mathcal{D}}\Lambda_i,\\
						&\hat{\mathcal{C}}_{i-1}^\top\Sigma_i\hat{\mathcal{C}}_{i-1}-\hat{\mathcal{C}}_i^\top\Sigma_i\hat{\mathcal{C}}_i=\Lambda_i^\top\tilde{\mathcal{D}}^\top\Sigma_i\tilde{\mathcal{D}}\Lambda_i
						-\hat{\mathcal{C}}^\top_i\Sigma_i\tilde{\mathcal{D}}\Lambda_i-\Lambda_i^\top\tilde{\mathcal{D}}^\top\Sigma_i\hat{\mathcal{C}}_i,\\
						&\hat{\mathcal{Q}}_{i-1}-\hat{\mathcal{Q}}_i=\Lambda_i^\top\tilde{R}\Lambda_i-\Upsilon_i^\top\tilde{\mathcal{M}}^{-1}\Upsilon_i+\Lambda_i^\top\big(\tilde{R}\Phi_1^i-\tilde{\mathcal{N}}^\top\big)\\
						&\qquad\qquad\qquad +\big(\tilde{R}\Phi_1^i-\tilde{\mathcal{N}}^\top\big)^\top\Lambda_i-\Upsilon_i^\top\tilde{\mathcal{M}}^{-1}\Phi_2^i-{\Phi_2^i}^\top\tilde{\mathcal{M}}^{-1}\Upsilon_i.
					\end{aligned}
				\end{equation*}
				Thus, we obtain
				\begin{equation*}
					\begin{aligned}
						&\quad -\big(\dot{\Delta}_i+\Delta_i\hat{\mathcal{A}}_i+\hat{\mathcal{A}}_i^\top\Delta_i+\hat{\mathcal{C}}_i^\top\Delta_i\hat{\mathcal{C}}_i \big)\\
						&=\Lambda_i^\top\big(\tilde{R}+\tilde{\mathcal{D}}^\top\Sigma_i\tilde{\mathcal{D}}\big)\Lambda_i-\Upsilon_i^\top\tilde{\mathcal{M}}^{-1}\Upsilon_i+\Lambda_i^\top\big(\tilde{R}\Phi_1^i
						-\tilde{\mathcal{N}}^\top-\tilde{\mathcal{B}}^\top\Sigma_i-\tilde{\mathcal{D}}^\top\Sigma_i\hat{\mathcal{C}}_i\big)\\
						&\quad +\big(\tilde{R}\Phi_1^i-\tilde{\mathcal{N}}^\top-\tilde{\mathcal{B}}^\top\Sigma_i-\tilde{\mathcal{D}}^\top\Sigma_i\hat{\mathcal{C}}_i\big)^\top\Lambda_i
						+\Upsilon_i^\top\big(\Sigma_i-\tilde{\mathcal{M}}^{-1}\Phi_2^i\big)+\big(\Sigma_i-\tilde{\mathcal{M}}^{-1}\Phi_2^i\big)^\top\Upsilon_i\\
						&=\Lambda_i^\top\big(\tilde{R}+\tilde{\mathcal{D}}^\top\Sigma_i\tilde{\mathcal{D}}\big)\Lambda_i-\Upsilon_i^\top\tilde{\mathcal{M}}^{-1}\Upsilon_i \geq 0.
					\end{aligned}
				\end{equation*}
				Noting that $\Delta_i(T)=0$ and Lemma 7.3 of Chapter 6 in \cite{YongZhou1999}, we have $\Sigma_i(s) \geq \Sigma_{i+1}(s),\,\, s \in [t,T]$. Thus $\{\Sigma_i(\cdot)\}$ is a decreasing sequence in $C([t,T],\mathbb{S}^{2n}_+)$, and thus has a limit denoted by $P_2(\cdot)$. Using the same method as Lemma 3.6 in Li et al. \cite{LiNieWu2023}, we can prove that $\{\Sigma_i(\cdot)\}$ is a Cauchy sequence in $C([t,T],\mathbb{S}^{2n}_+)$ and $\{\dot{\Sigma}_i(\cdot)\}$ is a Cauchy sequence in $C([t,T],\mathbb{S}^{2n})$. Clearly, $P_2(\cdot)$ is the solution to (\ref{P2-simple}), hence (\ref{P2_D1=0}).
				
				Finally, after obtaining the solvability of (\ref{P2_D1=0}) of $P_2(\cdot)$, we rewrite (\ref{Pi2_D1=0}) of $\Pi_2(\cdot)$:
				\begin{equation}\label{Pi2-2}
					\left\{\begin{aligned}
						&0=\dot{\Pi}_2+\big(\check{\mathcal{A}}-\check{\mathcal{B}}\check{R}^{-1}\check{\mathcal{N}}^\top\big)^\top \Pi_2+\Pi_2\big(\check{\mathcal{A}}-\check{\mathcal{B}}\check{R}^{-1}\check{\mathcal{N}}^\top\big)\\
						&\qquad +\Pi_2\big(\check{\mathcal{M}}-\check{\mathcal{B}}\check{R}^{-1}\check{\mathcal{B}}^\top\big)\Pi_2
						+\check{\mathcal{C}}^\top P_2\check{\mathcal{C}}+\check{\mathcal{H}}-\check{\mathcal{N}}\check{R}^{-1}\check{\mathcal{N}}^\top,\quad s\in[t,T],\\
						&\Pi_2(T)=\tilde{\mathcal{G}}+\check{\mathcal{G}},\quad  \check{R}>0.
					\end{aligned}\right.
				\end{equation}
				Noting that (H4), we can obtain $\check{\mathcal{C}}^\top P_2\check{\mathcal{C}}+\check{\mathcal{H}}-\check{\mathcal{N}}\check{R}^{-1}\check{\mathcal{N}}^\top \geq 0$, which means (\ref{Pi2_D1=0}) is uniquely solvable.
			\end{proof}
			
			\section{Concluding Remarks}
			
			In this paper, we have investigated open-loop and closed-loop Stackelberg equilibria for the mean-field type LQ Stackelberg stochastic differential game. We conclude that the existence of open-loop Stackelberg equilibrium is equivalent to the solvability of some MF-FBSDEs system as well as some convexity conditions (Lemma \ref{follower open-loop ns condition} and Theorem \ref{leader open-loop ns condition}). We describe the closed-loop solvability of follower's problem, that is, the existence of closed-loop optimal strategies is equivalent to the solvability of two Riccati equations and a MF-BSDE (Theorem \ref{Th-cl-f}). Necessary conditions for the existence of closed-loop optimal strategies are given (Theorem \ref{Th-cl-l}). The state equation of the leader is a MF-FBSDE, which leads to the failure of the completion-of-square technique and is impossible to obtain sufficient conditions for the existence of closed-loop optimal strategies of the leader. This is a gap that we hope to fill in the future. But fortunately, we get the expression of the value function of the leader's problem through two coupled Lyapunov equations and two coupled BSDEs (Theorem \ref{leader value function}). In a special case, the existence and uniqueness of solutions of Riccati equations for leader's optimization problem is discussed (Theorem \ref{Riccati equation}).
			

		\end{document}